\documentclass[12pt,a4paper]{amsart}
\usepackage{amssymb}
\usepackage[all,cmtip]{xy}

\calclayout

\newcommand{\+}{\protect\nobreakdash-}
\renewcommand{\:}{\colon}

\newcommand{\rarrow}{\longrightarrow}

\newcommand{\ot}{\otimes}

\newcommand{\lrarrow}{\mskip.5\thinmuskip\relbar\joinrel\relbar\joinrel
 \rightarrow\mskip.5\thinmuskip\relax}

\newcommand{\bu}{{\text{\smaller\smaller$\scriptstyle\bullet$}}}

\DeclareMathOperator{\Hom}{Hom}
\DeclareMathOperator{\Ext}{Ext}
\DeclareMathOperator{\Tor}{Tor}
\DeclareMathOperator{\Spec}{Spec}

\newcommand{\Modl}{{\operatorname{\mathsf{--Mod}}}}

\newcommand{\Fil}{\mathsf{Fil}}
\newcommand{\inj}{\mathsf{inj}}

\newcommand{\fl}{\mathsf{fl}}
\newcommand{\vfl}{\mathsf{vfl}}
\renewcommand{\cot}{\mathsf{cot}}
\newcommand{\cta}{\mathsf{cta}}

\newcommand{\ac}{\mathbf{ac}}
\newcommand{\coac}{\mathbf{coac}}
\newcommand{\ctrac}{\mathbf{ctrac}}
\newcommand{\hin}{\mathbf{hin}}
\newcommand{\hfl}{\mathbf{hfl}}
\newcommand{\bfl}{\mathbf{fl}}
\newcommand{\bvfl}{\mathbf{vfl}}

\newcommand{\R}{\mathcal R}

\newcommand{\bC}{\mathbf C}
\newcommand{\bH}{\mathbf H}

\newcommand{\sA}{\mathsf A}
\newcommand{\sB}{\mathsf B} 
\newcommand{\sC}{\mathsf C}
\newcommand{\sD}{\mathsf D}
\newcommand{\sE}{\mathsf E}
\newcommand{\sF}{\mathsf F}
\newcommand{\sK}{\mathsf K}
\newcommand{\sL}{\mathsf L}
\newcommand{\sS}{\mathsf S}
\newcommand{\sT}{\mathsf T}

\newcommand{\boZ}{\mathbb Z}

\newcommand{\Section}[1]{\bigskip\section{#1}\medskip}
\theoremstyle{plain}
\newtheorem{thm}{Theorem}[section]
\newtheorem{lem}[thm]{Lemma}
\newtheorem{prop}[thm]{Proposition}
\newtheorem{cor}[thm]{Corollary}
\theoremstyle{definition}
\newtheorem{ex}[thm]{Example}

\newtheorem{rem}[thm]{Remark}

\begin{document}

\title{Local, colocal, and antilocal properties of \\ modules
and complexes over commutative rings}

\author{Leonid Positselski}

\address{Institute of Mathematics, Czech Academy of Sciences \\
\v Zitn\'a~25, 115~67 Praha~1 \\ Czech Republic}

\email{positselski@math.cas.cz}

\begin{abstract}
 This paper is a commutative algebra introduction to the homological
theory of quasi-coherent sheaves and contraherent cosheaves over
quasi-compact semi-separated schemes.
 Antilocality is an alternative way in which global properties are
locally controlled in a finite affine open covering.
 For example, injectivity of modules over non-Noetherian commutative
rings is not preserved by localizations, while homotopy injectivity
of complexes of modules is not preserved by localizations even for
Noetherian rings.
 The latter also applies to the contraadjustedness and cotorsion
properties.
 All the mentioned properties of modules or complexes over commutative
rings are actually antilocal.
 They are also colocal, if one presumes contraadjustedness.
 Generally, if the left class in a (hereditary complete) cotorsion
theory for modules or complexes of modules over commutative rings
is local and preserved by direct images with respect to affine
open immersions, then the right class is antilocal.
 If the right class in a cotorsion theory for contraadjusted modules
or complexes of contraadjusted modules is colocal and preserved by
such direct images, then the left class is antilocal.
 As further examples, the class of flat contraadjusted modules
is antilocal, and so are the classes of acyclic, Becker-coacyclic,
or Becker-contraacyclic complexes of contraadjusted modules.
 The same applies to the classes of homotopy flat complexes of flat
contraadjusted modules and acyclic complexes of flat contraadjusted
modules with flat modules of cocycles.
\end{abstract}

\maketitle

\tableofcontents

\section*{Introduction}
\medskip

\subsection{{}}
 Relations between local and global properties are a fundamental
aspect of geometry.
 In the context of coherent and quasi-coherent sheaves over schemes,
the passage from the global to the local is expressed by
the localization functors in the language of commutative algebra.
 For Zariski open coverings, the localization basically means
inverting an element~$s$ in a commutative ring~$R$.
 So to an $R$\+module $M$ one assigns the $R[s^{-1}]$\+module
$M[s^{-1}]=R[s^{-1}]\ot_RM$.

 For example, flatness is a local property of modules over commutative
rings, as one can easily see.
 The \emph{very flatness}~\cite{Pcosh,ST,PSl} is also a local property
with respect to Zariski open coverings.
 A difficult theorem of Raynaud and Gruson~\cite[\S\,II.3.1]{RG},
\cite{Pe} tells that projectivity of modules is a local property.
 Acyclicity and \emph{coacyclicity}~\cite{EP,Bec} of a complex of
modules are Zariski local properties.

 On the other hand, injectivity of modules is a local property over
Noetherian rings~\cite[Section~II.7]{Hart}, but \emph{not} in general.
 The homotopy injectivity of complexes of (injective or arbitrary)
modules is \emph{not} a local property even over Noetherian
rings~\cite[Example~6.5]{Neem-bb}.
 So, obviously, not all the interesting properties of modules or
complexes of modules over commutative rings are local.

\subsection{{}}
 It can be argued that one is not supposed to localize injective
modules; rather, one should \emph{colocalize} them instead.
 Given a commutative ring $R$, an element $s\in R$, and
an $R$\+module $M$, the \emph{colocalization} of $M$ is
the $R[s^{-1}]$\+module $\Hom_R(R[s^{-1}],M)$.
 At least, the colocalization (unlike the localization) does preserve
injectivity of modules and homotopy injectivity of complexes quite
generally.

 Still, if one is interested in nonaffine schemes, then how exactly
one is supposed to use the colocalization?
 The restriction of a quasi-coherent sheaf to an open subscheme is
described algebraically in terms of the localization functors (as
we already mentioned above).
 The colocalization may make sense as an algebraic procedure, but what
does it mean geometrically?

 The answer is that, alongside with the quasi-coherent sheaves,
there is a different kind of global module gadgets over schemes,
called the \emph{contraherent cosheaves}~\cite{Pcosh}.
 In contraherent cosheaves, the restriction to an open subscheme
is expressed algebraically by the colocalization functors.
 So, one possible answer to the questions above is that, if one
wants to work globally with colocal properties, then one should
consider contraherent cosheaves rather than quasi-coherent sheaves.

 But what if one is interested specifically in injective quasi-coherent
sheaves over non-Noetherian schemes, or in homotopy injective complexes
of quasi-coherent sheaves?
 Is there any way to control such global properties locally?
 In this paper we offer an answer to this question.
 In addition to being well-behaved with respect to the colocalizations,
such properties as injectivity of modules or homotopy injectivity of
complexes also enjoy a completely different kind local-global principle
with respect to Zariski open coverings.
 We say that these properties are \emph{antilocal}.

\subsection{{}} \label{introd-cotorsion-pairs-formulations}
 A great variety of classes of modules and complexes arise in
connection with \emph{cotorsion pairs}, and it is such classes that
we consider in this paper.
 The proofs of our main results are based on a self-dual explicit
elementary construction of complete cotorsion pairs, avoiding the use
of the small object argument.
 In this respect, the present paper is a commutative algebra version
of~\cite{Pctrl}.
 While the constructions of complete cotorsion pairs in~\cite{Pctrl}
were inspired by the ones in~\cite{Psemi}, the constructions of
cotorsion pairs in the present paper are taken
from~\cite[Chapter~4]{Pcosh}.

 Let us describe our setting and the main results in some more detail.
 We consider a class of commutative rings $\R$ closed under
passages to localizations with respect to elements; so $R\in\R$ and
$s\in R$ implies $R[s^{-1}]\in\R$.
 For example, one can take $\R$ to be the class of all commutative rings
or the class of Noetherian commutative rings.
 Then we consider systems of classes of modules $(\sE_R)_{R\in\R}$,
where $\sE_R\subset R\Modl$ is a class of $R$\+modules given for
every ring $R\in\R$.
 We also consider similar systems of classes of complexes
$(\sE_R)_{R\in\R}$, where $\sE_R\subset\bC(R\Modl)$ is a class of
complexes of $R$\+modules given for every $R\in\R$.

 To increase generality, we work with cotorsion pairs
in \emph{exact subcategories} of the abelian categories of modules
or complexes.
 So we suppose given, for every ring $R\in\R$, a full subcategory of
modules $\sE_R\subset R\Modl$ or a full subcategory of complexes
$\sE_R\subset\bC(R\Modl)$ closed under extensions and direct summands
in the abelian exact structure of $R\Modl$ or $\bC(R\Modl)$.
 We assume that, as the ring $R\in\R$ varies, the property of
a module or complex to belong to $\sE_R$ is local with respect to
Zariski open coverings, and moreover, that it is \emph{very local}.
 The latter condition means, in addition to the locality, that
the property is preserved by the restrictions of scalars (direct images)
with respect to the localization morphisms $R\rarrow R[s^{-1}]$.

 Then we consider a system of \emph{hereditary complete} cotorsion
pairs $(\sA_R,\sB_R)$ in $\sE_R$ given for every ring $R\in\R$.
 In this context, we prove that if the system of classes $\sA$ is
very local, then the system of classes $\sB$ is antilocal.
 If the class $\sB$ is antilocal, then the class $\sA$ is local.
 Dually, assume that the system of classes $\sE$ is \emph{very colocal}.
 Under this assumption, if the system of classes $\sB$ is very colocal,
then the system of classes $\sA$ is antilocal.
 If the class $\sA$ is antilocal, then the class $\sB$ is colocal.

\subsection{{}}
 There is, however, one \emph{caveat} that one should always keep in
mind when speaking of the colocality of classes or colocal properties.
 The localization functors $M\longmapsto M[s^{-1}]=R[s^{-1}]\ot_RM$ are
exact.
 In other words, the $R$\+module $R[s^{-1}]$ is flat (in fact, it is
\emph{very flat} in the sense of~\cite{Pcosh,ST,PSl}).
 However, the colocalization functor $M\longmapsto\Hom_R(R[s^{-1}],M)$
is \emph{not} exact on the category of arbitrary $R$\+modules.
 In other words, the $R$\+module $R[s^{-1}]$ is usually \emph{not}
projective.

 Furthermore, the colocalization functors are not ``jointly faithful''
for a principal affine open covering.
 So, if $s_1$,~\dots, $s_d\in R$ is a collection of elements generating
the unit ideal in $R$, then $M[s_j^{-1}]=0$ for all $1\le j\le d$
for some $M\in R\Modl$ implies $M=0$.
 Unless the class of modules under consideration is suitably restricted,
the similar assertion does \emph{not} hold for the colocalizations.
 The suitable restriction is that, for the purposes of colocalization,
one should only consider \emph{contraadjusted} $R$\+modules
in the sense of~\cite{Pcosh,ST,Pcta,PSl}, i.~e., $R$\+modules $P$ such
that $\Ext^1_R(R[s^{-1}],P)=0$ for all $s\in R$.
 So, the assertions about colocal classes stated above at the end
of Section~\ref{introd-cotorsion-pairs-formulations} presume that
all modules in the class $\sE$ are contraadjusted (or all complexes
in the class $\sE$ are complexes of contraadjusted modules).

\subsection{{}}
 The time has come to explain what the \emph{antilocality} means.
 We say that a system of classes $\sF_R\subset R\Modl$ or
$\sF_R\subset\bC(R\Modl)$ given for all rings $R\in\R$ is
\emph{antilocal} if, for any ring $R\in\R$ and a finite collection of
elements $s_1$,~\dots, $s_d\in R$ generating the unit ideal in $R$,
the following condition holds:
 An $R$\+module or a complex of $R$\+modules $M$ belongs to $\sF_R$
if and only if $M$ is a direct summand of a module/complex admitting
a finite filtration with the successive quotients belonging
to $\sF_{R[s_j^{-1}]}$, where $1\le j\le d$.
 It is understood here that any $R[s_j^{-1}]$\+module can be considered
as an $R$\+module via the restriction of scalars (the direct image).

 To introduce and discuss this condition, which appears naturally from
the self-dual construction of complete cotorsion pairs for an affine
open covering in~\cite[Chapter~4]{Pcosh}, is the main aim of this paper.
 Many examples of antilocal properties or classes of modules and
complexes over commutative rings are presented in the paper.

 The constructions of cotorsion pairs in~\cite[Chapter~4]{Pcosh} apply
to quasi-compact semi-separated schemes, but we only consider modules
over rings in this paper.
 The reason for this choice of the generality level is that we want
to make the self-duality of our concepts and constructions explicit.
 Our main results, as stated at the end
of Section~\ref{introd-cotorsion-pairs-formulations}, apply both to
the local and colocal classes in a self-dual fashion.
 If one wants to globalize over schemes, one observes that
the assertions involving the locality and localizations are applicable
to \emph{quasi-coherent sheaves}, while the claims concerning
the colocality and colocalizations turn into results about
\emph{contraherent cosheaves}.

 The reader can find the relevant constructions of complete cotorsion
pairs in the category of quasi-coherent sheaves over a quasi-compact
semi-separated scheme in~\cite[Section~4.1]{Pcosh}, while the dual
constructions of complete cotorsion pairs in the categories of
contraherent cosheaves are spelled out
in~\cite[Sections~4.2\+-4.3]{Pcosh}.
 In order to avoid going into technical details about the contraherent
cosheaves and related concepts, thus making the exposition more
transparent and accessible, we chose to restrict ourselves to modules
and complexes over rings in this paper.

 Let us emphasize that all the discussions of locality, colocality,
and antilocality in this paper presume the Zariski topology on affine
schemes.
 All the mentions of descent and codescent, etc., refer to Zariski
descent and codescent.
 Similarly, the ``direct image'' refers to direct images under open
immersions of affine schemes.

\subsection*{Acknowledgement}
 I~am grateful to Silvana Bazzoni, Michal Hrbek,
Jan \v St\!'ov\'\i\v cek, and Jan Trlifaj for helpful discussions.
 The author is supported by the GA\v CR project 20-13778S and
research plan RVO:~67985840.

\Section{Preliminaries on Cotorsion Pairs in Exact Categories}
\label{preliminaries-cotorsion-pairs-secn}

 Let $\sE$ be an exact category (in Quillen's sense).
 We suggest the survey paper~\cite{Bueh} as the standard reference on
exact categories.
 A discussion of the Yoneda Ext functor in exact categories can be
found in~\cite[Sections~A.7\+-A.8]{Partin}.

 Let $\sA$ and $\sB\subset\sE$ be two classes of objects.
 One denotes by $\sA^{\perp_1}\subset\sE$ the class of all objects
$X\in\sE$ such that $\Ext_\sE^1(A,X)=0$ for all $A\in\sA$.
 Dually, ${}^{\perp_1}\sB\subset\sE$ is the class of all objects
$Y\in\sE$ such that $\Ext_\sE^1(Y,B)=0$ for all $B\in\sB$.

 A pair of classes of objects $(\sA,\sB)$ in $\sE$ is called
a \emph{cotorsion pair}~\cite{Sal} if $\sA^{\perp_1}=\sB$ and
${}^{\perp_1}\sB=\sA$.
 A cotorsion pair $(\sA,\sB)$ is said to be \emph{generated by}
a class of objects $\sS\subset\sE$ if $\sB=\sS^{\perp_1}$.
 Dually, a cotorsion pair $(\sA,\sB)$ is said to be \emph{cogenerated
by} a class of objects $\sT\subset\sE$ if $\sA={}^{\perp_1}\sT$.
 Clearly, any class of objects generates a cotorsion pair and
cogenerates another cotorsion pair in~$\sE$.

 A class of objects $\sA\subset\sE$ is said to be \emph{generating}
if for every object $E\in\sE$ there exists an admissible epimorphism
$A\rarrow E$ in $\sE$ with $A\in\sA$.
 Dually, a class of objects $\sB\subset\sE$ is said to be
\emph{cogenerating} if for every object $E\in\sE$ there exists
an admissible monomorphism $E\rarrow B$ with $B\in\sB$.

 Let $(\sA,\sB)$ be a cotorsion pair in~$\sE$.
 Assume that the class $\sA$ is generating and the class $\sB$ is
cogenerating in~$\sE$.
 (The former assumption holds automatically when
there are enough projective objects in $\sE$, and dually, the latter
assumption always holds if there are enough injective objects
in~$\sE$.)
 Under these assumptions, the following four conditions are
equivalent~\cite[Theorem~1.2.10]{GR}, \cite[Lemma~6.17]{Sto-ICRA},
\cite[Lemma~1.4]{PS4}:
\begin{enumerate}
\renewcommand{\theenumi}{\roman{enumi}}
\item the class $\sA$ is closed under the passages to the kernels of
admissible epimorphisms in~$\sE$;
\item the class $\sB$ is closed under the passages to the cokernels of
admissible monomorphisms in~$\sE$;
\item $\Ext_\sE^2(A,B)=0$ for all $A\in\sA$ and $B\in\sB$;
\item $\Ext_\sE^n(A,B)=0$ for all $A\in\sA$, \,$B\in\sB$, and $n\ge1$.
\end{enumerate}
 A cotorsion pair satisfying these conditions is said to be
\emph{hereditary}.

 Given two classes of objects $\sA$ and $\sB\subset\sE$, one denotes
by $\sA^{\perp_{\ge1}}\subset\sE$ the class of all objects
$X\in\sE$ such that $\Ext_\sE^n(A,X)=0$ for all $A\in\sA$ and $n\ge1$,
and by ${}^{\perp_{\ge1}}\sB\subset\sE$ is the class of all objects
$Y\in\sE$ such that $\Ext_\sE^n(Y,B)=0$ for all $B\in\sB$.
 So a cotorsion pair $(\sA,\sB)$ in $\sE$ (with a generating class
$\sA$ and a cogenerating class~$\sB$) is hereditary if and only if
$\sB=\sA^{\perp_{\ge1}}$, or equivalently, $\sA={}^{\perp_{\ge1}}\sB$.

\begin{lem} \label{hereditary-lemma}
\textup{(a)} Let\/ $\sS\subset\sE$ be a generating class of objects
closed under the kernels of admissible epimorphisms in\/~$\sE$.
 Then\/ $\sS^{\perp_1}=\sS^{\perp_{\ge1}}\subset\sE$, and the class\/
$\sB=\sS^{\perp_1}$ is closed under the cokernels of admissible
monomorphisms in\/~$\sE$.
 If the class\/ $\sB$ is cogenerating in\/ $\sE$, then the cotorsion
pair $(\sA,\sB)$ generated by\/ $\sS$ is hereditary in\/~$\sE$. \par
\textup{(b)} Let\/ $\sT\subset\sE$ be a cogenerating class of objects
closed under the cokernels of admissible monomorphisms in\/~$\sE$.
 Then\/ ${}^{\perp_1}\sT={}^{\perp_{\ge1}}\sT\subset\sE$, and
the class\/ $\sA={}^{\perp_1}\sT$ is closed under the kernels of
admissible epimorphisms in\/~$\sE$.
 If the class\/ $\sA$ is generating in\/ $\sE$, then the cotorsion
pair $(\sA,\sB)$ cogenerated by\/ $\sT$ is hereditary in\/~$\sE$.
\end{lem}

\begin{proof}
 This is a slight generalization of the lemma about the equivalence
of conditions~(i\+-iv) above.
 Parts~(a) and~(b) are dual to each other; and all the claims follow
easily from the very first assertions of~(a) and~(b).
 The latter are provable by the argument
from~\cite[Lemma~6.17]{Sto-ICRA}; see also~\cite[Lemma~1.3]{BHP}.
\end{proof}

\begin{lem} \label{generated-by-projdim1-lemma}
 In an exact category\/ $\sE$ with enough projective and injective
objects, any cotorsion pair $(\sA,\sB)$ generated by a class of objects
of projective dimension~$\le1$ is hereditary.
 Moreover, the class\/ $\sB$ is closed under admissible epimorphic
images in the exact category\/ $\sE$ in this case, i.~e.,
if $B\rarrow E$ is admissible epimorphism and $B\in\sB$,
then $E\in\sB$; while all objects from the class $\sA$ have
projective dimension~$\le1$ in\/~$\sE$.  \qed
\end{lem}

 A cotorsion pair $(\sA,\sB)$ in $\sE$ is said to be
\emph{complete}~\cite{Sal} if, for every object $E\in\sE$, there exist
(admissible) short exact sequences
\begin{gather}
 0\lrarrow B'\lrarrow A\lrarrow E\lrarrow0
 \label{spec-precover-sequence} \\
 0\lrarrow E\lrarrow B\lrarrow A'\lrarrow0
 \label{spec-preenvelope-sequence}
\end{gather}
in $\sE$ with objects $A$, $A'\in\sA$ and $B$, $B'\in\sB$.
 A short exact sequence~\eqref{spec-precover-sequence} is called
a \emph{special precover sequence}, and a short exact
sequence~\eqref{spec-preenvelope-sequence} is called a \emph{special
preenvelope sequence}.
 The short exact sequences~(\ref{spec-precover-sequence}\+-%
\ref{spec-preenvelope-sequence}) are collectively referred to as
the \emph{approximation sequences}.

\begin{lem} \label{salce-lemma}
 Let $(\sA,\sB)$ be a cotorsion pair in\/ $\sE$ such that the class\/
$\sA$ is generating and the class\/ $\sB$ is cogenerating in\/~$\sE$.
 Then the pair of classes $(\sA,\sB)$ admits special precover sequences
for all objects $E\in\sE$ if and only if it admits special preenvelope
sequences for all $E\in\sE$.
\end{lem}

\begin{proof}
 This is a category-theoretic version of the Salce lemmas~\cite{Sal}.
 The argument from~\cite[Lemma~1.1]{PS4} applies.
\end{proof}

 Given a class of objects $\sC\subset\sE$, let us denote by
$\sC^\oplus\subset\sE$ the class of all direct summands of objects
from $\sC$ in~$\sE$.

\begin{lem} \label{direct-summand-closure-is-complete-cotorsion}
 Let\/ $\sA$ and\/ $\sB\subset\sE$ be two classes of objects in
an exact category\/ $\sE$ such that\/ $\Ext^1_\sA(A,B)=0$ for all
$A\in\sA$ and $B\in\sB$.
 Assume that (admissible) short exact
sequences~\textup{(\ref{spec-precover-sequence}\+-%
\ref{spec-preenvelope-sequence})} with $A$, $A'\in\sA$ and
$B$, $B'\in\sB$ exist for all objects $E\in\sE$.
 Then $(\sA^\oplus,\sB^\oplus)$ is a complete cotorsion pair
in\/~$\sE$.
 In other words, $\sA^{\perp_1}=\sB^\oplus$ and\/ ${}^{\perp_1}\sB=
\sA^\oplus$.
\end{lem}

\begin{proof}
 See~\cite[Lemma~1.2]{PS4}.
\end{proof}

 Let $\sK$ be an exact category and $\sE\subset\sK$ be a full
additive subcategory.
 One says that a full subcategory $\sE$ \emph{inherits an exact category 
structure} from an ambient exact category $\sK$ if the class of all
(admissible) short exact sequences in $\sK$ with the terms belonging to
$\sE$ defines an exact category structure on~$\sE$.
 We refer to~\cite[Theorem~2.6]{DS} or~\cite[Lemma~4.20]{Pedg} for
a characterization of such full subcategories in exact categories.
 Any full subcategory closed under extensions in an exact category
inherits an exact category structure.
 The same applies to any full additive subcategory that is closed
under \emph{both} the admissible subobjects and admissible
epimorphic images.

 Let $\sK$ be an exact category, $\sE\subset\sK$ be a full subcategory
inheriting an exact category structure, and $(\sA,\sB)$ be a complete
cotorsion pair in~$\sK$.
 One says that a complete cotorsion pair $(\sA,\sB)$ \emph{restricts to}
(\emph{a complete cotorsion pair in}) the subcategory $\sE\subset\sK$
if the pair of classes ($\sE\cap\sA$, $\sE\cap\sB$) is a complete
cotorsion pair in~$\sE$.

\begin{lem} \label{cotorsion-pair-restricts}
 Let\/ $\sK$ be an exact category, $\sE\subset\sK$ be a full
subcategory closed under extensions, and $(\sA,\sB)$ be a complete
cotorsion pair in\/~$\sK$.
 Assume that either \par
\textup{(a)} $\sE$ is closed under kernels of admissible
epimorphisms in\/ $\sK$, and\/ $\sA\subset\sE$; or \par
\textup{(b)} $\sE$ is closed under cokernels of admissible
monomorphisms in\/ $\sK$, and\/ $\sB\subset\sE$. \par
\noindent Then the complete cotorsion pair $(\sA,\sB)$ in\/ $\sK$
restricts to a complete cotorsion pair in\/~$\sE$.
\end{lem}

\begin{proof}
 Let us prove part~(a); part~(b) is dual.
 We have to prove that ($\sA$, $\sE\cap\sB$) is a complete cotorsion
pair in~$\sE$.
 For this purpose, let us show that approximation
sequences~(\ref{spec-precover-sequence}\+-%
\ref{spec-preenvelope-sequence}) with objects $A$, $A'\in\sA$
and $B$, $B'\in\sE\cap\sB$ exist in $\sE$ for any object $E\in\sE$.
 Indeed, by assumption, a special precover
sequence~\eqref{spec-precover-sequence} with objects $B'\in\sB$
and $A\in\sA$ exists in $\sK$ for the given object $E\in\sE$.
 Since $E\in\sE$ and $A\in\sA\subset\sE$, and the full subcategory
$\sE$ is closed under kernels of admissible epimorphisms in $\sK$,
it follows that $B'\in\sE\cap\sB$.
 Similarly, a special preenvelope
sequence~\eqref{spec-preenvelope-sequence} with objects $B\in\sB$
and $A'\in\sA$ exists in~$\sK$.
 Since $E\in\sE$ and $A'\in\sA\subset\sE$, and the full subcategory
$\sE$ is closed under extensions in $\sK$, it follows that
$B\in\sE\cap\sB$.

 On the other hand, the classes $\sA$ and $\sE\cap\sB$ are
$\Ext^1$\+orthogonal in $\sE$, since the functors $\Ext^1$ in the exact
categories $\sE$ and $\sK$ agree.
 Furthermore, the class $\sA$ is closed under direct summands in $\sE$,
since it is closed under direct summands in~$\sK$; and the class
$\sE\cap\sB$ is closed under direct summands in $\sE$, since the class
$\sB$ is closed under direct summands in~$\sK$.
 It remains to apply
Lemma~\ref{direct-summand-closure-is-complete-cotorsion}.
\end{proof}

\begin{lem} \label{restricted-cotorsion-hereditary}
 Let\/ $\sK$ be an exact category, $\sE\subset\sK$ be a full
subcategory inheriting an exact category structure, and $(\sA,\sB)$ be
a complete cotorsion pair in\/ $\sK$ that restricts to a complete
cotorsion pair in\/~$\sE$.
 Assume that the cotorsion pair $(\sA,\sB)$ is hereditary in\/~$\sK$.
 Then the cotorsion pair $(\sE\cap\sA$, $\sE\cap\sB)$ is
hereditary in\/~$\sE$.
\end{lem}

\begin{proof}
 The class $\sE\cap\sA$ is generating in $\sE$, and the class
$\sE\cap\sB$ is cogenerating in $\sE$, since the pair of classes
$(\sE\cap\sA$, $\sE\cap\sB)$ is complete cotorsion pair in $\sE$
by assumption.
 Now it remains to observe that if the class $\sA$ is closed under
the kernels of admissible epimorphisms in $\sK$, then the class
$\sE\cap\sA$ is closed under the kernels of admissible epimorphisms
in~$\sE$.
 Alternatively, the dual argument proves that the class $\sE\cap\sB$
is closed under the cokernels of admissible monomorphisms in~$\sE$.
\end{proof}

 The following Ext\+adjunction lemma is helpful for comparing cotorsion
pairs in two exact categories connected by a pair of adjoint functors.
 It should be compared to Lemma~\ref{spectral-ext-adjunction-lemma}
below, which provides some additional explanations.

\begin{lem} \label{categorical-ext-adjunction-lemma}
 Let\/ $\sE$ and\/ $\sF$ be exact categories, and let\/ $\Phi\:\sE
\rarrow\sF$ and\/ $\Psi\:\sF\rarrow\sE$ be a pair of adjoint functors,
with\/ $\Phi$ left adjoint to\/~$\Psi$.
 In this context: \par
\textup{(a)} If the functors\/ $\Phi$ and\/ $\Psi$ are exact (i.~e.,
take admissible short exact sequences to admissible short exact
sequences), then for any two objects $E\in\sE$ and $F\in\sF$ and all
integers $m\ge0$ there is a natural isomorphism of the\/ $\Ext$ groups
\begin{equation} \label{ext-m-adjunction}
 \Ext_\sE^m(E,\Psi(F))\simeq\Ext_\sF^m(\Phi(E),F).
\end{equation} \par
\textup{(b)} If the functor\/ $\Phi$ is exact, then for any two objects
$E\in\sE$ and $F\in\sF$ there is a natural monomorphism of the groups\/
$\Ext^1$,
\begin{equation} \label{ext-1-monomorphism}
 \Ext_\sE^1(E,\Psi(F))\lrarrow\Ext_\sF^1(\Phi(E),F).
\end{equation} \par
\textup{(c)} More generally, if an object $E\in\sE$ has the property
that the functor\/ $\Phi$ takes any (admissible) short exact sequence\/
$0\rarrow E'\rarrow E''\rarrow E\rarrow0$ in\/ $\sE$ to a short exact
sequence in\/ $\sF$, then for any object $F\in\sF$ there is a natural
monomorphism~\eqref{ext-1-monomorphism} of the groups\/ $\Ext^1$. \par
\textup{(d)} Dually, if an object $F\in\sF$ has the property
that the functor\/ $\Psi$ takes any short exact sequence\/
$0\rarrow F\rarrow F''\rarrow F'\rarrow0$ in\/ $\sF$ to a short exact
sequence in\/ $\sE$, then for any object $E\in\sE$ there is a natural
monomorphism of the groups\/ $\Ext^1$ acting in the opposite direction,
\begin{equation} \label{ext-1-dual-monomorphism}
 \Ext_\sF^1(\Phi(E),F)\lrarrow\Ext_\sE^1(E,\Psi(F)).
\end{equation} \par
\textup{(e)} Under the combined assumptions of parts~\textup{(c)}
and~\textup{(d)} concerning objects $E\in\sE$ and $F\in\sF$, there is
a natural isomorphism of the groups\/ $\Ext^1$,
\begin{equation} \label{ext-1-isomorphism}
 \Ext_\sE^1(E,\Psi(F))\simeq\Ext_\sF^1(\Phi(E),F).
\end{equation}
\end{lem}

\begin{proof}
 We will only prove part~(c).
 Given a short exact sequence $0\rarrow\Psi(F)\rarrow E''\rarrow E
\rarrow0$ in $\sE$, the map~\eqref{ext-1-monomorphism} does
the following.
 Applying the functor $\Phi$ and using the assumption of part~(c),
we obtain a short exact sequence $0\rarrow\Phi\Psi(F)\rarrow\Phi(E'')
\rarrow\Phi(E)\rarrow0$ in~$\sF$.
 It remains to take the pushout of the latter short exact sequence
with respect to the adjunction morphism $\Phi\Psi(F)\rarrow F$ in
order to obtain the desired short exact sequence $0\rarrow F\rarrow
F''\rarrow\Phi(E)\rarrow0$ in~$\sF$.
$$
 \xymatrix{
  0\ar[r] & \Phi\Psi(F) \ar[r] \ar[d] & \Phi(E'') \ar[r] \ar@{..>}[ld]
  & \Phi(E) \ar[r] & 0 \\
  & F
 }
$$

 Now suppose that the resulting short exact sequence in $\sF$ splits.
 This means that the adjunction morphism $\Phi\Psi(F)\rarrow F$
factorizes through the admissible monomorphism $\Phi\Psi(F)\rarrow
\Phi(E'')$ in~$\sF$.
 So we obtain a morphism $\Phi(E'')\rarrow F$ in~$\sF$.
 By adjunction, there is the corresponding morphism $E''\rarrow
\Psi(F)$, which splits the original short exact sequence
$0\rarrow\Psi(F)\rarrow E''\rarrow E\rarrow0$ in~$\sE$.
\end{proof}

 For any exact category $\sE$, the additive category $\bC(\sE)$ of
cochain complexes in $\sE$ has a natural exact category structure in
which a short sequence of complexes $0\rarrow A^\bu\rarrow B^\bu
\rarrow C^\bu\rarrow0$ is (admissible) exact in $\bC(\sE)$ if and only
if, at every fixed cohomological degree~$n$, the short sequence
$0\rarrow A^n\rarrow B^n\rarrow C^n\rarrow0$ is exact in~$\sE$.
 If $\sE$ is an abelian exact category with the abelian exact structure,
then $\bC(\sE)$ is also an abelian exact category, and the exact
structure on it defined by the rule above is the abelian
exact structure.

 Given an object $E\in\sE$, we denote by $D_{n,n+1}^\bu(E)\in\bC(\sE)$
the contractible two-term complex $\dotsb\rarrow0\rarrow E
\overset=\rarrow E\rarrow0\rarrow\dotsb$ sitting at the cohomological
degrees $n$ and~$n+1$.
 The following lemma is quite standard.

\begin{lem} \label{disk-complexes-lemma}
 Let\/ $\sE$ be an exact category and $C^\bu\in\bC(\sE)$ be a complex
in\/~$\sE$.
 Then the isomorphisms of Yoneda Ext groups
\begin{align*}
 \Ext^i_{\bC(\sE)}(D_{n,n+1}^\bu(E),C^\bu) &\simeq
 \Ext^i_\sE(E,C^n) \\
 \Ext^i_{\bC(\sE)}(C^\bu,D_{n,n+1}^\bu(E)) &\simeq
 \Ext^i_\sE(C^{n+1},E)
\end{align*}
hold for all $E\in\sE$, \,$n\in\boZ$, and $i\ge0$.
\end{lem}

\begin{proof}
 Extend the construction of the complex $D_{n,n+1}^\bu(E)$ to a pair of
functors adjoint on the left and on the right to the forgetful functor
from $\bC(\sE)$ to the category $\sE^\boZ$ of graded objects in $\sE$,
and apply Lemma~\ref{categorical-ext-adjunction-lemma}(a).
 For the details on the construction of the adjoint functors, see,
e.~g., \cite[Section~5]{PS4}.
\end{proof}

 For any additive category $\sE$, we denote by $\bH(\sE)$
the triangulated category of cochain complexes in $\sE$ with
morphisms up to cochain homotopy.
 The notation $C^\bu\longmapsto C^\bu[n]$ refers to the functors of
shift of cohomological grading on the complexes, $C^\bu[n]^i=C^{n+i}$.

\begin{lem} \label{ext-homotopy-hom-lemma}
 Let\/ $\sE$ be an exact category and $A^\bu$, $B^\bu\in\bC(\sE)$ be
two complexes in\/~$\sE$.
 Then there is a natural injective map of abelian groups
$$
 \Hom_{\bH(\sE)}(A^\bu,B^\bu[1])\lrarrow
 \Ext_{\bC(\sE)}^1(A^\bu,B^\bu),
$$
whose image consists precisely of all the extensions represented
by short exact sequences of complexes\/ $0\rarrow B^\bu\rarrow C^\bu
\rarrow A^\bu\rarrow0$ in which the short exact sequence\/
$0\rarrow B^n\rarrow C^n\rarrow A^n\rarrow0$ is split in\/ $\sE$
for every $n\in\boZ$.
 In particular, if\/ $\Ext^1_\sE(A^n,B^n)=0$ for every $n\in\boZ$,
then
$$
 \Hom_{\bH(\sE)}(A^\bu,B^\bu[1])\simeq
 \Ext_{\bC(\sE)}^1(A^\bu,B^\bu).
$$
\end{lem}

\begin{proof}
 This observation is well-known and goes back, at least,
to~\cite[Section~1.3]{Bec}; see~\cite[Lemma~1.6]{BHP} for some details.
\end{proof}

 In the rest of the section we briefly discuss filtrations and
cotorsion pairs in Grothendieck abelian categories with enough
projective objects.
 This is a well-known particular case of a more general theory
developed in the papers~\cite{PR,PS4}; but only this particular case
is needed for the purposes of the present paper.

 Let $\sK$ be a Grothendieck category and $\alpha$~be an ordinal.
 An \emph{$\alpha$\+indexed filtration} on an object $F\in\sK$ is
a family of subobjects $(F_\beta\subset F)_{0\le\beta\le\alpha}$
satisfying the following conditions:
\begin{itemize}
\item $F_0=0$ and $F_\alpha=F$;
\item $F_\gamma\subset F_\beta$ for all $0\le\gamma\le\beta\le\alpha$;
\item $F_\beta=\bigcup_{\gamma<\beta}F_\gamma$ for all limit
ordinals $\beta\le\alpha$.
\end{itemize}
 Given a filtration $(F_\beta)_{0\le\beta\le\alpha}$ on an object
$F\in\sK$, one says that the object $F$ is \emph{filtered by}
the successive quotient objects $S_\beta=F_{\beta+1}/F_\beta$,
\ $0\le\beta<\alpha$.
 In an alternative language, the object $F$ is said to be
a \emph{transfinitely iterated extension} (in the sense of the direct
limit) of the objects $(S_\beta\in\sK)_{0\le\beta<\alpha}$.

 Given a class of objects $\sS\subset\sK$, we denote by $\Fil(\sS)$
the class of all objects filtered by (objects isomorphic to)
the objects from~$\sS$.
 The following result is known classically as
the \emph{Eklof lemma}~\cite[Lemma~1]{ET}.

\begin{lem} \label{eklof-lemma}
 Let\/ $\sK$ be a Grothendieck category and\/ $\sB\subset\sK$ be
a class of objects.
 Then the class of objects\/ ${}^{\perp_1}\sB$ is closed under
transfinitely iterated extensions in\/~$\sK$; so ${}^{\perp_1}\sB
=\Fil({}^{\perp_1}\sB)\subset\sK$.
\end{lem}

\begin{proof}
 This assertion, properly understood, holds in any exact category;
see~\cite[Lemma~4.5]{PR} for an argument applicable in such generality.
 For an exposition in the generality of certain exact category
analogues of Grothendieck abelian categories,
see~\cite[Proposition~5.7]{Sto-ICRA}.
\end{proof}

 The following classical result is due to Eklof
and Trlifaj~\cite[Theorems~2 and~10]{ET}.

\begin{thm} \label{eklof-trlifaj-theorem}
 Let\/ $\sK$ be a Grothendieck abelian category with enough projective
objects.
 Then any cotorsion pair generated by a \emph{set} of objects
in\/ $\sK$ is complete.
 More precisely, if\/ $\sS\subset\sK$ is a set of objects containing
a projective generator of\/ $\sK$, then the cotorsion pair $(\sA,\sB)$
generated by\/ $\sS$ in\/ $\sK$ is complete and\/ $\sA=\Fil(\sS)^\oplus
\subset\sK$.
\end{thm}

\begin{proof}
 This theorem, properly stated, holds in any locally presentable
abelian category; see~\cite[Corollary~3.6 and Theorem~4.8]{PR}
or~\cite[Theorems~3.3 and~3.4]{PS4}.
 For a version applicable to a certain class of exact categories
generalizing Grothendieck abelian categories,
see~\cite[Theorem~5.16]{Sto-ICRA}.
\end{proof}

 A class of objects $\sA$ in a Grothendieck abelian category $\sK$ is
said to be \emph{deconstructible} if there exists a set of objects
$\sS\subset\sK$ such that $\sA=\Fil(\sS)$.
 It follows from Lemma~\ref{eklof-lemma} and
Theorem~\ref{eklof-trlifaj-theorem} that \emph{any
cotorsion pair generated by a deconstructible class of objects in
a Grothendieck category with enough projective objects is complete}.
 Indeed, Lemma~\ref{eklof-lemma} tells that the cotorsion pair
generated by $\Fil(\sS)$ in $\sK$ coincides with the one generated
by~$\sS$.

\Section{Locality: First Examples and Counterexamples}
\label{locality-first-examples-secn}

 Throughout this paper, we work in the following notation and setting.
 We consider a commutative ring $R$, an arbitrary element $s\in R$,
and a finite collection of elements $s_1$,~\dots, $s_d\in R$ generating
the unit ideal in~$R$.

 The notation $R[s^{-1}]$ stands for the localization of the ring $R$
with respect to (the multiplicative subset spanned by) the element~$s$.
 So $R[s^{-1}]=S^{-1}R$, where $S=\{1,s,s^2,s^3,\dotsc\}$.
 For any $R$\+module $M$, we put $M[s^{-1}]=S^{-1}M=R[s^{-1}]\ot_RM$.
 Both the localization $M[s^{-1}]$ and the \emph{colocalization}
$\Hom_R(R[s^{-1}],M)$ are $R[s^{-1}]$\+modules.
 The same localization functor $M\longmapsto M[s^{-1}]$, as well as
the colocalization functor $M\longmapsto\Hom_R(R[s^{-1}],M)$, is
applied to complexes of $R$\+modules termwise.

 Furthermore, we suppose given a class of commutative rings $\R$ that
is stable under localizations with respect to elements, that is,
$R[s^{-1}]\in\R$ for any $R\in\R$ and $s\in R$.
 We denote by $\sK_R$ the big ambient abelian category associated with
$R$, which may be either the category of $R$\+modules $\sK_R=R\Modl$ or
the category of (unbounded) complexes of $R$\+modules
$\sK_R=\bC(R\Modl)$.

 Then we consider a system of classes of objects or full subcategories
$\sL_R\subset\sK_R$ defined for all the rings $R\in\R$.
 Abusing terminology, we will speak of
``a~class $\sL=(\sL_R)_{R\in\R}$'' for brevity.
 Alternatively, we will also speak of $\sL=(\sL_R)_{R\in\R}$ as
``a~property of modules or complexes over commutative rings $R\in\R$''
(presuming the property of a module or complex over $R$ to belong
to~$\sL_R$).

 A class/property of modules or complexes $\sL=(\sL_R)_{R\in\R}$ is
called \emph{local} if it satisfies the following two conditions:
\begin{description}
\item[Ascent] For any ring $R\in\R$, element $s\in R$, and module or
complex $M\in\sL_R$, the module/complex $M[s^{-1}]$ belongs to
$\sL_{R[s^{-1}]}$.
\item[Descent] Let $R\in\R$ be a ring and $s_1$,~\dots, $s_d\in R$ be
a finite collection of elements generating the unit ideal in~$R$.
 Let $M\in\sK_R$ be a module or complex such that $M[s_j^{-1}]\in
\sL_{R[s_j^{-1}]}$ for every $1\le j\le d$.
 Then $M\in\sL_R$.
\end{description}

\begin{rem}
 Speaking in terms of quasi-coherent sheaves over schemes, one can
say that the definitions of ascent and descent above are stated for
principal affine open subschemes $\Spec R[s^{-1}]$ in affine schemes
$\Spec R$ and principal affine open coverings $\Spec R=\bigcup_{j=1}^d
\Spec R[s_j^{-1}]$ only.
 More generally, one could consider \emph{not necessarily principal} 
affine open subschemes $\Spec S$ in affine schemes $\Spec R$
and coverings of affine schemes by such open subschemes.
 Quite generally, one can consider locality of properties of
quasi-coherent sheaves or complexes over schemes, defined in terms of
ascent and descent conditions for arbitrary Zariski open coverings.

 All these points of view are equivalent.
 Any local property of modules or complexes in the sense of our
definition above defines a local property of quasi-coherent sheaves
or complexes over schemes coverable by affine open subschemes
$\Spec R$ with $R\in\R$, where the locality is understood as ascent
for all open subschemes and descent for all Zariski open coverings.
 We refer to~\cite[Lemma~2.1]{ES}, \cite[Lemma~5.3.2]{Vak},
and~\cite[Section Tag~01OO]{SP} for further discussions with details.
\end{rem}

 Furthermore, we will say that a class $\sL=(\sL_R)_{R\in\R}$ is
\emph{very local} if it is local and satisfies the following additional
\begin{description}
\item[Direct image condition] For any ring $R\in\R$ and element
$s\in R$, any module or complex from $\sL_{R[s^{-1}]}$, viewed as
module/complex over $R$, belongs to~$\sL_R$.
\end{description}

\begin{lem} \label{ascent+direct-image-imply-descent}
 Assume that a class\/ $\sL=(\sL_R\subset\sK_R)_{R\in\R}$ satisfies
the ascent and direct image conditions.
 Assume further that, for every ring $R\in\R$, the full subcategory\/
$\sL_R$ is closed under finite direct sums and the kernels of
epimorphisms in\/~$\sK_R$.
 Then the class\/ $\sL$ also satisfies descent; so, it is very local.
\end{lem}

\begin{proof}
 Let $s_1$,~\dots, $s_d\in R$ be a collection of elements generating
the unit ideal.
 Then the \v Cech coresolution
\begin{multline} \label{cech-coresolution}
 0\lrarrow M\lrarrow\bigoplus\nolimits_{i=1}^d M[s_i^{-1}]\lrarrow
 \bigoplus\nolimits_{1\le i<j\le d}M[s_i^{-1},s_j^{-1}] \lrarrow\dotsb
 \\ \lrarrow\bigoplus\nolimits_{1\le i_1<\dotsb<i_k\le d}
 M[s_{i_1}^{-1},\dotsc,s_{i_k}^{-1}]\lrarrow\dotsb\lrarrow
 M[s_1^{-1},\dotsc,s_d^{-1}]\lrarrow0
\end{multline}
is a finite exact sequence of $R$\+modules for any $R$\+module~$M$.
 For a complex of $R$\+modules $M$,
the bicomplex~\eqref{cech-coresolution} is a finite exact sequence
of complexes of $R$\+modules.
 To show that the sequence~\eqref{cech-coresolution} is exact, one
can, e.~g., localize it at every prime ideal $\mathfrak p\subset R$
and see that the resulting complex of $R_{\mathfrak p}$\+modules
is contractible.

 Now if $M[s_j^{-1}]\in\sL_{R[s_j^{-1}]}$ for every $1\le j\le d$ and
the class $\sL$ satisfies the ascent and direct image conditions, then
all the $R$\+modules or complexes of $R$\+modules
$M[s_{i_1}^{-1},\dotsc,s_{i_k}^{-1}]$ appearing
in~\eqref{cech-coresolution}, except perhaps $M$ itself,
belong to~$\sL_R$.
 The assumption that $\sL_R$ is closed under finite direct sums
in $\sK_R$ implies that all the terms of the exact
sequence~\eqref{cech-coresolution}, except perhaps the leftmost one,
belong to~$\sL_R$.
 Finally, using the assumption that the class $\sL_R$ is closed under
the kernels of epimorphisms in $\sK_R$ and moving by induction from
the rightmost end of the sequence to its leftmost end, one proves
that $M\in\sL_R$.
\end{proof}

\begin{ex} \label{flatness-local-example}
 The flatness property of modules over commutative rings is very local.
 Indeed, one can easily see that the $S$\+module $S\ot_RF$ is flat
for any flat $R$\+module $F$ and any ring homomorphism $R\rarrow S$;
so the ascent is satisfied.
 Furthermore, the $R$\+module $G$ is flat for any commutative ring
homomorphism $R\rarrow S$ making $S$ a flat $R$\+module and any
flat $S$\+module~$G$; so the direct image condition holds as well.
 As the class of flat $R$\+modules is closed under (finite or
infinite) direct sums and kernels of epimorphisms,
Lemma~\ref{ascent+direct-image-imply-descent} tells that the descent
is satisfied.
\end{ex}

 Let $R$ be a commutative ring.
 An $R$\+module $C$ is said to be
\emph{contraadjusted}~\cite[Section~1.1]{Pcosh},
\cite[Section~5]{ST}, \cite[Section~2]{Pcta}
if $\Ext^1_R(R[s^{-1}],C)=0$ for all elements $s\in R$.
 One should keep in mind that the projective dimension of
the flat $R$\+module $R[s^{-1}]$ never exceeds~$1$
\,\cite[proof of Lemma~2.2]{ST}, \cite[proof of Lemma~2.1]{Pcta}
(see~\cite[Corollary~2.23]{GT} for a much more general result).

\begin{lem} \label{contraadjusted-characterized}
 Let $R$ be a commutative ring, $s\in R$ be an element, and
$C$ be an $R$\+module.
 Then one has $\Ext_R^1(R[s^{-1}],C)=0$ if and only if, for any
sequence of elements $a_0$, $a_1$, $a_2$,~\dots~$\in C$, the infinite
system of nonhomogeneous linear equations
\begin{equation} \label{contraadjustedness-equation-system}
 b_n-sb_{n+1}=a_n, \qquad n\ge0
\end{equation}
has a (possibly nonunique) solution $b_0$, $b_1$, $b_2$,~\dots~$\in C$.
\end{lem}

\begin{proof}
 This is what one obtains by computing the Ext module in question in
terms of a natural projective resolution of the $R$\+module $R[s^{-1}]$.
 See~\cite[Lemma~5.1]{ST} or~\cite[Lemma~2.1(a)]{Pcta} for the details.
\end{proof}

 An $R$\+module $F$ is said to be \emph{very flat} if
$\Ext^1_R(F,C)=0$ for all contraadjusted $R$\+modules~$C$.
 In other words, the pair of classes (very flat $R$\+modules,
contraadjusted $R$\+modules) is defined as the cotorsion pair
in $R\Modl$ generated by the set of modules
$\{\,R[s^{-1}]\in R\Modl\mid s\in R\,\}$.
 By Lemma~\ref{generated-by-projdim1-lemma} and
Theorem~\ref{eklof-trlifaj-theorem}, this cotorsion pair is
hereditary and complete.
 Moreover, Theorem~\ref{eklof-trlifaj-theorem} tells that the class
of very flat $R$\+modules can be described as the class of all direct
summands of $R$\+modules filtered by the $R$\+modules $R[s^{-1}]$,
\,$s\in R$ \,\cite[Theorem~1.1.1 and Corollary~1.1.4]{Pcosh}.
 Obviously, any projective $R$\+module is very flat, and any very
flat $R$\+module is flat.

\begin{ex} \label{very-flatness-local-example}
 The very flatness property of modules over commutative rings is
very local.
 The basic point is that the system of classes $\sL_R$ of all
$R$\+modules of the form $R[s^{-1}]$, \,$s\in R$, satisfies the ascent
and direct image conditions (for all commutative rings~$R$).
 The ascent and direct image conditions for the class of very
flat modules follow easily (e.~g, one can use the description of
very flat modules in terms of filtrations).

 More generally, the $S$\+module $S\ot_RF$ is very flat for any
very flat $R$\+module $F$ and any commutative ring homomorphism
$R\rarrow S$ \,\cite[Lemma~1.2.2(b)]{Pcosh}.
 The property that the underlying $R$\+module of any very flat
$S$\+module is very flat holds for any commutative ring
homomorphism $R\rarrow S$ such that the $R$\+module $S[s^{-1}]$
is very flat for all $s\in S$ \,\cite[Lemma~1.2.3(b)]{Pcosh}.

 The (very flat, contraadjusted) cotorsion pair is hereditary,
so the class of very flat modules is closed under kernels of
epimorphisms.
 It is also obviously closed under (finite or infinite) direct sums.
 Applying Lemma~\ref{ascent+direct-image-imply-descent}, one concludes
that the descent is satisfied~\cite[Lemma~1.2.6(a)]{Pcosh}.
\end{ex}

\begin{ex} \label{projectivity-local-example}
 The projectivity property of modules over commutative rings is local.
 The ascent is obvious: for any projective module $P$ over a ring $R$
and any ring homomorphism $R\rarrow S$, the $S$\+module $S\ot_RP$ is
projective.
 The descent is a celebrated (and difficult) theorem of Raynaud
and Gruson~\cite[\S\,II.3.1]{RG}, \cite{Pe}.

 However, the projectivity is \emph{not} very local: the direct image
condition is not satisfied.
 Indeed, given a commutative ring $R$ and an element $s\in R$,
the free $R[s^{-1}]$\+module $R[s^{-1}]$ is usually \emph{not}
projective as an $R$\+module.
 So Lemma~\ref{ascent+direct-image-imply-descent} is not applicable
to the class of projective modules.
\end{ex}

 An $R$\+module $C$ is said to be \emph{cotorsion} if $\Ext_R^1(F,C)=0$
for all flat $R$\+modules~$C$.
 Over any associative ring $R$, the pair of classes (flat modules,
cotorsion modules) is a hereditary complete cotorsion pair.
 This assertion (specifically, the completeness claim) is one of
the formulations of the \emph{flat cover conjecture}, which was
suggested by Enochs in~\cite[page~196]{En} and proved
in the paper~\cite{BBE}.
 Obviously, any cotorsion module over a commutative ring is
contraadjusted (see~\cite[Section~12]{Pcta} for an introductory
discussion in the case of abelian groups).

 As a preliminary remark to the examples below, it is helpful to
observe that the contraadjustedness and cotorsion properties are
preserved by all restrictions of scalars.
 For any (commutative) ring homomorphism $R\rarrow S$ and any
cotorsion (respectively, contraadjusted) $S$\+module $C$,
the underlying $R$\+module of $C$ is also cotorsion (resp.,
contraadjusted) \cite[Lemmas~1.3.4(a) and~1.2.2(a)]{Pcosh}.
 In particular, the preservation of the contraadjustedness property
by restrictions of scalars is immediately obvious
from Lemma~\ref{contraadjusted-characterized}.
 See Examples~\ref{contraadjustedness-colocal-example}
and~\ref{cotorsion-colocal-example} below for a brief discussion.

\begin{ex} \label{cotorsion-not-local-example}
 The class of cotorsion modules is \emph{not} local; in fact, it
does not even satisfy ascent (not even over Noetherian rings).
 The class of contraadjusted modules does \emph{not} satisfy ascent,
either.
 The following counterexample serves to demonstrate both of
these negations.

 Any complete Noetherian commutative local ring $\widehat R$ is
a cotorsion module over itself, as one can see from the classification
of flat cotorsion modules over Noetherian commutative rings
in~\cite[Section~2]{En2}.
 More generally, any $\widehat R$\+contramodule is a cotorsion
$\widehat R$\+module by~\cite[Theorem~9.3]{Pcta}.

 Let $R=k[x,y]$ be the ring of polynomials in two variables over
a field~$k$.
 Then the ring of formal Taylor power series $\widehat R=k[[x,y]]$
is a cotorsion module over itself, and consequently also a cotorsion
(hence contraadjusted) module over~$R$.
 We claim that the localization $\widehat R[x^{-1}]$ is \emph{not}
a contraadjusted $R$\+module, hence also not a contraadjusted
$R[x^{-1}]$\+module and not a cotorsion $R[x^{-1}]$\+module.

 The point is that the localization does not commute with
the completion.
 The complete one-dimensional Noetherian local ring
$k[[x]][x^{-1}][[y]]$ is a cotorsion module over itself and over
$k[x,y]$, but the localization $\widehat R[x^{-1}]=k[[x,y]][x^{-1}]$
is only a dense subring in $k[[x]][x^{-1}][[y]]$,
$$
 C=k[[x,y]][x^{-1}]\subsetneq k[[x]][x^{-1}][[y]]=\widehat C.
$$
 For example, the element $\sum_{n=0}^\infty x^{-n}y^n$ belongs
to $\widehat C$ but not to~$C$.

 Specifically, let us show that
$\Ext^1_R(R[y^{-1}],\widehat R[x^{-1}])\ne0$.
 Following Lemma~\ref{contraadjusted-characterized},
for this purpose if suffices to present a sequence of elements
$(a_n\in\widehat R[x^{-1}])_{n\ge0}$ such that the system of
nonhomogeneous linear
equations~\eqref{contraadjustedness-equation-system} for $s=y$ has
no solutions in $C=\widehat R[x^{-1}]$.
 The point is that the ambient $R$\+module
$\widehat C=k[[x]][x^{-1}][[y]]$ contains no elements infinitely
divisible by~$y$.
 For this reason, the system of
equations~\eqref{contraadjustedness-equation-system} is always
at most uniquely solvable in~$\widehat C$ \,\cite[Lemma~2.1(b)]{Pcta}
(and in fact the unique solution exists in $\widehat C$, as one can
see from~\cite[Theorem~3.3(c)]{Pcta}).

 Now put $a_n=x^{-n}\in C$ for all $n\ge0$.
 Then the system of elements $b_n=\sum_{i=0}^\infty y^na_{n+i}=
\sum_{i=0}^\infty x^{-n-i}y^n\in\widehat C$ forms a solution
of~\eqref{contraadjustedness-equation-system} for $s=y$.
 Since $b_0\notin C$, it follows
that~\eqref{contraadjustedness-equation-system} is unsolvable
in $C$ for $s=y$ and $a_n=x^{-n}$.

 One of the aims of this paper is to explain that the class of
cotorsion modules is \emph{antilocal}
(see Example~\ref{cotorsion-antilocal-example} below).
 The class of contraadjusted modules is likewise antilocal
(by Example~\ref{contraadjustedness-antilocal-example}).
 The cotorsion property is also \emph{colocal} presuming 
contraadjustedness, as we will see below in
Example~\ref{cotorsion-colocal-example}.
\end{ex}

\begin{ex} \label{injective-local-and-non-local-example}
 The class of injective modules satisfies the direct image
condition (see Example~\ref{injectivity-colocal-example} below).
 Over Noetherian rings, it also satisfies ascent and descent; this is
a theorem of Hartshorne~\cite[Lemma~II.7.16 and Theorem~II.7.18]{Hart}
based on Matlis' classification of injective modules over
Noetherian rings~\cite{Mat}.
 So the class of injective modules is very local over Noetherian rings.

 However, over non-Noetherian commutative rings injectivity is
\emph{not} a local property, and in fact, it does not even satisfy
ascent.
 The following example, based on the previous
Example~\ref{cotorsion-not-local-example}, shows that the localization
of an injective module with respect to an element of the ring
need not even be contraadjusted.

 Put $R=k[x,y]$ as in the previous example, and consider
the $R$\+module $M=k[x,x^{-1}]/k[x]\ot_k k[y,y^{-1}]/k[y]$.
 Denote by $S$ the trivial extension of the ring $R$ by the $R$\+module
$M$, i.~e., $S=R\oplus M$ and $R$ is a subring in $S$ with the product
of two elements from $R$ and $M$ in $S$ defined in terms of
the $R$\+module structure of $M$, while the product of any two
elements from $M$ is zero in~$S$.
 Here the $R$\+module $M$ is chosen in such a way that its dual
$R$\+module $\Hom_k(M,k)$ is isomorphic to the $R$\+module
$\widehat R=k[[x,y]]$ from the previous example.

 Put $J=\Hom_k(S,k)$; so $J$ is an injective $S$\+module.
 The underlying $R$\+module of $J$ is isomorphic to the direct sum
of the injective $R$\+module $\Hom_k(R,k)$ and
the $R$\+module~$\widehat R$, i.~e.,
$J\simeq\Hom_k(R,k)\oplus\widehat R$.
 Now we already know from the previous
Example~\ref{cotorsion-not-local-example} that $\widehat R[x^{-1}]$
is not a contraadjusted $R$\+module.
 Hence $J[x^{-1}]$ is not a contraadjusted $R$\+module, either;
and consequently $J[x^{-1}]$ is not a contraadjusted $S$\+module
and \emph{not} a contraadjusted $S[x^{-1}]$\+module.

 We will see below that the class of injective modules over arbitrary
commutative rings is actually both \emph{colocal} presuming
contraadjustedness (by Example~\ref{injectivity-colocal-example})
and \emph{strongly antilocal}
(see Example~\ref{injectivity-strongly-antilocal-example}).
\end{ex}

 A complex of $R$\+modules $J^\bu$ is said to be \emph{homotopy
injective} (or ``$K$\+injective'')~\cite{Spal} if, for any acyclic
complex of $R$\+modules $X^\bu$, any morphism of complexes of
$R$\+modules $X^\bu\rarrow J^\bu$ is homotopic to zero.
 Any complex of $R$\+modules is quasi-isomorphic to a homotopy
injective complex (and even to a homotopy injective complex of
injective $R$\+modules), which is defined uniquely up to
homotopy equivalence.

\begin{ex} \label{homotopy-injectivity-not-local-example}
 The homotopy injectivity of complexes of $R$\+modules is not a local
property, and in fact, it does not satisfy ascent, \emph{not} even
over Noetherian commutative rings~$R$.
 The following counterexample is essentially due to
Neeman~\cite[Example~6.5]{Neem-bb}.
 An exposition is available from~\cite{Bel}, so we restrict ourselves
here to a brief sketch.

 Let $R$ be any ``nontrivial enough'' Noetherian commutative ring and
$f\in R$ be a generic element.
 It suffices to take $R=k[x]$ to be the polynomial ring in one variable
over a field~$k$, and $f\in R$ any polynomial of degree~$\ge1$; or
$R=\boZ$ to be the ring of integers and $f\in R$ any integer with
the absolute value~$\ge2$.
 In both cases, it is important that $f$ be neither nilpotent nor
invertible.
 Specifically, we need the natural map of $R[f^{-1}]$\+modules
\begin{equation} \label{product-localization-not-commute}
 \left(\prod\nolimits_{m=0}^\infty R\right)[f^{-1}]\lrarrow
 \prod\nolimits_{m=0}^\infty R[f^{-1}]
\end{equation}
to be \emph{not} an isomorphism.

 Consider the algebra of dual numbers $S=R[\epsilon]/(\epsilon^2)$
over~$R$.
 Choose an injective coresolution $I_S^\bu$ of the free $S$\+module $S$,
and consider the homotopy injective complex of injective $S$\+modules
$J_S^\bu=\prod_{n\in\boZ}I_S^\bu[n]$.
 The complex $J_S^\bu$ is a quasi-isomorphic to the complex of
$S$\+modules $\prod_{n\in\boZ}S[n]=C_S^\bu=\bigoplus_{n\in\boZ}S[n]$
with all the terms isomorphic to $S$ and zero differential.
 Notice that the complex $C_S^\bu$ is also quasi-isomorphic to
the complex of injective $S$\+modules $\bigoplus_{n\in\boZ}I_S^\bu[n]$;
but the latter complex of injective $S$\+modules is \emph{not}
homotopy injective.

 Put $T=S[f^{-1}]=R[f^{-1}][\epsilon]/(\epsilon^2)$.
 Applying the same construction to the ring $T$, we obtain
a homotopy injective complex of injective $T$\+modules $J_T^\bu=
\prod_{n\in\boZ}I_T^\bu$, where we prefer to choose
$I_T^\bu=I_S^\bu[f^{-1}]$ as an injective coresolution of the free
$T$\+module~$T$.
 Here $I_S^\bu[f^{-1}]$ is a complex of injective $T$\+modules by
Hartshorne's theorem, as per the previous
Example~\ref{injective-local-and-non-local-example}.
 Similarly, $J_S^\bu[f^{-1}]$ is a complex of injective $T$\+modules.

 Now there is a natural morphism of complexes of injective $T$\+modules
$g\:J_S^\bu[f^{-1}]\allowbreak\rarrow J_T^\bu$, which is easily seen to
be a quasi-isomorphism.
 Indeed, both the complexes are quasi-isomorphic to~$C_T^\bu$.
 We claim that the complex of injective $T$\+modules $J_S^\bu[f^{-1}]$
is \emph{not} homotopy injective (this provides the promised
counterexample to ascent of homotopy injectivity).
 In order to prove as much, it suffices to show that $g$~is not
a homotopy equivalence of complexes of $T$\+modules.

 In fact, $g$~is not even a homotopy equivalence of complexes of
modules over the ring $\Lambda=\boZ[\epsilon]/(\epsilon^2)$.
 Indeed, applying the functor $\Hom_\Lambda(\boZ,{-})=
\Hom_T(T/\epsilon T,{-})$ transforms~$g$ into a morphism of complexes
of abelian groups (or $R[f^{-1}]$\+modules) which is not
a quasi-isomorphism.
 Specifically, computing the cohomology map induced by the morphism
of complexes $\Hom_\Lambda(\boZ,g)$ at any chosen cohomological
degree $n\in\boZ$ produces
the nonisomorphism~\eqref{product-localization-not-commute}.

 We will see below that the class of homotopy injective complexes of
injective modules over arbitrary commutative rings is both
\emph{colocal} presuming termwise contraadjustedness (by
Example~\ref{homotopy-injectivity-of-injectives-colocal-example})
and \emph{antilocal} (see
Example~\ref{homotopy-injectivity-antilocal-example}).
\end{ex}

\Section{Colocality: Counterexample and First Examples}
\label{colocality-first-examples-secn}

 As in Section~\ref{locality-first-examples-secn}, we consider
a commutative ring $R$, a arbitrary element $s\in R$, and a finite
collection of elements $s_1$,~\dots, $s_d\in R$ generating
the unit ideal of~$R$.
 We also suppose given a class of commutative rings $\R$ that is
stable under localizations with respect to elements.

 In this context, we would like to define what it means for a system
of classes of modules or complexes $(\sL_R)_{R\in\R}$ to be
\emph{colocal}.
 The idea is to replace the localization functors $M\longmapsto
M[s^{-1}]$ with the colocalization functors $M\longmapsto
\Hom_R(R[s^{-1}],M)$ in the definitions of ascent and descent in
Section~\ref{locality-first-examples-secn}.
 The resulting definition of coascent is pretty straightforward, but
the following conterexample shows that one has to be careful with
the codescent.

\begin{ex} \label{naive-codescent-fails}
 Let $R=\boZ$ be the ring of integers, or alternatively $R=k[x]$
the polynomial ring in one variable over a field~$k$.
 Let $s_1$ and~$s_2\in R$ be two coprime noninvertible elements;
so $(s_1,s_2)=R$ but $(s_1)\ne R\ne(s_2)$.
 Then for the $R$\+module $M=R$ one has $\Hom_R(R[s_1^{-1}],M)=0$
and $\Hom_R(R[s_2^{-1}],M)=0$.
 Thus a na\"\i ve dualization of the descent condition from
Section~\ref{locality-first-examples-secn} would render the codescent
effectively impossible to satisfy.
 In particular, such na\"\i ve codescent would fail for the class of
injective (as well as cotorsion) modules.
\end{ex}

 In view of Example~\ref{naive-codescent-fails}, we suggest that one
should consider the codescent and colocality \emph{within} the class
of contraadjusted modules or termwise contraadjusted complexes.
 Given a commutative ring $R$, denote by $R\Modl^\cta$ the class of
all contraadjusted $R$\+modules (as defined in
Section~\ref{locality-first-examples-secn}).

 Since the full subcategory $R\Modl^\cta$ is closed under extensions
in the abelian $R\Modl$, it inherits an exact category structure from
the abelian exact structure of $R\Modl$.
 Here contraadjustedness is chosen as the weakest one of all colocal
properties.
 In our approach, it will satisfy the codescent \emph{by definition}.

 So, in the context of colocality, we consider various classes of
objects within the big ambient exact category associated with $R$,
which may be either the category of contraadjusted $R$\+modules
$\sK_R^\cta=R\Modl^\cta$ or the category of complexes of contraadjusted
$R$\+modules $\sK_R^\cta=\bC(R\Modl^\cta)$.
 Then we consider a system of classes of objects $\sL_R\subset
\sK_R^\cta$ defined for all rings $R\in\R$.

 A class/property of modules or complexes $\sL=
(\sL_R\subset\sK_R^\cta)_{R\in\R}$ is called \emph{colocal} if it
satisfies the following two conditions:
\begin{description}
\item[Coascent] For any ring $R\in\R$, element $s\in R$, and module
or complex $M\in\sL_R$, the module/complex $\Hom_R(R[s^{-1}],M)$
belongs to $\sL_{R[s^{-1}]}$.
\item[Codescent] Let $R\in\R$ be a ring and $s_1$,~\dots, $s_d\in R$
be a finite collection of elements generating the unit ideal in~$R$.
 Let $M\in\sK_R^\cta$ be a module or complex such that
$\Hom_R(R[s_j^{-1}],M)\in\sL_{R[s_j^{-1}]}$ for every $1\le j\le d$.
 Then $M\in\sL_R$.
\end{description}

 Furthermore, we will say that a class $\sL=(\sL_R)_{R\in\R}$ is
\emph{very colocal} if it is colocal and satisfies the direct image
condition from Section~\ref{locality-first-examples-secn}.

\begin{ex} \label{contraadjustedness-colocal-example}
 The contraadjustedness property of modules over commutative rings is
very colocal.
 The codescent, as defined above, is trivially satisfied; but one still
needs to check the coascent and the direct image condition.
 Essentially, these hold for the same reasons that were put forward
as an explanation of the ascent and direct image properties of very
flatness in the first paragraph of
Example~\ref{very-flatness-local-example}.
 This is a general observation formulated in the next
Proposition~\ref{left-right-ascent-coascent-direct-image-prop}.

 Moreover, the $S$\+module $\Hom_R(S,C)$ is contraadjusted for
any contraadjusted $R$\+module $C$ and any commutative ring
homomorphism $R\rarrow S$ such that the localization $S[s^{-1}]$ is
a very flat $R$\+module for all $s\in S$ \,\cite[Lemma~1.2.3(a)]{Pcosh}.
 The underlying $R$\+module of any contraadjusted $S$\+module is
contraadjusted for any commutative ring homomorphism $R\rarrow S$
\,\cite[Lemma~1.2.2(a)]{Pcosh}.

 We will see in Example~\ref{contraadjustedness-antilocal-example}
below that contraadjustedness is also an \emph{antilocal} property
of modules over commutative rings.
\end{ex}

 The following proposition is our first result about the behavior
of cotorsion pairs of classes of modules or complexes over commutative
rings with respect to localizations of the rings by their elements.
 In fact, the argument is quite general.

\begin{prop} \label{left-right-ascent-coascent-direct-image-prop}
 Let $(\sA_R\subset\sK_R)_{R\in\R}$ and $(\sB_R\subset\sK_R)_{R\in\R}$
be two systems of classes of modules or complexes over commutative
rings $R\in\R$.
 Assume that $(\sA_R,\sB_R)$ is a cotorsion pair in\/ $\sK_R$ for
every $R\in\R$.
 Then the class\/ $\sA$ satisfies the ascent and direct image conditions
if and only if\/ $\sB_R\subset\sK_R^\cta$ for every $R\in\R$ \emph{and}
the class\/ $\sB$ satisfies the coascent and direct image conditions.

 More generally, for any systems of classes $(\sA_R)_{R\in\R}$ and
$(\sB_R)_{R\in\R}$, \par
\textup{(a)} if a system of classes $(\sA_R\subset\sK_R)_{R\in\R}$
satisfies ascent, then the direct image condition holds for the system
of classes $(\sA_R^{\perp_1}\subset\sK_R)$; \par
\textup{(b)} if the direct image condition holds for a system of
classes $(\sA_R\subset\sK_R)_{R\in\R}$, then the system of classes
$(\sA_R^{\perp_1}\subset\sK_R)$ satisfies coascent; \par
\textup{(c)} if the direct image condition holds for a system of
classes $(\sB_R\subset\sK_R)_{R\in\R}$, then the system of classes
$({}^{\perp_1}\sB_R\subset\sK_R)_{R\in\R}$ satisfies ascent; \par
\textup{(d)} if a system of classes $(\sB_R\subset\sK_R)_{R\in\R}$
satisfies coascent and\/ $\sB_R\subset\sK_R^\cta$ for all $R\in\R$,
then the direct image condition holds for the system of classes
$({}^{\perp_1}\sB_R\subset\nobreak\sK_R)_{R\in\R}$.
\end{prop}

 Our proof of
Proposition~\ref{left-right-ascent-coascent-direct-image-prop}
is based on the following module-theoretic version of
the Ext\+adjunction lemma.
 Alternatively, it can be just as well based on
the category-theoretic Lemma~\ref{categorical-ext-adjunction-lemma}.

\begin{lem} \label{spectral-ext-adjunction-lemma}
 Put either\/ $\sK_R=R\Modl$ or\/ $\sK_R=\bC(R\Modl)$ for all
associative rings $R$, and let $R\rarrow S$ be a ring homomorphism.
 Then \par
\textup{(a)} for any modules/complexes $M\in\sK_R$ and $N\in\sK_S$,
and every integer $m\ge0$, there is a natural abelian group map
\begin{equation} \label{ext-tensor-comparison}
 \Ext^m_{\sK_S}(S\ot_RM,\>N)\lrarrow\Ext^m_{\sK_R}(M,N),
\end{equation}
which is an isomorphism whenever\/ $\Tor^R_i(S,M)=0$ for all\/
$1\le i\le m$ (if $M$ is a module) or\/ $\Tor^R_i(S,M^n)=0$
for all\/ $n\in\boZ$ and\/ $1\le i\le m$ (if $M$ is complex); \par
\textup{(b)} for any modules/complexes $M\in\sK_R$ and $N\in\sK_S$,
and every integer $m\ge0$, there is a natural abelian group map
\begin{equation} \label{ext-hom-comparison}
 \Ext^m_{\sK_S}(N,\Hom_R(S,M))\lrarrow\Ext^m_{\sK_R}(N,M),
\end{equation}
which is an isomorphism whenever\/ $\Ext_R^i(S,M)=0$ for all\/
$1\le i\le m$ (if $M$ is a module) or\/ $\Ext_R^i(S,M^n)=0$
for all\/ $n\in\boZ$ and\/ $1\le i\le m$ (if $M$ is a complex); \par
\textup{(c)} for any modules/complexes $M\in\sK_R$ and $N\in\sK_S$,
there is a natural exact sequence of abelian groups
\begin{multline} \label{ext-hom-inflation-restriction}
 0\lrarrow\Ext^1_{\sK_S}(N,\Hom_R(S,M))\lrarrow\Ext^1_{\sK_R}(N,M) \\
 \lrarrow\Hom_{\sK_S}(N,\Ext^1_R(S,M))\lrarrow
 \Ext^2_{\sK_S}(N,\Hom_R(S,M)),
\end{multline}
where, in case $M$ is a complex, the functors\/ $\Hom_R(S,{-})$ and\/
$\Ext^m_R(S,{-})$ are presumed to be applied to $M$ termwise.
 In particular, the map~\eqref{ext-hom-comparison} is always
injective for $m=1$.
 (Similarly, the map~\eqref{ext-tensor-comparison} is
always injective for $m=1$.)
\end{lem}

\begin{proof}
 A simple proof of the isomorphism claim in part~(a) in the case
of module categories can be found in~\cite[Lemma~4.2(a)]{PSl}, and
the isomorphism claim in part~(b) is provable in the dual way.
 The case of categories of complexes is dealt with similarly to
the case of module categories.

 All the assertions of part~(a) and the parenthesized assertion in
part~(c) can be obtained from the spectral sequence
$$
 E_2^{p,q}=\Ext^p_{\sK_S}(\Tor^R_q(S,M),N)
 \Longrightarrow \mathrm{gr}^p\Ext_{\sK_R}^{p+q}(M,N)=E_\infty^{p,q}
$$
with the differentials $d_r^{p,q}\:E_r^{p,q}\rarrow E_r^{p+r,q-r+1}$,
where the functor $\Tor^R_q(S,{-})$ is presumed to be applied to $M$
termwise if $M$ is a complex.

 Similarly, all the assertions of part~(b) and the nonparenthesized
assertions of part~(c) can be obtained from the spectral sequence
$$
  E_2^{p,q}=\Ext^p_{\sK_S}(N,\Ext_R^q(S,M))
 \Longrightarrow \mathrm{gr}^p\Ext_{\sK_R}^{p+q}(N,M)=E_\infty^{p,q}
$$
with the differentials $d_r^{p,q}\:E_r^{p,q}\rarrow E_r^{p+r,q-r+1}$.
\end{proof}

\begin{proof}[Proof of
Proposition~\ref{left-right-ascent-coascent-direct-image-prop}]
 In the context of a system of cotorsion pairs $(\sA_R,\sB_R)$ in
$\sK_R$, where $R$ ranges over the class $\R$, let us first show that
the direct image condition for the class $\sA$ implies the inclusion
$\sB_R\subset\sK_R^\cta$ for all $R\in\R$.
 Here we have to consider two cases, corresponding to whether we are
dealing with modules ($\sK_R=R\Modl$) or with complexes of modules
($\sK_R=\bC(R\Modl)$).

 In the case of modules, $\sK_R=R\Modl$, we notice that all projective
$R$\+modules belong to $\sA_R$ for all $R\in\R$, since $(\sA_R,\sB_R)$
is a cotorsion pair in\/~$\sK_R$.
 If the class $\sA$ satisfies the direct image condition, then it
follows that $R[s^{-1}]\in\sA_R$ for all $R\in\R$ and $s\in R$,
hence $\sB_R\subset\sK_R^\cta$.
 It also follows that all very flat $R$\+modules belong to~$\sA_R$.
 In this sense, one can say that the class of very flat modules is
the smallest left class $\sA$ for all systems of cotorsion pairs
$(\sA_R,\sB_R)_{R\in\R}$ in $\sK_R=R\Modl$ such that the direct image
condition is satisfied for the class~$\sA$.

 In the case of complexes, $\sK_R=\bC(R\Modl)$, the argument is
essentially the same.
 All the projective objects of the category $\sK_R$ belong to $\sA_R$,
since $(\sA_R,\sB_R)$ is a cotorsion pair in~$\sK_R$.
 In particular, the acyclic two-term complexes of free $R$\+modules
$\dotsb\rarrow0\rarrow R\overset=\rarrow R\rarrow0\rarrow\dotsb$ are
projective objects in $\bC(R\Modl)$ (by
Lemma~\ref{disk-complexes-lemma}), so they have to belong to~$\sA_R$.
 If the class $\sA_R$ satisfies the direct image condition, then it
follows that the two-term complexes
$\dotsb\rarrow0\rarrow R[s^{-1}]\overset=\rarrow R[s^{-1}]\rarrow0
\rarrow\dotsb$ belong to $\sA_R$ for all $R\in\R$ and $s\in R$.
 Applying Lemma~\ref{disk-complexes-lemma} again (in $\sE=R\Modl$),
one concludes that $\sB_R\subset\bC(R\Modl^\cta)=\sK_R^\cta$.

 It remains to prove the assertions~(a\+-d).
 Parts~(a) and~(c) follow from the natural isomorphism
\begin{equation} \label{ext-tensor-isomorphism}
 \Ext^m_{\sK_S}(S\ot_RM,\>N)\simeq\Ext^m_{\sK_R}(M,N),
\end{equation}
which holds for all commutative ring homomorphisms $R\rarrow S$
making $S$ a flat $R$\+module, all $R$\+modules/complexes $M$, all
$S$\+modules/complexes $N$, and all $m\ge0$ by~\cite[Lemma~4.1(a)]{PSl}
or Lemma~\ref{spectral-ext-adjunction-lemma}(a) or
Lemma~\ref{categorical-ext-adjunction-lemma}(a).
 In the situation at hand, one needs
to use~\eqref{ext-tensor-isomorphism} for $S=R[s^{-1}]$ and $m=1$.

 In part~(b), it is a bit easier to prove that the system of classes
$\sA_R^{\perp_1}\cap\sK_R^\cta$ satisfies the coascent (in the spirit
of the formulation of the coascent condition above, which was stated
for subclasses of $\sK_R^\cta$ only).
 Then both part~(b) restricted to $\sK_R^\cta$ and part~(d) follow
from the natural isomorphism
\begin{equation} \label{ext-hom-isomorphism}
 \Ext^m_{\sK_S}(N,\Hom_R(S,M))\simeq\Ext^m_{\sK_R}(N,M),
\end{equation}
which holds for all ring homomorphisms $R\rarrow S$, all
$R$\+modules or complexes $M$ such that $\Ext_R^i(S,M)=0$ for
$1\le i\le m$ or $\Ext_R^i(S,M^n)=0$ for all $n\in\boZ$ and
$1\le i\le m$, all $S$\+modules/complexes $N$, and all $m\ge0$.
 This is a dual version of~\cite[Lemma~4.2(a)]{PSl}, or
Lemma~\ref{spectral-ext-adjunction-lemma}(b).
 For $m=1$, one can also apply
Lemma~\ref{categorical-ext-adjunction-lemma}(e).
 In the situation at hand, one needs
to use~\eqref{ext-hom-isomorphism} for $S=R[s^{-1}]$ and $m=1$.

 A full proof of part~(b) is based on
Lemma~\ref{spectral-ext-adjunction-lemma}(c) for $S=R[s^{-1}]$,
which tells that $\Ext_{\sK_S}^1(N,\Hom_R(S,M))=0$ whenever
$\Ext_{\sK_R}^1(N,M)=0$.
 Alternatively, one can use
Lemma~\ref{categorical-ext-adjunction-lemma}(b) to the same effect.
\end{proof}

 The next lemma is the dual version of
Lemma~\ref{ascent+direct-image-imply-descent}.

\begin{lem} \label{coascent+direct-image-imply-codescent}
 Assume that a class\/ $\sL=(\sL_R\subset\sK_R^\cta)_{R\in\R}$
satisfies the coascent and direct image conditions.
 Assume further that, for every ring $R\in\R$, the full subcategory\/
$\sL_R$ is closed under finite direct sums and the cokernels of
admissible monomorphisms in\/ $\sK_R^\cta$ (equivalently, under
the cokernels of all monomorphisms in\/~$\sK_R$).
 Then the class\/ $\sL$ also satisfies codescent;
so, it is very colocal.
\end{lem}

\begin{proof}
 The full subcategory $\sK_R^\cta$ is closed under epimorphic images in
the abelian category $\sK_R$ by Lemma~\ref{generated-by-projdim1-lemma}.
 Hence, for a class of objects $\sL_R\subset\sK_R^\cta$ to be closed
under the cokernels of admissible monomorphisms in $\sK_R^\cta$ is
equivalent to it being closed under the cokernels of all monomorphisms
in~$\sK_R$.

 To prove the codescent under the assumptions of the lemma, consider
the \v Cech coresolution~\eqref{cech-coresolution} for the free
module $M=R$,
\begin{multline} \label{free-module-cech-coresolution}
 0\lrarrow R\lrarrow\bigoplus\nolimits_{i=1}^d R[s_i^{-1}]\lrarrow
 \bigoplus\nolimits_{1\le i<j\le d}R[s_i^{-1},s_j^{-1}] \lrarrow\dotsb
 \\ \lrarrow\bigoplus\nolimits_{1\le i_1<\dotsb<i_k\le d}
 R[s_{i_1}^{-1},\dotsc,s_{i_k}^{-1}]\lrarrow\dotsb\lrarrow
 R[s_1^{-1},\dotsc,s_d^{-1}]\lrarrow0.
\end{multline}
 Now \eqref{free-module-cech-coresolution}~is a finite exact sequence
of very flat $R$\+modules.
 As the class of very flat $R$\+modules is closed under the kernels of
epimorphisms in $R\Modl$ (see
Example~\ref{very-flatness-local-example}), moving by induction from
the rightmost end of~\eqref{free-module-cech-coresolution} to its
leftmost end one shows that all the modules of cocycles in this exact
sequence are very flat.

 Therefore, for any $M\in\sK_R^\cta$, we have a finite exact sequence
of $R$\+modules or complexes of $R$\+modules obtained by applying
the functor $\Hom_R({-},M)$ to~\eqref{free-module-cech-coresolution}
\begin{multline} \label{hom-cech-resolution}
 0\lrarrow\Hom_R(R[s_1^{-1},\dotsc,s_d^{-1}],M)\lrarrow\dotsb\lrarrow \\
 \bigoplus\nolimits_{1\le i<j\le d}\Hom_R(R[s_i^{-1},s_j^{-1}],M)
 \lrarrow\bigoplus\nolimits_{i=1}^d\Hom_R(R[s_i^{-1}],M) \\
 \lrarrow M\lrarrow0.
\end{multline}
 One can see that~\eqref{hom-cech-resolution} is an exact sequence in
the exact category $\sK_R^\cta$, as the $R$\+module $\Hom_R(F,C)$ is
contraadjusted for any very flat $R$\+module $F$ and contraadjusted
$R$\+module~$C$ \,\cite[Lemma~1.2.1(b)]{Pcosh}.

 The rest of the argument is dual to the proof of
Lemma~\ref{ascent+direct-image-imply-descent}.
 If $\Hom_R(R[s_j^{-1}],M)\in\sL_{R[s_j^{-1}]}$ for every $1\le j\le d$
and the class $\sL$ satisfies the coascent and direct image conditions,
then all the $R$\+modules or complexes of $R$\+modules
$\Hom_R(R[s_{i_1}^{-1},\dotsc,s_{i_k}^{-1}],M)$ appearing
in~\eqref{hom-cech-resolution}, except perhaps $M$ itself, belong
to~$\sL_R$.
 Hence all the terms of the exact sequence~\eqref{hom-cech-resolution},
except perhaps the rightmost one, belong to~$\sL_R$.
 Finally, using the assumption that the class $\sL_R$ is closed under
the cokernels of monomorphisms in $\sK_R$ and moving by induction
from the leftmost end of the sequence to its rightmost end, one
proves that $M\in\sL_R$.
\end{proof}

 In the course of the proof of
Lemma~\ref{coascent+direct-image-imply-codescent}, we have also
proved the following useful lemma.

\begin{lem} \label{colocalization-admissible-epimorphism-lemma}
 Let $R$ be a commutative ring and $s_1$,~\dots, $s_d\in R$ be
a collection of elements generating the unit ideal in~$R$.
 Then, for any module or complex $M\in\sK_R^\cta$, the natural map
of $R$\+modules/complexes
$$
 \bigoplus\nolimits_{j=1}^d\Hom_R(R[s_j^{-1}],M)\lrarrow M
$$
is an admissible epimorphism in\/ $\sK_R^\cta$ (i.~e., an epimorphism
in\/ $\sK_R$ whose kernel computed in\/ $\sK_R$ belongs
to~$\sK_R^\cta$).
\end{lem}

\begin{proof}
 This was explained in the discussion of the \v Cech
resolution~\eqref{hom-cech-resolution}.
\end{proof}

\begin{ex} \label{injectivity-colocal-example}
 Injectivity of modules over commutative rings is a very colocal
property (presuming contraadjustedness in the codescent, i.~e.,
within the class $\sK_R^\cta=R\Modl^\cta$, as per the definition above).
 Indeed, one can easily see that the $S$\+module $\Hom_R(S,I)$ is
injective for any injective $R$\+module $I$ and any ring homomorphism
$R\rarrow S$; so the coascent is satisfied.
 Furthermore, the $R$\+module $J$ is injective for any commutative ring
homomorphism $R\rarrow S$ making $S$ a flat $R$\+module and any
injective $S$\+module~$J$; so the direct image condition holds as well.

 One could also refer to
Proposition~\ref{left-right-ascent-coascent-direct-image-prop}
to the effect that the class of injective modules satisfies the coascent
and direct image conditions since the class of all modules satisfies
the ascent and direct image.
 As the class of injective $R$\+modules is closed under (finite
or infinite) products and cokernels of monomorphisms,
Lemma~\ref{coascent+direct-image-imply-codescent} tells that
the codescent of injectivity holds.

 We will see in Example~\ref{injectivity-strongly-antilocal-example}
below that injectivity is also a \emph{strongly antilocal} property
of modules over commutative rings.
\end{ex}

\begin{ex} \label{cotorsion-colocal-example}
 The class of cotorsion modules over commutative rings is very
colocal (presuming contraadjustedness in the codescent).
 Indeed, the $S$\+module $\Hom_R(S,C)$ is cotorsion for any cotorsion
$R$\+module $C$ and any commutative ring homomorphism $R\rarrow S$
making $S$ a flat $R$\+module~\cite[Lemma~1.3.5(a)]{Pcosh}.
 The underlying $R$\+module of any cotorsion $R$\+module is cotorsion
for any ring homomorphism $R\rarrow S$ \,\cite[Lemma~1.3.4(a)]{Pcosh}.

 One could also refer to
Proposition~\ref{left-right-ascent-coascent-direct-image-prop} and
Example~\ref{flatness-local-example}, to the effect that the class of
cotorsion modules satisfies the coascent and direct image conditions
since the class of flat modules satisfies the ascent and direct image.
 As the class of cotorsion $R$\+modules is closed under (finite
or infinite) products and cokernels of monomorphisms (by
Lemma~\ref{hereditary-lemma}(a)),
Lemma~\ref{coascent+direct-image-imply-codescent} tells that
the class of cotorsion modules satisfies codescent.

 We will see in Example~\ref{cotorsion-antilocal-example} below that
the class of cotorsion modules over commutative rings is also antilocal.
\end{ex}

\begin{ex} \label{homotopy-injectivity-of-injectives-colocal-example}
 The class of homotopy injective complexes of injective modules
over commutative rings is very colocal (presuming termwise
contraadjustedness in the codescent, i.~e., within the class
$\sK_R^\cta=\bC(R\Modl^\cta)$, as per the definition above).

 Indeed, it is easy to see that the complex of $S$\+modules
$\Hom_R(S,I^\bu)$ is homotopy injective for any homotopy injective
complex of $R$\+modules $I^\bu$ and any ring homomorphism $R\rarrow S$.
 Furthermore, the complex of $R$\+modules $J^\bu$ is homotopy
injective for any commutative ring homomorphism $R\rarrow S$ making
$S$ a flat $R$\+module and any homotopy injective complex of
$S$\+modules~$J^\bu$.
 Together with the similar properties of injective modules mentioned
in Example~\ref{injectivity-colocal-example}, this implies the coascent
and direct image conditions for homotopy injective complexes of
injective modules.

 Alternatively, one notice that, in view of
Lemmas~\ref{disk-complexes-lemma} and~\ref{ext-homotopy-hom-lemma},
the class of homotopy injective complexes of injective modules
is precisely the right $\Ext^1$\+orthogonal class to the class of
acyclic complexes in $\bC(R\Modl)$ (see, e.~g., \cite{EJX},
\cite[Theorem~2.3.13]{Hov-book}, \cite[Example~3.2]{Hov},
\cite[Proposition~1.3.5(2)]{Bec}, or~\cite[Theorem~8.4]{PS4}).
 Since the class of acyclic complexes satisfies the ascent and
direct image conditions (cf.\
Example~\ref{homotopy-injectivity-antilocal-example} below),
Proposition~\ref{left-right-ascent-coascent-direct-image-prop} tells
that the class of homotopy injective complexes of injective modules
satisfies the coascent and direct image.

 As the class of homotopy injective complexes of injective modules
is (obviously) closed under products and cokernels of monomorphisms,
Lemma~\ref{coascent+direct-image-imply-codescent} tells that
the codescent holds for this class.
\end{ex}

\Section{Local and Antilocal Classes in Cotorsion Pairs}
\label{local-antilocal-secn}

 As in the previous two sections, we suppose to have chosen a class of
commutative rings $\R$ that is stable under localizations with respect
to elements, and denote by $\sK_R=R\Modl$ or $\sK_R=\bC(R\Modl)$
the abelian category of modules or complexes of modules over
a commutative ring~$R$.

 A system of classes of modules or complexes
$(\sF_R\subset\sK_R)_{R\in\R}$ is said to be \emph{antilocal} if it
satisfies the following condition.
\begin{description}
\item[Antilocality] For any ring $R\in\R$ and any finite collection
of elements $s_1$,~\dots, $s_d\in R$ generating the unit ideal of $R$,
the class $\sF_R$ can be described as follows.
 A module or complex $M\in\sK_R$ belongs to $\sF_R$ \emph{if and only
if} $M$ is a direct summand of a module or complex $F\in\sK_R$
having a finite filtration by $R$\+submodules or subcomplexes of
$R$\+submodules $0=F_0\subset F_1\subset F_2\subset\dotsb\subset
F_{N-1}\subset F_N=F$ with the following property.
 For every $1\le i\le N$ there exists $1\le j\le d$ for which
the successive quotient $F_i/F_{i-1}$, viewed as an $R$\+module or
a complex of $R$\+modules, is obtained by restriction of scalars
from an $R[s_j^{-1}]$\+module or complex of $R[s_j^{-1}]$\+modules
belonging to~$\sF_{R[s_j^{-1}]}$.
\end{description}

 Notice that the ``if'' implication in the definition of antilocality
contains in itself the assertion that the restriction of scalars with
respect to the localization map $R\rarrow R[s_j^{-1}]$ takes
modules/complexes from $\sF_{R[s_j^{-1}]}$ to modules/complexes
from~$\sF_R$.
 Thus any antilocal class satisfies the direct image condition from
Section~\ref{locality-first-examples-secn}.

\begin{ex} \label{all-modules-not-antilocal}
 The class of all $R$\+modules (or all complexes of $R$\+modules) is
usually \emph{not} antilocal.
 As in Example~\ref{naive-codescent-fails}, let $R=\boZ$ be the ring
of integers, or alternatively $R=k[x]$ the ring of polynomials in one
variable over a field~$k$; and let $s_1$, $s_2\in R$ be two coprime
noninvertible elements.
 Then for any $R[s_j]$\+module $N$ (where $j=1$ or~$2$) one has
$\Hom_R(N,R)=0$.
 It follows that $\Hom_R(F,R)=0$ for any $R$\+module $F$ finitely
filtered by $R$\+modules coming from $R[s_j]$\+modules via
the restriction of scalars.
 Thus the free $R$\+module $M=R$ is \emph{not} a direct summand of
an $R$\+module admitting such a filtration.
\end{ex}

 Furthermore, let us say that a system of modules or complexes
$(\sD_R\subset\sK_R)_{R\in\R}$ is \emph{strongly antilocal} if
the following condition is satisfied.
\begin{description}
\item[Strong antilocality] For any ring $R\in\R$ and any finite
collection of elements $s_1$,~\dots, $s_d\in R$ generating the unit
ideal of $R$, the class $\sD_R$ can be described as follows.
 A module or complex $M\in\sK_R$ belongs to $\sD_R$ \emph{if and only
if} $M$ is a direct summand of a module or complex $D\in\sK_R$
isomorphic to a finite direct sum $D=\bigoplus_{j=1}^d D_j$, where
the $R$\+module or complex of $R$\+modules $D_j$ is obtained by
restriction of scalars from an $R[s_j^{-1}]$\+module or complex of
$R[s_j^{-1}]$\+modules belonging to~$\sD_{R[s_j^{-1}]}$.
\end{description}

 The following theorem is the first main result of this paper.
 In its formulation, it is presumed that the exact category structure
on $\sE_R$ is inherited from the abelian exact category structure
of~$\sK_R$.

\begin{thm} \label{locality-antilocality-theorem}
 Let $(\sE_R\subset\sK_R)_{R\in\R}$ be a system of classes of modules
or complexes such that, for every $R\in\R$, the class\/ $\sE_R$ is
closed under extensions and direct summands in\/~$\sK_R$.
 Assume further that the system of classes\/ $\sE$ is very local.
 Let $(\sA_R\subset\sE_R)_{R\in\R}$ and $(\sB_R\subset\sE_R)_{R\in\R}$
be two systems of classes such that, for every $R\in\R$, the pair
of classes $(\sA_R,\sB_R)$ is a hereditary complete cotorsion pair
in the exact category\/~$\sE_R$.
 Then the following implications hold:
\begin{enumerate}
\renewcommand{\theenumi}{\alph{enumi}}
\item if the class\/ $\sB$ is antilocal, then the class\/ $\sA$ is
local;
\item if the class\/ $\sA$ is very local, then the class\/ $\sB$ is
antilocal;
\item if the class\/ $\sA$ is very local, then the system of classes\/
$\sA\cap\sB=(\sA_R\cap\sB_R)_{R\in\R}$ is strongly antilocal.
\end{enumerate}
\end{thm}

\begin{proof}[Proof of Theorem~\ref{locality-antilocality-theorem}(a)]
 This part of the theorem is a formal consequence of
the $\Ext^1$\+adjunction properties of the functors of extension
and restriction of scalars, as in
Lemmas~\ref{categorical-ext-adjunction-lemma}
and~\ref{spectral-ext-adjunction-lemma}, and in the proof of
Proposition~\ref{left-right-ascent-coascent-direct-image-prop}.

 Since the class $\sB$ satisfies the direct image condition and
the class $\sE$ satisfies ascent, the class
$\sA=\sE\cap{}^{\perp_1}\sB\subset\sK$ satisfies ascent by
Proposition~\ref{left-right-ascent-coascent-direct-image-prop}(c).
 To prove descent for the class $\sA$, let $s_1$,~\dots, $s_d\in R$
be a collection of elements generating the unit ideal.
 Let $M\in\sK_R$ be a module or complex such that
$M[s_j^{-1}]\in\sA_{R[s_j^{-1}]}$ for all $1\le j\le d$.
 Since $\sA\subset\sE$ and the class $\sE$ satisfies descent, it follows
that $M\in\sE_R$.
 In order to show that $M\in\sA_R$, it remains to check that
$\Ext^1_{\sE_R}(M,B)=0$ for all $B\in\sB_R$.

 By antilocality, the $R$\+module or complex of $R$\+modules $B$ is
a direct summand of a module/complex finitely filtered by
modules/complexes obtained from the objects of $\sB_{R[s_j^{-1}]}$,
\,$1\le j\le d$, via the restriction of scalars.
 So it suffices to check that $\Ext^1_{\sK_R}(M,B_j)=0$ for any
module or complex $B_j\in\sB_{R[s_j^{-1}]}$, where $1\le j\le d$.
 Now, by formula~\eqref{ext-tensor-isomorphism} we have
$\Ext^1_{\sK_R}(M,B_j)\simeq
\Ext^1_{\sK_{R[s_j^{-1}]}}(M[s_j^{-1}],\>B_j)=0$, since
$M[s_j^{-1}]\in\sA_{R[s_j^{-1}]}$ and $(\sA_{R[s_j^{-1}]}$,
$\sB_{R[s_j^{-1}]})$ is a cotorsion pair in~$\sE_{R[s_j^{-1}]}$.
 Notice that it follows from the assumptions of the theorem that
the functors $\Ext^1$ in the exact categories $\sE_{R[s_j^{-1}]}$
and $\sK_{R[s_j^{-1}]}$ agree.
\end{proof}

 The proof of Theorem~\ref{locality-antilocality-theorem}(b) is based
on the following proposition, which is stated in a greater generality
of full subcategories $\sE_R\subset\sK_R$ inheriting an exact category
structure from the abelian exact category structure of the abelian
categories~$\sK_R$ (in the sense of the definition in
Section~\ref{preliminaries-cotorsion-pairs-secn}).

\begin{prop} \label{locality-antilocality-main-prop}
 Let $(\sE_R\subset\sK_R)_{R\in\R}$ be a system of classes of modules
or complexes such that, for every $R\in\R$, the full subcategory\/
$\sE_R$ inherits an exact category structure from the abelian exact
category structure of\/~$\sK_R$.
 Assume further that the class\/ $\sE$ is very local.
 Suppose given a system of classes $(\sA_R\subset\sE_R)_{R\in\R}$
such that the class\/ $\sA$ is also very local.
 Let $R\in\R$ be a ring and $s_1$,~\dots, $s_d\in R$ be a collection
of elements generating the unit ideal.
 Assume that, for every\/ $1\le j\le d$, there is a hereditary complete
cotorsion pair $(\sA_j,\sB_j)$ in the exact category\/ $\sE_j=
\sE_{R[s_j^{-1}]}$ with\/ $\sA_j=\sA_{R[s_j^{-1}]}$.

 Let\/ $\sB'_R$ be the following subclass of objects in\/~$\sE_R$.
 A module/complex $M\in\sE_R$ belongs to\/ $\sB'_R$ if and only if $M$
is a direct summand of a module/complex $B\in\sE_R$ admitting a finite
decreasing filtration $B=B^0\supset B^1\supset B^2\supset\dotsb\supset
B^{d+1}=0$ in the exact category~$\sE_R$ (with admissible monomorphisms
in\/ $\sE_R$ as inclusion maps) having the following properties.
 For every\/ $1\le j\le d$, the $R$\+module/complex
$B^j/B^{j+1}\in\sE_R$ is obtained by restriction of scalars from
an $R[s_j^{-1}]$\+module/complex belonging to the class\/
$\sB_j\subset\sE_j$.
 The $R$\+module/complex $B^0/B^1=B/B^1$ is a finite direct sum
$B^0/B^1\simeq\bigoplus_{j=1}^d B'_j$, where the $R$\+module/complex
$B'_j$ is obtained by restriction of scalars from
an $R[s_j^{-1}]$\+module/complex belonging to the class\/
$\sB_j\subset\sE_j$.

 Then $(\sA_R,\sB'_R)$ is a hereditary complete cotorsion pair
in\/~$\sE_R$.
\end{prop}

\begin{proof}
 We will prove the following properties:
\begin{itemize}
\item the class $\sA_R$ is closed under direct summands and
kernels of admissible epimorphisms in\/~$\sE_R$;
\item $\Ext^1_{\sE_R}(A,B)=0$ for all $A\in\sA_R$ and $B\in\sB'_R$;
\item every module/complex $E\in\sE_R$ admits a special precover
sequence~\eqref{spec-precover-sequence} in $\sE_R$ with $B'\in\sB'_R$
and $A\in\sA_R$;
\item every module/complex $E\in\sE_R$ admits a special preenvelope
sequence~\eqref{spec-preenvelope-sequence} in $\sE_R$ with
$B\in\sB'_R$ and $A'\in\sA_R$.
\end{itemize}
 Then it will follow by virtue of
Lemma~\ref{direct-summand-closure-is-complete-cotorsion} that
$(\sA_R,\sB'_R)$ is a cotorsion pair in\/~$\sE_R$, which is clearly
hereditary and complete.

 Indeed, the classes $\sA_j=\sA_{R[s_j^{-1}]}$ are closed under
direct summands and kernels of admissible epimorphisms in
$\sE_j=\sE_{R[s_j^{-1}]}$, since $(\sA_j,\sB_j)$ is a hereditary
cotorsion pair in\/~$\sE_j$.
 As both the systems of classes $\sA$ and $\sE$ are local by
assumption, it follows that the class $\sA_R$ is closed under direct
summands and kernels of admissible epimorphisms in~$\sE_R$.
 This uses the ascent for $\sE$ and the descent for~$\sA$.

 To prove the $\Ext^1$\+orthogonality, it suffices to check that
$\Ext^1_{\sE_R}(A,B_j)=0$ for all $A\in\sA_R$ and $B_j\in\sB_j$.
 Indeed, the $\Ext^1$\+adjunction isomorphism
$$
 \Ext^1_{\sE_R}(A,B_j)\simeq\Ext^1_{\sE_j}(A[s_j^{-1}],B_j)
$$
holds by Lemma~\ref{categorical-ext-adjunction-lemma}(a),
and it remains to use the ascent for $\sA$ and
the $\Ext^1$\+or\-thog\-o\-nal\-ity of the classes $\sA_j$ and $\sB_j$
in~$\sE_j$.

 The key step is the construction of the special precover sequences.
 We follow the construction of~\cite[Lemma~A.1]{EP}
and~\cite[Lemma~4.1.1 or~4.1.8]{Pcosh} specialized to the case of
an affine scheme~$X$.
 The idea to plug whole complexes (rather than just sheaves) into this
construction was used in~\cite[Lemma~6.3]{ES}.

 Let $E\in\sE_R$ be an object.
 We proceed by increasing induction on $0\le j\le d$, constructing
a sequence of admissible epimorphisms
$$
 E(d)\,\twoheadrightarrow\, E(d-1)\,\twoheadrightarrow\,\dotsb
 \,\twoheadrightarrow\,E(1)\,\twoheadrightarrow\,E(0)=E
$$
in~$\sE_R$ with the following properties.
 Firstly, the kernel $B_j$ of the admissible epimorphism $E(j)\rarrow
E(j-1)$, \ $1\le j\le d$, comes via the restriction of scalars from
a module/complex belonging to~$\sB_j$.
 Secondly, for every $1\le k\le j\le d$, the module/complex
$E(j)[s_k^{-1}]$ belongs to the class~$\sA_k$.

 Then, in particular, $E(d)[s_j^{-1}]\in\sA_j$ for all $1\le j\le d$,
and it follows that $E(d)\in\sA_R$ (since the class $\sA$ satisfies
descent).
 On the other hand, the kernel $B''$ of the composition
$E(d)\twoheadrightarrow E$ of the sequence of admissible epimorphisms
above is an $R$\+module or complex of $R$\+modules filtered by
the $R[s_j^{-1}]$\+modules or complexes of $R[s_j^{-1}]$\+modules
$B_j\in\sB_j$, so $B''\in\sB'_R$.

 The induction base, $j=0$, is dealt with by setting $E(0)=E$.
 Now assume that we have already constructed admissible
epimorphisms $E(j-1)\twoheadrightarrow E(j-2)\twoheadrightarrow\dotsb
\twoheadrightarrow E(1)\twoheadrightarrow E(0)$ with the desired
properties, for some $j\ge1$.
 Let us construct an admissible epimorphism $E(j)
\twoheadrightarrow E(j-1)$.

 We have $E(j-1)[s_j^{-1}]\in\sE_j$, since $E(j-1)\in\sE_R$ and
the class $\sE$ satisfies ascent.
 Pick a special precover sequence
\begin{equation} \label{building-block-precover}
 0\lrarrow B_j\lrarrow A_j\lrarrow E(j-1)[s_j^{-1}]\lrarrow0
\end{equation}
with respect to the cotorsion pair $(\sA_j,\sB_j)$ in the exact
category~$\sE_j$.
 Taking the restriction of scalars, one can
view~\eqref{building-block-precover} as a short exact sequence in
$\sE_R$ with $A_j\in\sA_R$, since the classes $\sE$ and $\sA$
satisfy the direct image condition.
 Consider the pullback diagram in the exact category~$\sE_R$
\begin{equation} \label{induction-step-pullback-diagram}
\begin{gathered}
 \xymatrix{
  0 \ar[r] & B_j \ar[r] & A_j \ar[r] & E(j-1)[s_j^{-1}] \ar[r]
  & 0 \\
  0 \ar[r] & B_j \ar[r] \ar@{=}[u] & E(j) \ar[r] \ar[u]
  & E(j-1) \ar[r] \ar[u] & 0
 }
\end{gathered}
\end{equation}
where $E(j-1)\rarrow E(j-1)[s_j^{-1}]$ is the localization map.
 The object $E(j)\in\sE_R$ is constructed as the pullback object in
this diagram.

 To show that $E(j)[s_j^{-1}]\in\sA_j$, it suffices to notice that
the localization functor $M\longmapsto M[s_j^{-1}]$ is exact as
a functor $\sK_R\rarrow\sK_{R[s_j^{-1}]}$, and consequently, also
as a functor $\sE_R\rarrow\sE_{R[s_j^{-1}]}=\sE_j$; so it
preserves pullbacks of short exact sequences.
 Therefore, the induced morphism $E(j)[s_j^{-1}]\rarrow A_j[s_j^{-1}]
=A_j$ is an isomorphism.

 To prove that $E(j)[s_k^{-1}]\in\sA_k$ for $1\le k<j$, we use
the induction assumption telling that $E(j-1)[s_k^{-1}]\in\sA_k$.
 Applying to~\eqref{induction-step-pullback-diagram} the localization
functor $M\longmapsto M[s_k^{-1}]$, we obtain a pullback diagram
in~$\sE_k$
\begin{equation} \label{induction-step-pullback-diagram-localized}
\begin{gathered}
 \xymatrix{
  0 \ar[r] & B_j[s_k^{-1}] \ar[r] & A_j[s_k^{-1}] \ar[r]
  & E(j-1)[s_k^{-1},s_j^{-1}] \ar[r] & 0 \\
  0 \ar[r] & B_j[s_k^{-1}] \ar[r] \ar@{=}[u]
  & E(j)[s_k^{-1}] \ar[r] \ar[u] & E(j-1)[s_k^{-1}] \ar[r] \ar[u] & 0
 }
\end{gathered}
\end{equation}

 Now we have $E(j-1)[s_k^{-1},s_j^{-1}]\in\sA_k$, since
$E(j-1)[s_k^{-1}]\in\sA_k$ and the class $\sA$ satisfies the ascent
and direct image conditions.
 Similarly $A_j[s_k^{-1}]\in\sA_k$, since $A_j\in\sA_j$.
 The cotorsion pair $(\sA_k,\sB_k)$ in the exact category $\sE_k$ is
hereditary by assumption, so it follows from the short exact sequence
in the upper line of~\eqref{induction-step-pullback-diagram-localized}
that $B_j[s_k^{-1}]\in\sA_k$.
 Finally, in the lower line of the diagram we have
$B_j[s_k^{-1}]\in\sA_k$ and $E(j-1)[s_k^{-1}]\in\sA_k$, hence
$E(j)[s_k^{-1}]\in\sA_k$, as the class $\sA_k$ is closed
under extensions in~$\sE_k$.

 We have constructed the special precover sequence
\begin{equation} \label{special-precover-sequence-constructed}
 0\lrarrow B''\lrarrow E(d)\lrarrow E\lrarrow0
\end{equation}
in $\sE_R$ with $B''\in\sB'_R$ and $E(d)\in\sA_R$.
 In order to finish the proof of the proposition, it remains to
produce the special preenvelope sequences.
 For this purpose, we use the construction from the proof of
the Salce lemma (Lemma~\ref{salce-lemma}).

 Let $F\in\sE_R$ be an object.
 We start with choosing special preenvelopes
$$
 0\lrarrow F[s_k^{-1}]\lrarrow B'_k\lrarrow A'_k\lrarrow0
$$
with respect to the cotorsion pairs $(\sA_k,\sB_k)$ in the exact
categories $\sE_k$ for all $1\le k\le d$.
 Now $F\rarrow\bigoplus_{k=1}^d F[s_k^{-1}]$ is a monomorphism in
$\sK_R$, and consequently the composition
$F\rarrow\bigoplus_{k=1}^d F[s_k^{-1}]\rarrow\bigoplus_{k=1}^d B'_k$
is a monomorphism in $\sK_R$ as well.
 So we have a short exact sequence
\begin{equation} \label{cogeneration-sequence-in-salce-lemma}
 0\lrarrow F\lrarrow\bigoplus\nolimits_{k=1}^d B'_k\lrarrow E
 \lrarrow0
\end{equation}
in the abelian category~$\sK_R$.
 In order to show that~\eqref{cogeneration-sequence-in-salce-lemma} is
a short exact sequence in the exact category $\sE_R$, we need to prove
that $E\in\sE_R$.

 Since the class $\sE$ satisfies descent by assumption, it suffices to
check that $E[s_j^{-1}]\in\sE_j$ for all $1\le j\le d$.
 The admissible monomorphism $F[s_j^{-1}]\rarrow B'_j$ in $\sE_j$
factorizes into the composition
$$
 F[s_j^{-1}]\lrarrow\bigoplus\nolimits_{k=1}^d B'_k[s_j^{-1}]
 \lrarrow B'_j,
$$
where $B'_k[s_j^{-1}]\in\sE_j$ since $B'_k\in\sB_k\subset\sE_k$
and the class $\sE$ satisfies the ascent and direct image conditions.
 So the morphism $\bigoplus_{k=1}^d B'_k[s_j^{-1}]\rarrow B'_j$ is
a (split) admissible epimorphism in~$\sE_j$.
 By K\"unzer's axiom~\cite[Exercise~3.11(i)]{Bueh}, it follows that
$F[s_j^{-1}]\rarrow\bigoplus_{k=1}^d B'_k[s_j^{-1}]$ is an admissible
monomorphism in~$\sE_j$.
 As $E[s_j^{-1}]$ is the cokernel of the latter morphism in
$\sK_{R[s_j^{-1}]}$, we can conclude that $E[s_j^{-1}]\in\sE_j$.

 We have shown that~\eqref{cogeneration-sequence-in-salce-lemma} is
an (admissible) short exact sequence in~$\sE_R$.
 Now we can proceed with the construction from the proof of the Salce
lemma.
 Applying the construction of the special precover sequence spelled
out above to the object $E\in\sE_R$, we obtain a short exact
sequence~\eqref{special-precover-sequence-constructed}.
 Consider the pullback diagram
of~\eqref{cogeneration-sequence-in-salce-lemma}
and~\eqref{special-precover-sequence-constructed}
in the exact category~$\sE_R$
\begin{equation} \label{salce-pullback-diagram}
\begin{gathered}
 \xymatrix{
  & B'' \ar@{=}[r] \ar@{>->}[d] & B'' \ar@{>->}[d] \\
  F \ar@{>->}[r] \ar@{=}[d] & B \ar@{->>}[r] \ar@{->>}[d]
  & E(d) \ar@{->>}[d] \\
  F \ar@{>->}[r] & B' \ar@{->>}[r] & E
 }
\end{gathered}
\end{equation}
where $B'=\bigoplus\nolimits_{k=1}^d B'_k$ and $B$ is the pullback
of the pair of admissible epimorphisms $B'\rarrow E$ and $E(d)
\rarrow E$ in~$\sE_R$.
 It remains to say that the middle horizontal short exact sequence
$0\rarrow F\rarrow B\rarrow E(d)\rarrow0$
in~\eqref{salce-pullback-diagram} is the desired special preenvelope
sequence for the object $F\in\sE_R$, where one has $B\in\sB'_R$ in
view of the middle vertical short exact sequence
$0\rarrow B''\rarrow B\rarrow B'\rarrow0$.
\end{proof}

 Now we are done with the proof of
Proposition~\ref{locality-antilocality-main-prop},
and so we can prove part~(b) of the theorem.

\begin{proof}[Proof of Theorem~\ref{locality-antilocality-theorem}(b)]
 By assumptions, both the classes $\sE$ and $\sA$ are very local; and
the full subcategory $\sE_R$ in closed under extensions in $\sK_R$ for
every $R\in\R$, so it inherits an exact category structure.
 Hence Proposition~\ref{locality-antilocality-main-prop} is applicable.

 Let $R\in\R$ be a ring and $s_1$,~\dots, $s_d\in R$ be a collection of
elements generating the unit ideal.
 From these data, the proposition produces a cotorsion pair
$(\sA_R,\sB'_R)$ in $\sE_R$, while in the assumptions of the theorem
we are given a cotorsion pair $(\sA_R,\sB_R)$ in~$\sE_R$.
 It follows that
\begin{itemize}
\item $\sB_R=\sB'_R$;
\item the class $\sB'_R$ is closed under extensions in $\sE_R$,
and consequently also in~$\sK_R$.
\end{itemize}
 These observations prove that the class $\sB$ is antilocal.

 In addition, they provide a precise form of the finite filtration
appearing in the antilocality condition for $\sB$, with a specific
bound on the length of such filtration.
 One can say that the length of the filtration certainly
does not exceed~$2d$.
\end{proof}

\begin{proof}[Proof of Theorem~\ref{locality-antilocality-theorem}(c)]
 Let $R\in\R$ be a ring and $s_1$,~\dots, $s_d\in R$ be a collection of
elements generating the unit ideal.
 We keep the notation $\sA_j=\sA_{R[s_j^{-1}]}$, \
$\sB_j=\sA_{R[s_j^{-1}]}$, and $\sE_j=\sE_{R[s_j^{-1}]}$ for
all $1\le j\le d$.
 Suppose given a module/complex $D\in\sA_R\cap\sB_R$.
 Then $D[s_j^{-1}]\in\sA_j$ for every $1\le j\le d$, since 
the class $\sA$ satisfies ascent.
 For every $1\le k\le d$, pick a special preenvelope sequence
$$
 0\lrarrow D[s_k^{-1}]\lrarrow B_k\lrarrow A_k\lrarrow0
$$
with respect to the cotorsion pair $(\sA_k,\sB_k)$ in the exact
category~$\sE_k$; so $B_k\in\sB_k$ and $A_k\in\sA_k$.
 Then we have $D[s_k^{-1}]\in\sA_k$ and $A_k\in\sA_k$, hence
$B_k\in\sA_k\cap\sB_k$.

 The composition of the natural monomorphism $D\rarrow
\bigoplus_{k=1}^d D[s_k^{-1}]$ in $\sK_R$ with the direct sum
of the monomorphisms $D[s_k^{-1}]\rarrow B_k$ (viewed as morphisms
in $\sK_R$ via the restriction of scalars) provides a monomorphism
$D\rarrow\bigoplus_{k=1}^d B_k$ in~$\sK_R$.
 Hence we have a short exact sequence
\begin{equation} \label{kernel-special-preenvelope}
 0\lrarrow D\lrarrow\bigoplus\nolimits_{k=1}^d B_k\lrarrow A\lrarrow0
\end{equation}
in the abelian category~$\sK_R$.
 We have already seen in the construction of the special preenvelope
in the proof of Proposition~\ref{locality-antilocality-main-prop}
that $A\in\sE_R$.
 Let us show that in the situation at hand $A\in\sA_R$.

 The argument is the same.
 Since the class $\sA$ satisfies descent by assumption, it suffices
to check that $A[s_j^{-1}]\in\sA_j$ for all $1\le j\le d$.
 As a full subcategory closed under extensions in $\sE_j$, the additive
category $\sA_j$ inherits an exact category structure.
 The morphism $D[s_j^{-1}]\rarrow B_j$ is an admissible monomorphism
in $\sA_j$, since it is an admissible monomorphism in $\sE_j$ between
two objects from $\sA_j$ with the cokernel $A_j\in\sA_j$.
 On the other hand, this morphism factorizes into the composition
$$
 D[s_j^{-1}]\lrarrow\bigoplus\nolimits_{k=1}^d B_k[s_j^{-1}]
 \lrarrow B_j,
$$
where $B_k[s_j^{-1}]\in\sA_j$ since $B_k\in\sA_k$ and the class $\sA$
satisfies the ascent and direct image conditions.
 Hence the morphism $\bigoplus_{k=1}^d B_k[s_j^{-1}]\rarrow B_j$ is
a (split) admissible epimorphism in~$\sA_j$.
 So~\cite[Exercise~3.11(i)]{Bueh} tells that $D[s_j^{-1}]\rarrow
\bigoplus_{k=1}^d B[s_j^{-1}]$ is an admissible monomorphism in~$\sA_j$.
 As $A[s_j^{-1}]$ is the cokernel of this morphism in
$\sK_{R[s_j^{-1}]}$, we can conclude that $A[s_j^{-1}]\in\sA_j$
and $A\in\sA_R$.

 Instead of referring to K\"unzer's axiom, one could spell out its
proof in the situation at hand.
 Let us do it, for the reader's convenience.
 Consider the pullback diagram in the exact category $\sE_j$, or even
in the abelian category~$\sK_{R[s_j^{-1}]}$
$$
 \xymatrix{
  & B''[s_j^{-1}] \ar@{=}[r] \ar@{>->}[d]
  & B''[s_j^{-1}] \ar@{>->}[d] \\
  D[s_j^{-1}] \ar@{>->}[r] \ar@{=}[d]
  & B'[s_j^{-1}] \ar@{->>}[r] \ar@{->>}[d]
  & A[s_j^{-1}] \ar@{->>}[d] \\
  D[s_j^{-1}] \ar@{>->}[r] & B_j \ar@{->>}[r] & A_j
 }
$$
where $B'=\bigoplus_{k=1}^d B_k$ and $B''=\bigoplus_{k\ne j}B_k$.
 We have $A_j\in\sA_j$ and $B_k[s_j^{-1}]\in\sA_j$ for all
$1\le k\le d$, so it follows from the rightmost vertical short exact
sequence $0\rarrow B''[s_j^{-1}]\rarrow A[s_j^{-1}]\rarrow A_j\rarrow0$
that $A[s_j^{-1}]\in\sA_j$.
 Indeed, the class $\sA_j$ is closed under extensions in $\sE_j$ and
the class $\sE_j$ is closed under extensions in $\sK_{R[s_j^{-1}]}$
by the assumptions of the theorem.

 Finally, we use the observation that
$\Ext^1_{\sE_R}(A,D)=0$, since $A\in\sA_R$ and $D\in\sB_R$.
 Thus it follows from the short exact
sequence~\eqref{kernel-special-preenvelope} that $D$ is a direct
summand in $\bigoplus_{k=1}^d B_k$.
 It remains to recall that $B_k\in\sA_k\cap\sB_k$.
\end{proof}

\begin{rem}
 Let us point out that the converse implication to
Theorem~\ref{locality-antilocality-theorem}(a) is \emph{not} true,
or in other words, the very locality assumption in part~(b) of
the theorem \emph{cannot} be replaced with the locality assumption.

 Indeed, let $\R$ be the class of all commutative rings and
$\sE_R=\sK_R=R\Modl$ for all $R\in\R$.
 Consider the system of cotorsion pairs $(\sA_R,\sB_R)$ in $R\Modl$,
where $\sA_R$ is the class of projective $R$\+modules and $\sB_R$ is
the class of all $R$\+modules.
 Then $\sA$ is a local class, as projectivity of modules over
commutative rings is a local property by the Raynaud--Gruson theorem
(see Example~\ref{projectivity-local-example}).
 However, the class $\sB$ of all modules is \emph{not} antilocal,
as we have seen in Example~\ref{all-modules-not-antilocal}.
\end{rem}

\begin{rem} \label{relative-antilocality-remark}
 Similarly to the formulation of
Proposition~\ref{locality-antilocality-main-prop}, one could relax
the assumptions of Theorem~\ref{locality-antilocality-theorem} by
requiring the full subcategory $\sE_R$ to inherit an exact category
structure from the abelian exact structure of $\sK_R$, but not
necessarily to be closed under extensions or direct summands
in~$\sK_R$.
 All the arguments in the proof can be made to work in this context.
 The only difference is that then one would have to speak about
the class $\sB$ being antilocal \emph{within} $\sE$ rather than
in the whole system of abelian categories $\sK$, in the sense that
the finitely iterated extensions and direct summands in the definition
of antilocality would have to be taken in $\sE_R$ and not in~$\sK_R$.
 The point is that our definition of antilocality implies that any
antilocal class is closed under extensions and direct summands
in~$\sK$.
 Antilocality within $\sE$ would mean that the class is only closed
under extensions and direct summands in~$\sE$.
\end{rem}

\Section{Antilocal and Colocal Classes in Cotorsion Pairs}
\label{antilocal-colocal-secn}

 As in the previous sections, we suppose to have chosen a class of
commutative rings $\R$ such that $R\in\R$ and $s\in R$ implies
$R[s^{-1}]\in\R$.
 As in Section~\ref{colocality-first-examples-secn}, we consider
classes of \emph{contraadjusted modules} $\sE_R\subset\sK_R^\cta=
R\Modl^\cta$ or classes of \emph{complexes of contraadjusted
modules} $\sE_R\subset\sK_R^\cta=\bC(R\Modl^\cta)$.

 Recall that the full subcategory $\sK_R^\cta$ is closed under
extensions (as well as products and quotients) in the abelian
category $\sK_R$, so it inherits an exact category structure from
the abelian exact structure of~$\sK_R$.
 Under suitable assumptions, we will endow full subcategories
$\sE_R\subset\sK_R^\cta$ with the inherited exact category structures.

 The definitions of a \emph{colocal class} and a \emph{very colocal
class} of modules or complexes were given in
Section~\ref{colocality-first-examples-secn}, while the definition of
an \emph{antilocal class} can be found in
Section~\ref{local-antilocal-secn}.
 Let us emphasize once again that our definition of colocality (or
more specifically, the codescent) presumes contraadjustedness of
modules or termwise contraadjustedness of complexes, i.~e., it is
always codescent \emph{within the class\/~$\sK_R^\cta$}.

 The following theorem is the dual version of
Theorem~\ref{locality-antilocality-theorem} and the second main result
of this paper.

\begin{thm} \label{antilocality-colocality-theorem}
 Let $(\sE_R\subset\sK_R^\cta)_{R\in\R}$ be a system of classes of
contraadjusted modules or complexes of contraadjusted modules such that,
for every $R\in\R$, the class\/ $\sE_R$ is closed under extensions and
direct summands in\/~$\sK_R^\cta$.
 Assume further that the system of classes\/ $\sE$ is very colocal.
 Let $(\sA_R\subset\sE_R)_{R\in\R}$ and $(\sB_R\subset\sE_R)_{R\in\R}$
be two systems of classes such that, for every $R\in\R$, the pair of
classes $(\sA_R,\sB_R)$ is a hereditary complete cotorsion pair in
the exact category\/~$\sE_R$.
 Then the following implications hold:
\begin{enumerate}
\renewcommand{\theenumi}{\alph{enumi}}
\item if the class\/ $\sA$ is antilocal, then the class\/ $\sB$ is
colocal;
\item if the class\/ $\sB$ is very colocal, then the class\/ $\sA$ is
antilocal;
\item if the class\/ $\sB$ is very colocal, then the system of classes\/
$\sA\cap\sB=(\sA_R\cap\sB_R)_{R\in\R}$ is strongly antilocal.
\end{enumerate}
\end{thm}

\begin{proof}[Proof of
Theorem~\ref{antilocality-colocality-theorem}(a)]
 Once again, this part of the theorem is a formal consequence of
the $\Ext^1$\+adjunction properties of the functors of restriction
and coextension of scalars, as in
Lemmas~\ref{categorical-ext-adjunction-lemma}
and~\ref{spectral-ext-adjunction-lemma}, and in the proof of
Proposition~\ref{left-right-ascent-coascent-direct-image-prop}.

 Since the class $\sA$ satisfies the direct image condition and
the class $\sE$ satisfies coascent, the class
$\sB=\sA^{\perp_1}\cap\sE$ satisfies coascent by
Proposition~\ref{left-right-ascent-coascent-direct-image-prop}(b).
 To prove codescent for the class $\sB$, let $s_1$,~\dots, $s_d\in R$
be a collection of elements generating the unit ideal.
 Let $M\in\sK_R^\cta$ be a module or complex such that
$\Hom_R(R[s_j^{-1}],M)\in\sB_{R[s_j^{-1}]}$ for all $1\le j\le d$.
 Since $\sB\subset\sE$ and the class $\sE$ satisfies codescent, it
follows that $M\in\sE_R$.
 In order to show that $M\in\sB_R$, it remains to check that
$\Ext^1_{\sE_R}(A,M)=0$ for all $A\in\sA_R$.

 By antilocality, the $R$\+module or complex of $R$\+modules $A$ is
a direct summand of a module/complex finitely filtered by
modules/complexes obtained from the objects of $\sA_{R[s_j^{-1}]}$,
\,$1\le j\le d$, via the restriction of scalars.
 So it suffices to check that $\Ext^1_{\sK_R}(A_j,M)=0$ for any
module or complex $A_j\in\sA_{R[s_j^{-1}]}$, where $1\le j\le d$.
 Now, by formula~\eqref{ext-hom-isomorphism} we have
$\Ext^1_{\sK_R}(A_j,M)\simeq
\Ext^1_{\sK_{R[s_j^{-1}]}}(A_j,\Hom_R(R[s_j^{-1}],M))=0$, since
$\Hom_R(R[s_j^{-1}],M)\in\sB_{R[s_j^{-1}]}$ and $(\sA_{R[s_j^{-1}]}$,
$\sB_{R[s_j^{-1}]})$ is a cotorsion pair in~$\sE_{R[s_j^{-1}]}$.
  Notice that it follows from the assumptions of the theorem that
the functors $\Ext^1$ in the exact categories $\sE_{R[s_j^{-1}]}$,
\,$\sK_{R[s_j^{-1}]}^\cta$, and $\sK_{R[s_j^{-1}]}$ agree.
\end{proof}

 The proof of Theorem~\ref{antilocality-colocality-theorem}(b) is based
on the following proposition, which is stated in a greater generality
of full subcategories $\sE_R\subset\sK_R^\cta$ inheriting an exact
category structure from $\sK_R^\cta$ (or equivalently, from~$\sK_R$),
as defined in Section~\ref{preliminaries-cotorsion-pairs-secn}.

\begin{prop} \label{antilocality-colocality-main-prop}
 Let $(\sE_R\subset\sK_R^\cta)_{R\in\R}$ be a system of classes of
contraadjusted modules or complexes of contraadjusted modules such that,
for every $R\in\R$, the full subcategory\/ $\sE_R$ inherits an exact
category structure from\/~$\sK_R^\cta$.
 Assume further that the class\/ $\sE$ is very colocal.
 Suppose given a system of classes $(\sB_R\subset\sE_R)_{R\in\R}$ such
that the class\/ $\sB$ is also very colocal.
 Let $R\in\R$ be a ring and $s_1$,~\dots, $s_d\in R$ be a collection of
elements generating the unit ideal.
 Assume that, for every\/ $1\le j\le d$, there is a hereditary complete
cotorsion pair $(\sA_j,\sB_j)$ in the exact category\/
$\sE_j=\sE_{R[s_j^{-1}]}$ with\/ $\sB_j=\sB_{R[s_j^{-1}]}$.

 Let\/ $\sA'_R$ be the following subclass of objects in\/~$\sE_R$.
 A module/complex $M\in\sE_R$ belongs to\/ $\sA'_R$ if and only if $M$
is a direct summand of a module/complex $A\in\sE_R$ admitting a finite
increasing filtration $0=A_{-1}\subset A_0\subset A_1\subset A_2
\subset\dotsb\subset A_d=A$ in the exact category\/ $\sE_R$ (with
admissible monomorphisms in\/ $\sE_R$ as inclusion maps) having
the following properties.
 For every $1\le j\le d$, the $R$\+module/complex $A_j/A_{j-1}\in\sE_R$
is obtained by restriction of scalars from
an $R[s_j^{-1}]$\+module/complex belonging to the class\/
$\sA_j\subset\sE_j$.
 The $R$\+module/complex $A_0$ is a finite direct sum $A_0=
\bigoplus_{j=1}^d A'_j$, where the $R$\+module/complex $A'_j$ is
obtained by restriction of scalars from an $R[s_j^{-1}]$\+module/complex
belonging to the class\/ $\sA_j\subset\sE_j$.

 Then $(\sA'_R,\sB_R)$ is a hereditary complete cotorsion pair
in\/~$\sE_R$.
\end{prop}

\begin{proof}
 We will prove the following properties:
\begin{itemize}
\item the class $\sB_R$ is closed under direct summands and cokernels
of admissible monomorphisms in~$\sE_R$;
\item $\Ext^1_{\sE_R}(A,B)=0$ for all $A\in\sA'_R$ and $B\in\sB_R$;
\item every module/complex $E\in\sE_R$ admits a special preenvelope
sequence~\eqref{spec-preenvelope-sequence} in $\sE_R$ with
$B\in\sB_R$ and $A'\in\sA'_R$;
\item every module/complex $E\in\sE_R$ admits a special precover
sequence~\eqref{spec-precover-sequence} in $\sE_R$ with
$B'\in\sB_R$ and $A\in\sA'_R$.
\end{itemize}
 Then it will follow by
Lemma~\ref{direct-summand-closure-is-complete-cotorsion} that
$(\sA'_R,\sB_R)$ is a cotorsion pair in $\sE_R$, which is clearly
hereditary and complete.

 Indeed, the classes $\sB_j=\sB_{R[s_j^{-1}]}$ are closed under direct
summands and cokernels of admissible monomorphisms in 
$\sE_j=\sE_{R[s_j^{-1}]}$, since $(\sA_j,\sB_j)$ is a hereditary
cotorsion pair in~$\sE_j$.
 As both the systems of classes are colocal by assumption and
the colocalization functor
\begin{equation} \label{colocalization-functor}
 \Hom_R(R[s_j^{-1}],{-})\:\sE_R\lrarrow\sE_j
\end{equation}
is exact (since it is exact as a functor $\Hom_R(R[s_j^{-1}],{-})\:
\sK_R^\cta\rarrow\sK_{R[s_j^{-1}]}^\cta$), it follows that the class
$\sB_R$ is closed under direct summands and cokernels of admissible
monomorphisms in~$\sE_R$.
 This uses the coascent for $\sE$ and the codescent for $\sB$,
together with the assumption that $\sE_R\subset\sK_R^\cta$.

 To prove the $\Ext^1$\+orthogonality, it suffices to check that
$\Ext_{\sE_R}^1(A_j,B)=0$ for all $A_j\in\sA_j$ and $B\in\sB_R$.
 Indeed, the $\Ext^1$\+adjunction isomorphism
$$
 \Ext_{\sE_R}^1(A_j,B)\simeq
 \Ext_{\sE_j}^1(A_j,\Hom_R(R[s_j^{-1}],B))
$$
holds by Lemma~\ref{categorical-ext-adjunction-lemma}(a), and it
remains to use the coascent for $\sB$ and
the $\Ext^1$\+or\-thog\-o\-nal\-ity of the classes $\sA_j$ and
$\sB_j$ in~$\sE_j$.

 The key step is the construction of the special preenvelope sequences.
 We follow the construction of~\cite[Lemma~4.2.2 or~4.3.1]{Pcosh}
specialized to the case of an affine scheme~$X$.
 Let us \emph{warn} the reader that what we call \emph{colocal} in
this paper is called ``local'' in~\cite{Pcosh}, while what we call
\emph{antilocal} in this paper is called ``colocal'' in~\cite{Pcosh}.
 The ``colocally projective'' contraherent cosheaves of
the terminology of~\cite{Pcosh} might be simply called \emph{antilocal}
in the terminology of the present paper (see the terminological
discussion in~\cite[Section~5.2]{Pphil}).

 Let $E\in\sE_R$ be an object.
 We proceed by increasing induction on $0\le j\le d$, constructing
a sequence of admissible monomorphisms
$$
 E=E(0)\,\rightarrowtail\,E(1)\,\rightarrowtail\,\dotsb
 \,\rightarrowtail\,E(d-1)\,\rightarrowtail\,E(d)
$$
in $\sE_R$ with the following properties.
 Firstly, the cokernel $A_j$ of the admissible monomorphism
$E(j-1)\rarrow E(j)$, \ $1\le j\le d$, comes via the restriction
of scalars from a module/complex belonging to~$\sA_j$.
 Secondly, for every $1\le k\le j\le d$, the module/complex
$\Hom_R(R[s_k^{-1}],E(j))$ belongs to the class~$\sB_k$.

 Then, in particular, $\Hom_R(R[s_j^{-1}],E(d))\in\sB_j$ for all
$1\le j\le d$, and it follows that $E(d)\in\sB_R$ (since the class
$\sB$ satisfies codescent).
 On the other hand, the cokernel $A''$ of the composition
$E\rightarrowtail E(d)$ of the sequence of admissible monomorphisms
above is an $R$\+module or complex of $R$\+modules filtered by
the $R[s_j^{-1}]$\+modules or complexes of $R[s_j^{-1}]$\+modules
$A_j\in\sA_j$, so $A''\in\sA'_R$.

 The induction base, $j=0$, is dealt with by setting $E(0)=E$.
 Assume that we have already constructed admissible monomorphisms
$E(0)\rightarrowtail E(1)\rightarrowtail\dotsb\rightarrowtail E(j-2)
\rightarrowtail E(j-1)$ with the desired properties, for some $j\ge1$.
 Let us construct an admissible monomorphism
$E(j-1)\rightarrowtail E(j)$.
 
 We have $\Hom_R(R[s_j^{-1}],E(j-1))\in\sE_j$, since $E(j-1)\in\sE_R$
and the class $\sE$ satisfies coascent.
 Choose a special preenvelope sequence
\begin{equation} \label{building-block-preenvelope}
 0\lrarrow\Hom_R(R[s_j^{-1}],E(j-1))\lrarrow B_j\lrarrow A_j\lrarrow0
\end{equation}
with respect to the cotorsion pair $(\sA_j,\sB_j)$ in the exact
category~$\sE_j$.
 Taking the restriction of scalars, one can
view~\eqref{building-block-preenvelope} as a short exact sequence
in $\sE_R$ with $B_j\in\sB_R$, since the classes $\sE$ and $\sB$
satisfy the direct image condition.
 Consider the pushout diagram in the exact category~$\sE_R$
\begin{equation} \label{induction-step-pushout-diagram}
\begin{gathered}
 \xymatrix{
  0 \ar[r] & \Hom_R(R[s_j^{-1}],E(j-1)) \ar[r] \ar[d]
  & B_j \ar[r] \ar[d] & A_j \ar[r] \ar@{=}[d] & 0 \\
  0 \ar[r] & E(j-1) \ar[r] & E(j) \ar[r] & A_j \ar[r] & 0
 }
\end{gathered}
\end{equation}
where the colocalization map $\Hom_R(R[s_j^{-1}],E(j-1))\rarrow E(j-1)$
is induced by the $R$\+module morphism $R\rarrow R[s_j^{-1}]$.
 The object $E(j)\in\sE_R$ is constructed as the pushout object in
this diagram.

 To show that $\Hom_R(R[s_j^{-1}],E(j))\in\sB_j$, it suffices to point
out again that the colocalization
functor~\eqref{colocalization-functor} is exact, so it preserves
pushouts of short exact sequences.
 Therefore, the induced morphism $B_j=\Hom_R(R[s_j^{-1}],B_j)\rarrow
\Hom_R(R[s_j^{-1}],E(j))$ is an isomorphism.

 To prove that $\Hom_R(R[s_k^{-1}],E(j))\in\sB_k$ for $1\le k<j$, we
use the induction assumption telling that $\Hom_R(R[s_k^{-1}],E(j-1))
\in\sB_k$.
 Applying to~\eqref{induction-step-pushout-diagram} the colocalization
functor $M\longmapsto\Hom_R(R[s_k^{-1}],M)$, we obtain a pushout
diagram of short exact sequences in~$\sE_k$
\begin{equation} \label{induction-step-pushout-diagram-colocalized}
\begin{gathered}
 \xymatrix{
  \Hom_R(R[s_k^{-1},s_j^{-1}],E(j-1)) \ar@{>->}[r] \ar[d]
  & \Hom_R(R[s_k^{-1}],B_j) \ar@{->>}[r] \ar[d]
  & \Hom_R(R[s_k^{-1}],A_j) \ar@{=}[d] \\
  \Hom_R(R[s_k^{-1}],E(j-1)) \ar@{>->}[r]
  & \Hom_R(R[s_k^{-1}],E(j)) \ar@{->>}[r]
  & \Hom_R(R[s_k^{-1}],A_j)
 }
\end{gathered}
\end{equation}

 Now we have $\Hom_R(R[s_k^{-1},s_j^{-1}],E(j-1))\in\sB_k$, since
$\Hom_R(R[s_k^{-1}],E(j-1))\in\sB_k$ and the class $\sB$ satisfies
the coascent and direct image conditions.
 Similarly $\Hom_R(R[s_k^{-1}],B_j)\in\sB_k$, since $B_j\in\sB_j$.
 The cotorsion pair $(\sA_k,\sB_k)$ in the exact category $\sE_k$ is
hereditary by assumption, so it follows from the short exact
sequence in the upper line
of~\eqref{induction-step-pushout-diagram-colocalized} that
$\Hom_R(R[s_k^{-1}],A_j)\in\sB_k$.
 Finally, in the lower line of the diagram we have
$\Hom_R(R[s_k^{-1}],E(j-1))\in\sB_k$ and $\Hom_R(R[s_k^{-1}],A_j)
\in\sB_k$, hence $ \Hom_R(R[s_k^{-1}],E(j))\in\sB_k$, as
the class $\sB_k$ is closed under extensions in~$\sE_k$.

 We have constructed the special preenvelope sequence
\begin{equation} \label{special-preenvelope-sequence-constructed}
 0\lrarrow E\lrarrow E(d)\lrarrow A''\lrarrow0
\end{equation}
in $\sE_R$ with $E(d)\in\sB_R$ and $A''\in\sA'_R$.
 In order to finish the proof of the proposition, it remains to
produce the special precover sequences.
 Here we use the construction from the proof of the Salce lemma
(Lemma~\ref{salce-lemma}).

 Let $F\in\sE_R$ be an object.
 We start with choosing special precover sequences
$$
 0\lrarrow B'_k\lrarrow A'_k\lrarrow\Hom_R(R[s_k^{-1}],F)\lrarrow0
$$
with respect to the cotorsion pairs $(\sA_k,\sB_k)$ in the exact
categories $\sE_k$ for all $1\le k\le d$.
 Recall that the natural map $\bigoplus_{k=1}^d\Hom_R(R[s_k^{-1}],F)
\rarrow F$ is an epimorphism in $\sK_R$ and an admissible epimorphism
in $\sK_R^\cta$ by
Lemma~\ref{colocalization-admissible-epimorphism-lemma}.
 Consequently, the composition $\bigoplus_{k=1}^d A'_k\rarrow
\bigoplus_{k=1}^d\Hom_R(R[s_k^{-1}],F)\rarrow F$ is also an admissible
epimorphism in~$\sK_R^\cta$.
 So we have an (admissible) short exact sequence
\begin{equation} \label{generation-sequence-in-salce-lemma}
 0\lrarrow E\lrarrow\bigoplus\nolimits_{k=1}^d A'_k\lrarrow F\lrarrow0
\end{equation}
in the exact category~$\sK_R^\cta$.
 In order to show that~\eqref{generation-sequence-in-salce-lemma} is
a short exact sequence in $\sE_R$, we need to prove that $E\in\sE_R$.

 Since the class $\sE$ satisfies codescent by assumption and
$E\in\sK_R^\cta$, it suffices to check that $\Hom_R(R[s_j^{-1}],E)
\in\sE_j$ for all $1\le j\le d$.
 The admissible epimorphism $A'_j\rarrow\Hom_R(R[s_j^{-1}],F)$ in
$\sE_j$ factorizes into the composition
$$
 A'_j\lrarrow\bigoplus\nolimits_{k=1}^d\Hom_R(R[s_j^{-1}],A'_k)
 \lrarrow\Hom_R(R[s_j^{-1}],F)
$$
where $\Hom_R(R[s_j^{-1}],A'_k)\in\sE_j$ since $A'_k\in\sA_k\subset
\sE_k$ and the class $\sE$ satisfies the coascent and direct image
conditions.
 So the morphism $A'_j\rarrow\bigoplus_{k=1}^d\Hom_R(R[s_j^{-1}],A'_k)$
is a (split) admissible monomorphism in~$\sE_j$.
 By the dual version of~\cite[Exercise~3.11(i)]{Bueh}, it follows that
$\bigoplus_{k=1}^d\Hom_R(R[s_j^{-1}],A'_k)\rarrow\Hom_R(R[s_j^{-1}],F)$
is an admissible epimorphism in~$\sE_j$.
 As $\Hom_R(R[s_j^{-1}],E)$ is the kernel of the latter morphism in
$\sK_{R[s_j^{-1}]}^\cta$, we can conclude that
$\Hom_R(R[s_j^{-1}],E)\in\sE_j$.

 We have shown that~\eqref{generation-sequence-in-salce-lemma} is
an (admissible) short exact sequence in~$\sE_R$.
 Now we can proceed with the construction from the proof of the Salce
lemma.
 Applying the construction of the special preenvelope sequence spelled
out above to the object $E\in\sE_R$, we obtain a short exact
sequence~\eqref{special-preenvelope-sequence-constructed}.
 Consider the pushout diagram
of~\eqref{generation-sequence-in-salce-lemma}
and~\eqref{special-preenvelope-sequence-constructed}
in the exact category~$\sE_R$
\begin{equation} \label{salce-pushout-diagram}
\begin{gathered}
 \xymatrix{
  E \ar@{>->}[r] \ar@{>->}[d] & A' \ar@{->>}[r] \ar@{>->}[d]
  & F \ar@{=}[d] \\
  E(d) \ar@{>->}[r] \ar@{->>}[d] & A \ar@{->>}[r] \ar@{->>}[d] & F \\
  A'' \ar@{=}[r] & A''
 }
\end{gathered}
\end{equation}
where $A'=\bigoplus_{k=1}^d A'_k$ and $A$ is the pushout of the pair of
admissible monomorphisms $E\rarrow A'$ and $E\rarrow E(d)$ in~$\sE_R$.
 It remains to say that the middle horizontal short exact sequence
$0\rarrow E(d)\rarrow A\rarrow F\rarrow0$
in~\eqref{salce-pushout-diagram} is the desired special precover
sequence for the object $F\in\sE_R$, where one has $A\in\sA'_R$ in
view of the middle vertical short exact sequence $0\rarrow A'
\rarrow A\rarrow A''\rarrow0$.
\end{proof}

 We are done with the proof of
Proposition~\ref{antilocality-colocality-main-prop}, and now we
can prove part~(b) of the theorem.

\begin{proof}[Proof of
Theorem~\ref{antilocality-colocality-theorem}(b)]
 By assumptions, both the classes $\sE$ and $\sB$ are very colocal;
and the full subcategory $\sE_R$ is closed under extensions in $\sK_R$
for every $R\in\R$, so it inherits an exact category structure.
 Hence Proposition~\ref{antilocality-colocality-main-prop}
is applicable.

 Let $R\in\R$ be a ring and $s_1$,~\dots, $s_d\in R$ be a collection
of elements generating the unit ideal.
 From these data, the proposition produces a cotorsion pair
$(\sA'_R,\sB_R)$ in $\sE_R$, while in the assumptions of the theorem
we are given a cotorsion pair $(\sA_R,\sB_R)$ in~$\sE_R$.
 It follows that
\begin{itemize}
\item $\sA_R=\sA'_R$;
\item the class $\sA'_R$ is closed under extensions in $\sE_R$,
and consequently also in $\sK^\cta_R$ and in~$\sK_R$.
\end{itemize}
 These observations prove that the class $\sA$ is antilocal.

 In addition, they provide a precise form of the finite filtration
appearing in the antilocality condition for $\sA$, with a specific
bound on the length of the filtration.
 One can certainly say that the filtration length does not exceed~$2d$.
\end{proof}

\begin{proof}[Proof of
Theorem~\ref{antilocality-colocality-theorem}(c)]
 Let $R\in\R$ be a ring and $s_1$,~\dots, $s_d\in R$ be a collection of
elements generating the unit ideal.
 We still keep the notation $\sA_j=\sA_{R[s_j^{-1}]}$, \
$\sB_j=\sA_{R[s_j^{-1}]}$, and $\sE_j=\sE_{R[s_j^{-1}]}$, where
$1\le j\le d$.
 Suppose given a module/complex $D\in\sA_R\cap\sB_R$.
 Then $\Hom_R(R[s_j^{-1}],D)\in\sB_j$ for every $1\le j\le d$, since
the class $\sB$ satisfies coascent.
 For every $1\le k\le d$, pick a special precover sequence
$$
 0\lrarrow B_k\lrarrow A_k\lrarrow\Hom_R(R[s_k^{-1}],D)\lrarrow0
$$
with respect to the cotorsion pair $(\sA_k,\sB_k)$ in the exact
category~$\sE_k$; so $B_k\in\sB_k$ and $A_k\in\sA_k$.
 Then we have $\Hom_R(R[s_k^{-1}],D)\in\sB_k$ and $B_k\in\sB_k$,
hence $A_k\in\sA_k\cap\sB_k$.

 By Lemma~\ref{colocalization-admissible-epimorphism-lemma}, the map
$\bigoplus_{k=1}^d\Hom_R(R[s_k^{-1}],D)\rarrow D$ is an admissible
epimorphism in~$\sK_R^\cta$.
 The composition of this admissible epimorphism with the direct sum
of the admissible epimorphisms $A_k\rarrow\Hom_R(R[s_k^{-1}],D)$
(viewed as morphisms in $\sK_R^\cta$ via the restriction of scalars)
provides an admissible epimorphism $\bigoplus_{k=1}^d A_k\rarrow D$
in~$\sK_R^\cta$.
 Hence we have a short exact sequence
\begin{equation} \label{kernel-special-precover}
 0\lrarrow B\lrarrow\bigoplus\nolimits_{k=1}^d A_k\lrarrow D\lrarrow0
\end{equation}
in the exact category~$\sK_R^\cta$.
 We have already seen in the construction of the special precover
in the proof of Proposition~\ref{antilocality-colocality-main-prop}
that $B\in\sE_R$.
 Let us show that in the situation at hand $B\in\sB_R$.

 The argument is the same.
 Since the class $\sB$ satisfies codescent by assumption, it suffices
to check that $\Hom_R(R[s_j^{-1}],B)\in\sB_j$ for all $1\le j\le d$.
 As a full subcategory closed under extensions in $\sE_j$,
the additive category $\sB_j$ inherits an exact category structure.
 The morphism $A_j\rarrow\Hom_R(R[s_j^{-1}],D)$ is an admissible
epimorphism in $\sB_j$, since it is an admissible epimorphism in $\sE_j$
between two objects from $\sB_j$ with the kernel $B_j\in\sB_j$.
 On the other hand, this morphism factorizes into the composition
$$
 A_j\lrarrow\bigoplus\nolimits_{k=1}^d\Hom_R(R[s_j^{-1}],A_k)
 \lrarrow\Hom_R(R[s_j^{-1}],D)
$$
with $\Hom_R(R[s_j^{-1}],A_k)\in\sB_j$ since $A_k\in\sB_k$ and
the class $\sB$ satisfies the coascent and direct image conditions.
 Hence the morphism $A_j\rarrow\bigoplus_{k=1}^d\Hom_R(R[s_j^{-1}],A_k)$
is a (split) admissible monomorphism in~$\sB_j$.
 So the dual version of~\cite[Exercise~3.11(i)]{Bueh} tells that
$\bigoplus_{k=1}^d\Hom_R(R[s_j^{-1}],A_k)\rarrow\Hom_R(R[s_j^{-1}],D)$
is an admissible epimorphism in~$\sB_j$.
 As $\Hom_R(R[s_j^{-1}],B)$ is the kernel of this morphism in
$\sK_{R[s_j^{-1}]}^\cta$, we can conclude that $\Hom_R(R[s_j^{-1}],B)
\in\sB_j$ and $B\in\sB_R$.

 Similarly to the proof of
Theorem~\ref{locality-antilocality-theorem}(c), as an alternative to
the reference to K\"unzer's axiom, let us spell out its proof in
the situation at hand, for the reader's convenience.
 Consider the pushout diagram in the exact category $\sE_j$, or even
in the abelian category~$\sK_{R[s_j^{-1}]}$
$$
 \xymatrix{
  B_j \ar@{>->}[r] \ar@{>->}[d] & A_j \ar@{->>}[r] \ar@{>->}[d]
  & \Hom_R(R[s_j^{-1}],D) \ar@{=}[d] \\
  \Hom_R(R[s_j^{-1}],B) \ar@{>->}[r] \ar@{->>}[d]
  & \Hom_R(R[s_j^{-1}],A') \ar@{->>}[r] \ar@{->>}[d]
  & \Hom_R(R[s_j^{-1}],D) \\
  \Hom_R(R[s_j^{-1}],A'') \ar@{=}[r] & \Hom_R(R[s_j^{-1}],A'')
 }
$$
where $A'=\bigoplus_{k=1}^d A_k$ and $A''=\bigoplus_{k\ne j} A_k$.
 We have $B_j\in\sB_j$ and $\Hom_R(R[s_j^{-1}],A_k)\in\sB_j$ for all
$1\le k\le d$, so it follows from the leftmost vertical short exact
sequence $0\rarrow B_j\rarrow\Hom_R(R[s_j^{-1}],B)\rarrow
\Hom_R(R[s_j^{-1}],A'')\rarrow0$ that $\Hom_R(R[s_j^{-1}],B)\in\sB_j$.
 Indeed, the class $\sB_j$ is closed under extensions in $\sE_j$
and the class $\sE_j$ is closed under extensions in $\sK_{R[s_j^{-1}]}$
by the assumptions of the theorem.

 Finally, we use the observation that $\Ext_{\sE_R}^1(D,B)=0$, since
$D\in\sA_R$ and $B\in\sB_R$.
 Thus it follows from the short exact
sequence~\eqref{kernel-special-precover} that $D$ is a direct summand
in $\bigoplus_{k=1}^d A_k$.
 It remains to recall that $A_k\in\sA_k\cap\sB_k$.
\end{proof}

\begin{rem}
 Similarly to Proposition~\ref{antilocality-colocality-main-prop}, one
could relax the assumptions of
Theorem~\ref{antilocality-colocality-theorem} by requiring the full
subcategory $\sE_R$ to inherit an exact category structure from
the ambient exact category $\sK_R^\cta$ or from the abelian exact
structure of $\sK_R$, but not necessarily to be closed under extensions
or direct summands in~$\sK_R^\cta$.
 All the arguments in the proof can be made to work in this context.
 The only difference is that then one would have to speak about
 the class $\sA$ being antilocal \emph{within} $\sE$ rather than
in $\sK$, in the sense expained in
Remark~\ref{relative-antilocality-remark} above.
\end{rem}

\Section{Examples of Local and Antilocal Classes}
\label{local-antilocal-classes-examples-secn}

 In this section we discuss examples of cotorsion pairs formed by
local classes $\sA$ and antilocal classes $\sB$, as described by
Theorem~\ref{locality-antilocality-theorem}.
 In all the examples in this section, $\R$ is the class of all
commutative rings.

\begin{ex} \label{injectivity-strongly-antilocal-example}
 Let $\sE_R=\sK_R=R\Modl$ be the abelian category of $R$\+modules.
 Consider the following hereditary complete cotorsion pair in
$R\Modl$: the left class $\sA_R$ is the class of all $R$\+modules,
$\sA_R=R\Modl$, and the right class $\sB_R$ is the class of injective
$R$\+modules, $\sB_R=R\Modl^\inj$.

 Then the class $\sA=\sE$ is obviously very local, and
the class $\sE_R$ is closed under extensions and direct summands
in~$\sK_R$.
 So Theorem~\ref{locality-antilocality-theorem}(c) tells that
the class of injective modules $\sA_R\cap\sB_R=R\Modl^\inj$
is strongly antilocal.
\end{ex} 

\begin{ex} \label{cotorsion-antilocal-example}
 Let $\sE_R=\sK_R=R\Modl$ be the abelian category of $R$\+modules.
 Consider the following hereditary complete cotorsion pair in
$R\Modl$: the left class $\sA_R$ is the class of flat $R$\+modules,
$\sA_R=R\Modl_\fl$, and the right class $\sB_R$ is the class of
cotorsion $R$\+modules, $\sB_R=R\Modl^\cot$
(see Section~\ref{locality-first-examples-secn} for the definition).

 Then the class $\sE$ is obviously very local, and the class $\sA$ is
very local by Example~\ref{flatness-local-example}.
 Thus the class $\sB$ of cotorsion modules is antilocal by
Theorem~\ref{locality-antilocality-theorem}(b).
 (See Example~\ref{cotorsion-antilocal-via-colocality} below for
an alternative proof of this fact.)

 Furthermore, the class of flat cotorsion modules
$R\Modl_\fl^\cot=\sA_R\cap\sB_R$ is strongly antilocal by
Theorem~\ref{locality-antilocality-theorem}(c).
 Notice that, for Noetherian rings, Enochs' classification of flat
cotorsion modules~\cite[Section~2]{En2} is a much stronger result
(but our assertion is applicable to arbitrary commutative rings).
\end{ex}

\begin{ex} \label{contraadjustedness-antilocal-example}
 Let $\sE_R=\sK_R=R\Modl$ be the abelian category of $R$\+modules.
 Consider the following hereditary complete cotorsion pair in
$R\Modl$: the left class $\sA_R$ is the class of very flat
$R$\+modules, $\sA_R=R\Modl_\vfl$, and the right class $\sB_R$ is
the class of contraadjusted $R$\+modules, $\sB_R=R\Modl^\cta$
(see Section~\ref{locality-first-examples-secn} for the definition).

 Then the class $\sE$ is obviously very local, and the class $\sA$ is
very local by Example~\ref{very-flatness-local-example}.
 Thus the class $\sB$ of contraadjusted modules is antilocal by
Theorem~\ref{locality-antilocality-theorem}(b).
 (See Example~\ref{contraadjustedness-antilocal-via-colocality} below
for an alternative proof of this fact.)

 Furthermore, the class of very flat contraadjusted modules
$R\Modl_\vfl^\cta=\sA_R\cap\sB_R$ is strongly antilocal by
Theorem~\ref{locality-antilocality-theorem}(c).
\end{ex}

\begin{ex} \label{contraadjusted-flatness-antilocal-example}
 Let $\sK_R=R\Modl$ be the abelian category of $R$\+modules and
$\sE_R\subset\sK_R$ be the full subcategory of flat $R$\+modules,
$\sE_R=R\Modl_\fl$.
 Then the very flat cotorsion pair $(R\Modl_\vfl$, $R\Modl^\cta)$ in
$R\Modl$ restricts to the full subcategory $\sE_R$, that is, the pair
of classes of very flat $R$\+modules $\sA_R=R\Modl_\vfl$ and flat
contraadjusted $R$\+modules $\sB_R=R\Modl_\fl^\cta=R\Modl_\fl\cap
R\Modl^\cta$ is a hereditary complete cotorsion pair $(\sA_R,\sB_R)$
in~$\sE_R$.
 This holds because the class $\sE_R$ is closed under extensions
and kernels of epimorphisms in $\sK_R$, and all very flat $R$\+modules
belong to~$\sE_R$ (see Lemmas~\ref{cotorsion-pair-restricts}(a)
and~\ref{restricted-cotorsion-hereditary}).

 Now both the classes $\sE$ and $\sA$ are very local, and the class
$\sE_R$ is closed under extensions and direct summands in~$\sK_R$.
 Thus the class $\sB$ of flat contraadjusted modules is antilocal by
Theorem~\ref{locality-antilocality-theorem}(b).
 (See Example~\ref{contraadjusted-flatness-antilocal-via-colocality}
below for an alternative proof of this result.)
\end{ex}

\begin{ex} \label{homotopy-injectivity-antilocal-example}
 Let $\sE_R=\sK_R=\bC(R\Modl)$ be the abelian category of complexes of
$R$\+modules.
 Consider the following hereditary complete cotorsion pair in
$\bC(R\Modl)$: the left class $\sA_R=\bC_\ac(R\Modl)$ is the class of
acyclic complexes of $R$\+modules, and the right class
$\sB_R=\bC^\hin(R\Modl^\inj)$ is the class of homotopy injective
complexes of injective $R$\+modules (as defined in
Section~\ref{locality-first-examples-secn}).
 The pair of classes ($\bC_\ac(R\Modl)$, $\bC^\hin(R\Modl^\inj)$) is
indeed a complete cotorsion pair in $\bC(R\Modl)$ by~\cite{EJX},
\cite[Theorem~2.3.13]{Hov-book}, \cite[Example~3.2]{Hov},
\cite[Proposition~1.3.5(2)]{Bec}, or~\cite[Theorem~8.4]{PS4}.
 This cotorsion pair is hereditary, because the class of acyclic
complexes is obviously closed under the kernels of epimorphisms (cf.\
Example~\ref{homotopy-injectivity-of-injectives-colocal-example}).

 The class $\sE$ of all complexes is obviously very local.
 The class $\sA$ of acyclic complexes is very local as well.
 More generally, the complex of $S$\+modules $S\ot_R X^\bu$ is acyclic
for any acyclic complex of $R$\+modules $X^\bu$ and any commutative
ring homomorphism $R\rarrow S$ making $S$ a flat $R$\+module.
 The underlying complex of $R$\+modules of an acyclic complex of
$S$\+modules is obviously acyclic for any ring homomorphism
$R\rarrow S$.
 So the ascent and direct image conditions hold, and
Lemma~\ref{ascent+direct-image-imply-descent} provides the (well-known)
implication that the acyclicity satisfies descent.

 Thus we can conclude that the class $\sB$ of homotopy injective
complexes of injective modules is antilocal by
Theorem~\ref{locality-antilocality-theorem}(b).
 Another proof of this fact can be obtained as a particular case
of Corollary~\ref{many-colocal-are-antilocal} below using
Example~\ref{homotopy-injectivity-of-injectives-colocal-example}.
\end{ex}

\begin{rem} \label{diagrams-termwise-in-class-remark}
 In the following Example~\ref{termwise-injectivity-antilocal-example}
we discuss antilocality of the class of all complexes of injective modules.
 Notice that, for any \emph{local} class of modules $\sL_R\subset
R\Modl$, the class $\bC(\sL)\subset\bC(R\Modl)$ of all complexes with
the terms in $\sL$ is obviously local.
 Likewise, for any \emph{colocal} class of modules $\sL_R\subset
R\Modl^\cta$, the class $\bC(\sL)\subset\bC(R\Modl^\cta)$ of all
complexes with the terms in $\sL$ is obviously colocal.
 However, given an antilocal class of modules
$\sF\subset R\Modl$, it is \emph{not at all clear} why the class
$\bC(\sF)\subset\bC(R\Modl)$ of all complexes with the terms in $\sF$
should be \emph{antilocal}.
 The same discussion applies, more generally, to categories of
diagrams of modules from various classes, etc.
 So, the example is nontrivial.
\end{rem}

 A complex of $R$\+modules $X^\bu$ is said to be \emph{coacyclic in
the sense of Becker} (``Becker-coacyclic'' for brevity)
\cite[Proposition~1.3.6(2)]{Bec}, \cite[Section~9]{PS4} if, for any
complex of injective $R$\+modules $J^\bu$, any morphism of complexes
of $R$\+modules $X^\bu\rarrow J^\bu$ is homotopic to zero.

\begin{ex} \label{termwise-injectivity-antilocal-example}
 Let $\sE_R=\sK_R=\bC(R\Modl)$ be the abelian category of complexes of
$R$\+modules.
 Consider the following hereditary complete cotorsion pair in
$\bC(R\Modl)$: the left class $\sA_R=\bC_\coac(R\Modl)$ is the class
of Becker-coacyclic complexes of $R$\+modules, and the right class
$\sB_R=\bC(R\Modl^\inj)$ is the class of all complexes of injective
$R$\+modules.
 The pair of classes ($\bC_\coac(R\Modl)$, $\bC(R\Modl^\inj)$) is
indeed a complete cotorsion pair in $\bC(R\Modl)$
by~\cite[Proposition~1.3.6(2)]{Bec}, \cite[Theorem~4.2]{Gil},
or~\cite[Theorem~9.3]{PS4}.
 This cotorsion pair is hereditary, because the class of all complexes
of injective modules is obviously closed under the cokernels of
monomorphisms.

 The class $\sA$ of Becker-coacyclic complexes satisfies the ascent
and direct image conditions, because all injective modules are
contraadjusted and the class $\sB$ of all complexes of injective modules
satisfies the coascent and direct image conditions (see
Example~\ref{injectivity-colocal-example} and
Proposition~\ref{left-right-ascent-coascent-direct-image-prop}).
 Since the class $\sA$ is also closed under kernels of epimorphisms,
it is very local by Lemma~\ref{ascent+direct-image-imply-descent}.
 We refer to~\cite[Section~A.2]{Psemten} for a further discussion of
the locality properties of coacyclic complexes.

 Thus we can conclude that the class $\sB$ of all complexes of injective
modules is antilocal by Theorem~\ref{locality-antilocality-theorem}(b).
 Another proof of this fact can be obtained as a particular case
of Corollary~\ref{many-colocal-are-antilocal} below using
Example~\ref{injectivity-colocal-example}
with Remark~\ref{diagrams-termwise-in-class-remark}.
\end{ex}

\begin{rem}
 Of course, the finite filtration appearing in the definition of
antilocality for the classes of homotopy injective complexes of
injective modules or all complexes of injective modules (from
Examples~\ref{homotopy-injectivity-antilocal-example}
and~\ref{termwise-injectivity-antilocal-example}) is always
\emph{termwise} split (as any filtration by injective modules is split).
 However, these filtrations do \emph{not} split in the abelian
category of complexes (generally speaking).
 In other words, the mentioned two classes of complexes of modules are
\emph{not} strongly antilocal.

 Indeed, let $M$ be an $R$\+module and $J^\bu$ be its injective
coresolution.
 Then $J^\bu$ is a bounded below complex of injective modules; all
such complexes are homotopy injective.
 Still, $J^\bu$ is usually \emph{not} isomorphic to a direct summand of
a direct sum of complexes of $R[s_j^{-1}]$\+modules (with $1\le j\le d$
as in our usual notation); for otherwise one could pass to
the cohomology modules and conclude that $M$ is a direct summand of
a direct sum of $R[s_j^{-1}]$\+modules, which is usually not the case.
 (Cf.\ Example~\ref{all-modules-not-antilocal}.)
\end{rem}

\begin{ex}
 Let $\sK_R=\bC(R\Modl)$ be the abelian category of complexes of
$R$\+mod\-ules and $\sE_R\subset\sK_R$ be the full subcategory of
acyclic complexes, $\sE_R=\bC_\ac(R\Modl)$.
 Then the Becker coderived cotorsion pair ($\bC_\coac(R\Modl)$,
$\bC(R\Modl^\inj)$) in $\bC(R\Modl)$ (see
Example~\ref{termwise-injectivity-antilocal-example}) restricts to
the full subcategory $\sE_R$, that is, the pair of classes of
Becker-coacyclic complexes of $R$\+modules $\sA_R=\bC_\coac(R\Modl)$
and acyclic complexes of injective $R$\+modules $\sB_R=
\bC_\ac(R\Modl^\inj)=\bC_\ac(R\Modl)\cap\bC(R\Modl^\inj)$ is
a hereditary complete cotorsion pair $(\sA_R,\sB_R)$ in~$\sE_R$.
 This holds because the class $\sE_R$ is closed under extensions and
kernels of epimorphisms in $\sK_R$, and all Becker-coacyclic
complexes are acyclic (see Lemmas~\ref{cotorsion-pair-restricts}(a)
and~\ref{restricted-cotorsion-hereditary}).
{\hbadness=2200\par}

 Now both the classes $\sE$ and $\sA$ are very local
(by Examples~\ref{homotopy-injectivity-antilocal-example}
and~\ref{termwise-injectivity-antilocal-example}), and the class
$\sE_R$ is closed under extensions and direct summands in~$\sK_R$.
 Thus the class $\sB$ of acyclic complexes of injective modules
is antilocal by Theorem~\ref{locality-antilocality-theorem}(b).
\end{ex}

 Given a commutative ring $R$, we denote by $\bC_\ac^\bvfl(R\Modl_\vfl)$
the class of acyclic complexes of very flat $R$\+modules \emph{with
very flat modules of cocycles}.
 More generally, for any associative ring $R$, we let
$\bC_\ac^\bfl(R\Modl_\fl)$ denote the class of acyclic complexes
of flat left $R$\+modules \emph{with flat modules of cocycles}.

\begin{prop} \label{very-flat-tilde-termwise-cta-cot-cotorsion-pair}
\textup{(a)} For any commutative ring $R$, consider the following pair
of classes of complexes of $R$\+modules.
 The left class\/ $\sA_R=\bC_\ac^\bvfl(R\Modl_\vfl)$ is the class of
acyclic complexes of very flat $R$\+modules with very flat modules of
cocycles.
 The right class\/ $\sB_R=\bC(R\Modl^\cta)$ is the class of all
complexes of contraadjusted $R$\+modules.
 Then the pair of classes $(\sA_R,\sB_R)$ is a hereditary complete
cotorsion pair in the abelian category of complexes\/ $\bC(R\Modl)$.
\par
\textup{(b)}  For any associative ring $R$, consider the following pair
of classes of complexes of left $R$\+modules.
 The left class\/ $\sA_R=\bC_\ac^\bfl(R\Modl_\fl)$ is the class of
acyclic complexes of flat $R$\+modules with flat modules of cocycles.
 The right class\/ $\sB_R=\bC(R\Modl^\cot)$ is the class of all
complexes of cotorsion $R$\+modules.
 Then the pair of classes $(\sA_R,\sB_R)$ is a hereditary complete
cotorsion pair in the abelian category of complexes\/ $\bC(R\Modl)$.
\end{prop}

\begin{proof}
 Part~(a): the category of complexes $\bC(R\Modl)$ is a Grothendieck
abelian category with enough projective objects; in fact,
the projective objects of $\bC(R\Modl)$ are the contractible complexes
of projective modules~\cite[Theorem~IV.3.2]{EM},
\cite[Lemma~1.3.3]{Bec}, \cite[Lemma~5.2(b)]{PS4}.
 By~\cite[Proposition~4.4]{Sto}, the class $\sA_R=
\bC_\ac^\bvfl(R\Modl_\vfl)$ of acyclic complexes of very flat modules
with very flat modules of cocycles is deconstructible in $\bC(R\Modl)$.
 Furthermore, all the projective objects of $\bC(R\Modl_\vfl)$ belong
to~$\sA_R$.
 According to Theorem~\ref{eklof-trlifaj-theorem} (see the last
paragraph of Section~\ref{preliminaries-cotorsion-pairs-secn}), it
follows that $(\sA_R,\sA_R^{\perp_1})$ is a complete cotorsion pair
in $\bC(R\Modl)$. {\emergencystretch=1em\par}

 Let us show that $\sA_R^{\perp_1}=\bC(R\Modl^\cta)$.
 The inclusion $\sA_R^{\perp_1}\subset\bC(R\Modl^\cta)$ is implied
by Lemma~\ref{disk-complexes-lemma}.
 Conversely, let $F^\bu$ be an acyclic complex of very flat $R$\+modules
with very flat modules of cocycles, and $C^\bu$ be a complex of
contraadjusted $R$\+modules.
 Choose a contractible complex of projective $R$\+modules $P^\bu$
together with a termwise surjective morphism of complexes
$P^\bu\rarrow F^\bu$.
 Since very flat $R$\+modules have projective dimension~$\le1$
(by Lemma~\ref{generated-by-projdim1-lemma}), and since the functors
assigning to an acyclic complex its modules of cocycles are exact,
the kernel $Q^\bu$ of the latter morphism of complexes is also
a contractible complex of projective $R$\+modules.
  Now we have a short exact sequence of complexes
$0\rarrow\Hom_R(F^\bu,C^\bu)\rarrow\Hom_R(P^\bu,C^\bu)\rarrow
\Hom_R(Q^\bu,C^\bu)\rarrow0$ (since $F^\bu$ is a complex of very flat
$R$\+modules and $C^\bu$ is a complex of contraadjusted $R$\+modules).
 Since the complexes $\Hom_R(P^\bu,C^\bu)$ and $\Hom_R(Q^\bu,C^\bu)$
are acyclic, so is the complex $\Hom_R(F^\bu,C^\bu)$.
 It remains to apply Lemma~\ref{ext-homotopy-hom-lemma} in order to
conclude that $\Ext^1_{\bC(R\Modl)}(F^\bu,C^\bu)=0$.

 Finally, the class $\bC(R\Modl^\cta)$ is closed under the cokernels of
monomorphisms in $\bC(R\Modl)$, because the class $R\Modl^\cta$ is
closed under the cokernels of monomorphisms in $R\Modl$.
 Hence the cotorsion pair $(\sA_R,\sB_R)$ in $\bC(R\Modl)$
is hereditary.

 Part~(b) is~\cite[Corollary~4.10]{Gil0} together
with~\cite[Theorem~5.3]{BCE}.
 Essentially, the first paragraph of the proof of part~(a) above
applies in the context of part~(b) as well (using the fact that
the class of flat modules is deconstructible
by~\cite[Lemma~1 and Proposition~2]{BBE} or~\cite[Lemma~6.23]{GT});
and the inclusion $\sA_R^{\perp_1}\subset\bC(R\Modl^\cot)$
once again holds by Lemma~\ref{disk-complexes-lemma}.
 Then the result of~\cite[Theorem~5.3]{BCE} tells that any morphism
from an acyclic complex of flat modules with flat modules of cocycles
to a complex of cotorsion modules is homotopic to zero.
\end{proof}

\begin{ex} \label{termwise-contraadjusted-antilocal-example}
 Let $\sE_R=\sK_R=\bC(R\Modl)$ be the abelian category of complexes
of $R$\+modules.
 Consider the hereditary complete cotorsion pair $(\sA_R,\sB_R)$ from
Proposition~\ref{very-flat-tilde-termwise-cta-cot-cotorsion-pair}(a):
so, $\sA_R=\bC_\ac^\bvfl(R\Modl_\vfl)$ is the class of acyclic
complexes of very flat $R$\+modules with very flat modules of
cocycles, and $\sB_R=\bC(R\Modl^\cta)$ is the class of all complexes
of contraadjusted $R$\+modules.

 The class $\sA$ of acyclic complexes of very flat modules with
very flat modules of cocycles is very local.
 More generally, the complex $S\ot_R F^\bu$ belongs to
$\bC_\ac^\bvfl(S\Modl_\vfl)$ for any complex $F^\bu\in
\bC_\ac^\bvfl(R\Modl_\vfl)$ and any commutative ring homomorphism
$R\rarrow S$.
 The underlying complex of $R$\+modules of any complex from
$\bC_\ac^\bvfl(S\Modl_\vfl)$ belongs to $\bC_\ac^\bvfl(R\Modl_\vfl)$
for any commutative ring homomorphism $R\rarrow S$ such that
the $R$\+module $S[s^{-1}]$ is very flat for all $s\in S$
(see Example~\ref{very-flatness-local-example}).
 Hence the ascent and direct image conditions hold.
 The class $\sA_R$ is closed under kernels of epimorphisms in
$\bC(R\Modl)$, since the cotorsion pair $(\sA_R,\sB_R)$ is
hereditary in $\bC(R\Modl)$.
 So Lemma~\ref{ascent+direct-image-imply-descent} tells that
the class $\sA$ satisfies descent.

 Therefore, we can conclude that the class $\sB$ of all complexes of
contraadjusted modules is antilocal
by Theorem~\ref{locality-antilocality-theorem}(b).
 (See Example~\ref{termwise-cta-antilocal-via-colocality} below for
an alternative proof of this fact.)
\end{ex}

\begin{ex} \label{termwise-cotorsion-antilocal-example}
 Let $\sE_R=\sK_R=\bC(R\Modl)$ be the abelian category of complexes
of $R$\+modules.
 Consider the hereditary complete cotorsion pair $(\sA_R,\sB_R)$ from
Proposition~\ref{very-flat-tilde-termwise-cta-cot-cotorsion-pair}(b)
(for a commutative ring~$R$).
 So, $\sA_R=\bC_\ac^\bfl(R\Modl_\fl)$ is the class of acyclic complexes
of flat $R$\+modules with flat modules of cocycles, and
$\sB_R=\bC(R\Modl^\cot)$ is the class of all complexes
of cotorsion $R$\+modules.

 The class $\sA$ of acyclic complexes of flat modules with
flat modules of cocycles is very local.
 More generally, the complex $S\ot_R F^\bu$ belongs to
$\bC_\ac^\bfl(S\Modl_\fl)$ for any complex $F^\bu\in
\bC_\ac^\bfl(R\Modl_\fl)$ and any ring homomorphism $R\rarrow S$.
 The underlying complex of $R$\+modules of any complex from
$\bC_\ac^\bfl(S\Modl_\fl)$ belongs to $\bC_\ac^\bfl(R\Modl_\fl)$
for any commutative ring homomorphism $R\rarrow S$ making $S$
a flat $R$\+module.
 Hence the ascent and direct image conditions hold.
 The class $\sA_R$ is closed under kernels of epimorphisms in
$\bC(R\Modl)$, since the cotorsion pair $(\sA_R,\sB_R)$ is
hereditary in $\bC(R\Modl)$.
 So Lemma~\ref{ascent+direct-image-imply-descent} tells that
the class $\sA$ satisfies descent.

 Therefore, we can conclude that the class $\sB$ of all complexes of
cotorsion modules is antilocal
by Theorem~\ref{locality-antilocality-theorem}(b).
 (See Example~\ref{termwise-cotorsion-antilocal-via-colocality}
below for an alternative proof of this fact.)
\end{ex}

\begin{ex} \label{termwise-flat-contraadjusted-antilocal-example}
 Let $\sK_R=\bC(R\Modl)$ be the abelian category of complexes of
$R$\+modules and $\sE_R\subset\sK_R$ be the full subcategory of
complexes of flat $R$\+modules, $\sE_R=\bC(R\Modl_\fl)$.
 Then the cotorsion pair ($\bC_\ac^\bvfl(R\Modl_\vfl)$,
$\bC(R\Modl^\cta)$) in $\bC(R\Modl)$ (see
Proposition~\ref{very-flat-tilde-termwise-cta-cot-cotorsion-pair}(a))
restricts to the full subcategory $\sE_R$, that is, the pair of
classes of acyclic complexes of very flat $R$\+modules with very flat
modules of cocycles $\sA_R=\bC_\ac^\bvfl(R\Modl_\vfl)$ and
all complexes of flat contraadjusted $R$\+modules
$\sB_R=\bC(R\Modl_\fl)\cap\bC(R\Modl^\cta)=\bC(R\Modl_\fl^\cta)$
is a hereditary complete cotorsion pair $(\sA_R,\sB_R)$ in~$\sE_R$.
 This holds by Lemmas~\ref{cotorsion-pair-restricts}(a)
and~\ref{restricted-cotorsion-hereditary}, because the class $\sE_R$
is closed under extensions and the kernels of epimorphisms in~$\sK_R$.

 Now the class $\sE$ is very local essentially by
Example~\ref{flatness-local-example} (cf.\
Remark~\ref{diagrams-termwise-in-class-remark}), and
the class $\sA$ is very local by
Example~\ref{termwise-contraadjusted-antilocal-example}.
 Furthermore, the class $\sE_R$ is closed under extensions and direct
summands in~$\sK_R$.
 Thus the class $\sB$ of all complexes of flat contraadjusted modules
is antilocal by Theorem~\ref{locality-antilocality-theorem}(b).
 (This is also provable by a method sketched below at the end of
Example~\ref{termwise-very-flat-cta-antilocal-via-colocality}.)
\end{ex}

\begin{ex} \label{termwise-flat-cotorsion-antilocal-example}
 Let $\sK_R=\bC(R\Modl)$ be the abelian category of complexes of
$R$\+modules and $\sE_R\subset\sK_R$ be the full subcategory of
complexes of flat $R$\+modules, $\sE_R=\bC(R\Modl_\fl)$.
 Then the cotorsion pair ($\bC_\ac^\bfl(R\Modl_\fl)$,
$\bC(R\Modl^\cot)$) in $\bC(R\Modl)$ (see
Proposition~\ref{very-flat-tilde-termwise-cta-cot-cotorsion-pair}(b))
restricts to the full subcategory $\sE_R$, that is, the pair of
classes of acyclic complexes of flat $R$\+modules with flat modules of
cocycles $\sA_R=\bC_\ac^\bfl(R\Modl_\fl)$ and all complexes of flat
cotorsion $R$\+modules
$\sB_R=\bC(R\Modl_\fl)\cap\bC(R\Modl^\cot)=\bC(R\Modl_\fl^\cot)$
is a hereditary complete cotorsion pair $(\sA_R,\sB_R)$ in~$\sE_R$.
 This holds by Lemmas~\ref{cotorsion-pair-restricts}(a)
and~\ref{restricted-cotorsion-hereditary}.

 Now the class $\sE$ is very local essentially by
Example~\ref{flatness-local-example}, and the class $\sA$ is very local
by Example~\ref{termwise-cotorsion-antilocal-example}.
 Furthermore, class $\sE_R$ is closed under extensions and direct
summands in~$\sK_R$.
 Thus the class $\sB$ of all complexes of flat cotorsion modules
is antilocal by Theorem~\ref{locality-antilocality-theorem}(b).
 (This is also provable by a method sketched below at the end of
Example~\ref{termwise-very-flat-cta-antilocal-via-colocality}.)
\end{ex}

\begin{ex} \label{termwise-very-flat-contraadjusted-antilocal-example}
 Let $\sK_R=\bC(R\Modl)$ be the abelian category of complexes of
$R$\+modules and $\sE_R\subset\sK_R$ be the full subcategory of
complexes of very flat $R$\+modules, $\sE_R=\bC(R\Modl_\vfl)$.
 Then the cotorsion pair ($\bC_\ac^\bvfl(R\Modl_\vfl)$,
$\bC(R\Modl^\cta)$) in $\bC(R\Modl)$ (see
Proposition~\ref{very-flat-tilde-termwise-cta-cot-cotorsion-pair}(a))
restricts to the full subcategory $\sE_R$, that is, the pair of
classes of acyclic complexes of very flat $R$\+modules with very flat
modules of cocycles $\sA_R=\bC_\ac^\bvfl(R\Modl_\vfl)$ and
all complexes of very flat contraadjusted $R$\+modules
$\sB_R=\bC(R\Modl_\vfl)\cap\bC(R\Modl^\cta)=\bC(R\Modl_\vfl^\cta)$
is a hereditary complete cotorsion pair $(\sA_R,\sB_R)$ in~$\sE_R$.
 This holds by Lemmas~\ref{cotorsion-pair-restricts}(a)
and~\ref{restricted-cotorsion-hereditary}, because the class $\sE_R$
is closed under extensions and the kernels of epimorphisms in~$\sK_R$.

 Now the class $\sE$ is very local essentially by
Example~\ref{very-flatness-local-example} (cf.\
Remark~\ref{diagrams-termwise-in-class-remark}), and
the class $\sA$ is very local by
Example~\ref{termwise-contraadjusted-antilocal-example}.
 Furthermore, the class $\sE_R$ is closed under extensions and direct
summands in~$\sK_R$.
 Thus the class $\sB$ of all complexes of very flat contraadjusted
modules is antilocal by Theorem~\ref{locality-antilocality-theorem}(b).
 (See Example~\ref{termwise-very-flat-cta-antilocal-via-colocality}
below for an alternative proof of this fact.)
\end{ex}

 A complex of $R$\+modules $Y^\bu$ is said to be \emph{contraacyclic
in the sense of Becker} (``Becker-contraacyclic'' for brevity)
\cite[Proposition~1.3.6(1)]{Bec}, \cite[Section~7]{PS4} if, for any
complex of projective $R$\+modules $P^\bu$, any morphism of complexes
of $R$\+modules $P^\bu\rarrow Y^\bu$ is homotopic to zero.
 We denote the class of contraacyclic complexes of $R$\+modules
by $\bC^\ctrac(R\Modl)\subset\bC(R\Modl)$.

\begin{prop} \label{very-flat-contraacyclic-cotorsion-pair}
\textup{(a)} For any commutative ring $R$, consider the following pair
of classes of complexes of $R$\+modules.
 The left class\/ $\sA_R=\bC(R\Modl_\vfl)$ is the class of arbitrary
complexes of very flat $R$\+modules.
 The right class\/ $\sB_R=\bC^\ctrac(R\Modl^\cta)=\bC^\ctrac(R\Modl)
\cap\bC(R\Modl^\cta)$ is the class of contraacyclic complexes of
contraadjusted $R$\+modules, i.~e., the intersection of the classes of
contraacyclic complexes of $R$\+modules and arbitrary complexes of
contraadjusted $R$\+modules.
 Then the pair of classes $(\sA_R,\sB_R)$ is a hereditary complete
cotorsion pair in the abelian category of complexes\/ $\bC(R\Modl)$.
\par
\textup{(b)} For any associative ring $R$, consider the following pair
of classes of complexes of left $R$\+modules.
 The left class\/ $\sA_R=\bC(R\Modl_\fl)$ is the class of arbitrary
complexes of flat $R$\+modules.
 The right class\/ $\sB_R=\bC^\ctrac(R\Modl^\cot)=\bC^\ctrac(R\Modl)
\cap\bC(R\Modl^\cot)$ is the class of contraacyclic complexes of
cotorsion $R$\+modules, i.~e., the intersection of the classes of
contraacyclic complexes of $R$\+modules and arbitrary complexes of
cotorsion $R$\+modules.
 Then the pair of classes $(\sA_R,\sB_R)$ is a hereditary complete
cotorsion pair in the abelian category of complexes\/ $\bC(R\Modl)$.
\end{prop}

\begin{proof}
 Part~(a): as mentioned in the proof of
Proposition~\ref{very-flat-tilde-termwise-cta-cot-cotorsion-pair}(a),
the category of complexes $\bC(R\Modl)$ is a Grothendieck abelian
category with enough projective objects.
 By~\cite[Proposition~4.3]{Sto}, the class $\sA_R=\bC(R\Modl_\vfl)$ of
all complexes of very flat modules is deconstructible in
$\bC(R\Modl)$.
 Furthermore, all the projective objects of $\bC(R\Modl_\vfl)$ belong
to~$\sA_R$.
 According to Theorem~\ref{eklof-trlifaj-theorem} (see the last
paragraph of Section~\ref{preliminaries-cotorsion-pairs-secn}), it
follows that $(\sA_R,\sA_R^{\perp_1})$ is a complete cotorsion pair
in $\bC(R\Modl)$. {\emergencystretch=1em\par}

 Let us show that $\sA_R^{\perp_1}=\bC^\ctrac(R\Modl^\cta)$.
 The inclusion $\sA_R^{\perp_1}\subset\bC^\ctrac(R\Modl)$ holds,
because all the complexes of projective $R$\+modules belong to~$\sA_R$.
 The inclusion $\sA_R^{\perp_1}\subset\bC(R\Modl^\cta)$ is implied
by Lemma~\ref{disk-complexes-lemma}.
 Conversely, let $F^\bu$ be a complex of very flat $R$\+modules and
$C^\bu$ be a Becker-contraacyclic complex of contraadjusted
$R$\+modules.
 Choose a complex of projective $R$\+modules $P^\bu$ together with
a termwise surjective morphism of complexes $P^\bu\rarrow F^\bu$.
 Since very flat $R$\+modules have projective dimension~$\le1$
(by Lemma~\ref{generated-by-projdim1-lemma}), the kernel $Q^\bu$
of the latter morphism of complexes is also a complex of projective
$R$\+modules.
 The argument finishes similarly to the second paragraph of
the proof of
Proposition~\ref{very-flat-tilde-termwise-cta-cot-cotorsion-pair}(a).

 Finally, the class $\bC(R\Modl_\vfl)$ is closed under the kernels of
epimorphisms in $\bC(R\Modl)$, because the class $R\Modl_\vfl$ is
closed under the kernels of epimorphisms in $R\Modl$.
 Hence the cotorsion pair $(\sA_R,\sB_R)$ in $\bC(R\Modl)$
is hereditary.

 Part~(b): by~\cite[Lemma~1 and Proposition~2]{BBE}
or~\cite[Lemma~6.23]{GT}, the class of flat modules is
deconstructible in $R\Modl$.
 By~\cite[Proposition~4.3]{Sto}, it follows that the class $\sA_R=
\bC(R\Modl_\fl)$ of all complexes of flat modules is deconstructible
in $\bC(R\Modl)$.
 Similarly to the proof of part~(a), it remains to show that
$\sA_R^{\perp_1}=\sB_R$, and the only nontrivial aspect of it is
to prove that the complex of abelian groups $\Hom_R(F^\bu,C^\bu)$ is
acyclic for any complexes of $R$\+modules $F^\bu\in\sA_R$ and
$C^\bu\in\sB_R$.
 By~\cite[Remark~2.11 and Lemma~8.5]{Neem-fl}, there exists a complex
of projective $R$\+modules $P^\bu$ together with a morphism of
complexes $P^\bu\rarrow F^\bu$ whose cone $G^\bu$ is an acyclic complex
of flat $R$\+modules with flat modules of cocycles.

 Now the complex $\Hom_R(P^\bu,C^\bu)$ is acyclic, since the complex
$C^\bu$ is contraacyclic.
 By~\cite[Theorem~5.1(2)]{BCE}, any acyclic complex of cotorsion
modules has cotorsion modules of cocycles; in particular, $C^\bu$ is
an acyclic complex of cotorsion modules with cotorsion modules of
cocycles.
 Following~\cite[Definition~3.3 and Lemma~3.9]{Gil0}, the complex
$\Hom_R(G^\bu,C^\bu)$ is acyclic as well.
\end{proof}

\begin{ex} \label{contraadjusted-contraacyclicity-antilocal-example}
 Let $\sE_R=\sK_R=\bC(R\Modl)$ be the abelian category of complexes of
$R$\+modules.
 Consider the hereditary complete cotorsion pair $(\sA_R,\sB_R)$ from
Proposition~\ref{very-flat-contraacyclic-cotorsion-pair}(a): so,
$\sA_R=\bC(R\Modl_\vfl)$ is the class of all complexes of very flat
$R$\+modules, and $\sB_R=\bC^\ctrac(R\Modl^\cta)$ is the class of
Becker-contraacyclic complexes of contraadjusted $R$\+modules.

 Then the class $\sA$ is very local essentially by
Example~\ref{very-flatness-local-example} (cf.\
Remark~\ref{diagrams-termwise-in-class-remark}).
 Thus the class $\sB$ of Becker-contraacyclic complexes of
contraadjusted modules is antilocal
by Theorem~\ref{locality-antilocality-theorem}(b).
\end{ex}

\begin{ex}
 Let $\sE_R=\sK_R=\bC(R\Modl)$ be the abelian category of complexes of
$R$\+modules.
 Consider the hereditary complete cotorsion pair $(\sA_R,\sB_R)$ from
Proposition~\ref{very-flat-contraacyclic-cotorsion-pair}(b)
(for a commutative ring~$R$).
 So, $\sA_R=\bC(R\Modl_\fl)$ is the class of all complexes of flat
$R$\+modules, and $\sB=\bC^\ctrac(R\Modl^\cot)$ is the class of
Becker-contraacyclic complexes of cotorsion $R$\+modules.

 Then the class $\sA$ is very local essentially by
Example~\ref{flatness-local-example} (cf.\
Remark~\ref{diagrams-termwise-in-class-remark}).
 Thus the class $\sB$ of Becker-contraacyclic complexes of cotorsion
modules is antilocal by Theorem~\ref{locality-antilocality-theorem}(b).
\end{ex}

\Section{Examples of Colocal and Antilocal Classes}
\label{antilocal-colocal-classes-examples-secn}

 In this section we demonstrate examples of cotorsion pairs formed by
antilocal classes $\sA$ and colocal classes $\sB$, as described by
Theorem~\ref{antilocality-colocality-theorem}.
 In all the examples in this section, $\R$ is the class of all
commutative rings.

 We start with suggesting alternative approaches to
Examples~\ref{cotorsion-antilocal-example},
\ref{contraadjustedness-antilocal-example},
\ref{contraadjusted-flatness-antilocal-example},
\ref{termwise-contraadjusted-antilocal-example},
and~\ref{termwise-cotorsion-antilocal-example}
from the previous section.

\begin{ex} \label{contraadjustedness-antilocal-via-colocality}
 Let $\sE_R=\sK_R^\cta=R\Modl^\cta$ be the exact category of
contraadjusted $R$\+modules.
 Consider the following hereditary complete cotorsion pair in
$R\Modl^\cta$: the left class $\sA_R$ is the class of all contraadjusted
$R$\+modules, $\sA_R=R\Modl^\cta$, and the right class $\sB_R$ is
the class of injective $R$\+modules, $\sB_R=R\Modl^\inj$.

 To say that this is a hereditary complete cotorsion pair, means,
basically, that the cotorsion pair (all modules, injective modules)
from Example~\ref{injectivity-strongly-antilocal-example} restricts to
the exact subcategory of contraadjusted modules in $R\Modl$.
 This holds because the class $R\Modl^\cta$ is closed under extensions
and cokernels of monomorphisms in $R\Modl$, and all injective
modules are contraadjusted (see Lemmas~\ref{cotorsion-pair-restricts}(b)
and~\ref{restricted-cotorsion-hereditary}).

 Now the class $\sE$ is very colocal by
Example~\ref{contraadjustedness-colocal-example}, and the class $\sB$
is very colocal by Example~\ref{injectivity-colocal-example}.
 Thus the class $\sA$ of contraadjusted modules is antilocal
by Theorem~\ref{antilocality-colocality-theorem}(b).
 We have obtained an alternative proof of the antilocality of
contraadjustedness, based on
Theorem~\ref{antilocality-colocality-theorem}(b) instead of
Theorem~\ref{locality-antilocality-theorem}(b)
(cf.\ Example~\ref{contraadjustedness-antilocal-example}).
\end{ex}

\begin{ex} \label{cotorsion-antilocal-via-colocality}
 Let $\sK_R^\cta=R\Modl^\cta$ be the exact category of
contraadjusted $R$\+mod\-ules and $\sE_R\subset\sK_R^\cta$ be the full
subcategory of cotorsion $R$\+modules, $\sE_R=R\Modl^\cot$.
 Then the injective cotorsion pair ($R\Modl$, $R\Modl^\inj$) in
$R\Modl$ restricts  to the full subcategory $\sE_R$, i.~e.,
the pair of classes of cotorsion $R$\+modules $\sA_R=R\Modl^\cot$ and
injective $R$\+modules $\sB_R=R\Modl^\inj$ is a hereditary complete
cotorsion pair $(\sA_R,\sB_R)$ in $\sE_R=R\Modl^\cot$.
 This holds because the class $\sE_R$ is closed under extensions
and cokernels of monomorphisms in $R\Modl$, and all injective
$R$\+modules belong to~$\sE_R$
(see Lemmas~\ref{cotorsion-pair-restricts}(b)
and~\ref{restricted-cotorsion-hereditary}).

 Now the class $\sE$ is very colocal by
Example~\ref{cotorsion-colocal-example}, and the class $\sB$
is very colocal by Example~\ref{injectivity-colocal-example}.
 The class $\sE_R$ is also (obviously) closed under extensions and
direct summands in $R\Modl^\cta$.
 Thus the class $\sA$ of cotorsion modules is antilocal
by Theorem~\ref{antilocality-colocality-theorem}(b).
 We have obtained an alternative proof of the antilocality
of cotorsion, based on Theorem~\ref{antilocality-colocality-theorem}(b)
instead of Theorem~\ref{locality-antilocality-theorem}(b)
(cf.\ Example~\ref{cotorsion-antilocal-example}).
\end{ex}

\begin{ex} \label{contraadjusted-flatness-antilocal-via-colocality}
 Let $\sE_R=\sK_R^\cta=R\Modl^\cta$ be the exact category of
contraadjusted $R$\+modules.
  Then the flat cotorsion pair ($R\Modl_\fl$, $R\Modl^\cot$) in
$R\Modl$ restricts  to the full subcategory $\sE_R$, that is,
the pair of classes of flat contraadjusted $R$\+modules
$\sA_R=R\Modl_\fl^\cta$ and cotorsion $R$\+modules $\sB_R=R\Modl^\cot$
is a hereditary complete cotorsion pair $(\sA_R,\sB_R)$ in
$\sE_R=R\Modl^\cta$.
 This holds because the class $\sE_R$ is closed under extensions
and cokernels of monomorphisms in $R\Modl$, and all cotorsion
$R$\+modules belong to~$\sE_R$
(see Lemmas~\ref{cotorsion-pair-restricts}(b)
and~\ref{restricted-cotorsion-hereditary}).

 Now the class $\sE$ is very colocal by
Example~\ref{contraadjustedness-colocal-example}, and the class $\sB$
is very colocal by Example~\ref{cotorsion-colocal-example}.
 Thus the class $\sA$ of flat contraadjusted modules is antilocal
by Theorem~\ref{antilocality-colocality-theorem}(b).
 We have obtained an alternative proof of the antilocality of
contraadjusted flatness, based on
Theorem~\ref{antilocality-colocality-theorem}(b)
instead of Theorem~\ref{locality-antilocality-theorem}(b)
(cf.\ Example~\ref{contraadjusted-flatness-antilocal-example}).
\end{ex}

 Examples~\ref{contraadjustedness-antilocal-via-colocality}
and~\ref{cotorsion-antilocal-via-colocality} admit a common
generalization, which is also applicable to complexes of modules.
 Recall that the category of complexes $\bC(R\Modl)$ has enough
injective objects; the injective objects of $\bC(R\Modl)$ are
the contractible complexes of injective
modules~\cite[Theorem~IV.3.2]{EM}, \cite[Lemma~1.3.3]{Bec},
\cite[Lemma~5.2(a)]{PS4}.
 Notice that, both for $\sK_R=R\Modl$ and for $\sK_R=\bC(R\Modl)$,
all the injective objects of $\sK_R$ belong to $\sK_R^\cta$, and
there are enough of them in the exact category $\sK_R^\cta$; so
the classes of injective objects in $\sK_R$ and $\sK_R^\cta$ coincide.
 Denote this class of injective modules or complexes by
$\sK_R^\inj\subset\sK_R^\cta\subset\sK_R$.

\begin{cor} \label{many-colocal-are-antilocal}
 Let\/ $(\sE_R\subset\sK_R^\cta)_{R\in\R}$ be a system of classes of
contraadjusted modules or complexes of contraadjusted modules such
that, for every $R\in\R$, the class\/ $\sE_R$ is closed under
extensions, direct summands, and cokernels of admissible monomorphisms
in\/~$\sK_R^\cta$.
 Assume further that all the injective objects of\/ $\sK_R$ belong
to\/~$\sE_R$, and that the class\/ $\sE$ is very colocal.
 Then the class\/ $\sE$ is also antilocal.
\end{cor}

\begin{proof}
 The pair of classes $\sA_R=\sE_R$ and $\sB_R=\sK_R^\inj$ is
a hereditary complete cotorsion pair in~$\sE_R$.
 Indeed, the pair of classes (all objects, injective objects) is
a hereditary complete cotorsion pair in any exact category with
enough injective objects.
 The exact category $\sE_R$ has enough injective objects in our
assumptions, and the class of injective objects in $\sE_R$ coincides
with $\sK_R^\inj$, since $\sE_R$ a full subcategory closed under
cokernels of admissible monomorphisms in~$\sK_R^\cta$ (or which is
the same, closed under cokernels of monomorphisms in~$\sK_R$).

 Furthermore, the system of classes of injective modules or
complexes $(\sK_R^\inj)_{R\in\R}$ is very colocal.
 For modules, this holds by Example~\ref{injectivity-colocal-example}.
 For complexes of modules, the similar arguments are applicable.
 Thus we can conclude that the class $\sA=\sE$ is antilocal
by Theorem~\ref{antilocality-colocality-theorem}(b).
\end{proof}

\begin{ex} \label{termwise-cta-antilocal-via-colocality}
 Let $\sE_R=\sK_R=\bC(R\Modl^\cta)$ be the exact category of complexes
of contraadjusted $R$\+modules.
 Then all the assumptions or Corollary~\ref{many-colocal-are-antilocal}
are satisfied; in particular, the class $\sE$ is very colocal
essentially by Example~\ref{contraadjustedness-colocal-example}
(cf.\ Remark~\ref{diagrams-termwise-in-class-remark}).
 Thus the class $\sE$ of all complexes of contraadjusted modules is
antilocal.
 We have obtained an alternative proof of the antilocality of
termwise contraadjustedness of complexes, based on
Theorem~\ref{antilocality-colocality-theorem}(b) instead of
Theorem~\ref{locality-antilocality-theorem}(b)
(cf.\ Example~\ref{termwise-contraadjusted-antilocal-example}).
\end{ex}

\begin{ex} \label{termwise-cotorsion-antilocal-via-colocality}
 Let $\sK_R=\bC(R\Modl^\cta)$ be the exact category of complexes of
contraadjusted $R$\+modules and $\sE_R\subset\sK_R^\cta$ be the full
subcategory of complexes of cotorsion $R$\+modules, $\sE_R=
\bC(R\Modl^\cot)$.
 Then all the assumptions or Corollary~\ref{many-colocal-are-antilocal}
are satisfied; in particular, the class $\sE$ is very colocal
essentially by Example~\ref{cotorsion-colocal-example}
(cf.\ Remark~\ref{diagrams-termwise-in-class-remark}).
 Thus the class $\sE$ of all complexes of cotorsion modules is
antilocal.
 We have obtained an alternative proof of the antilocality of termwise
cotorsion, based on Theorem~\ref{antilocality-colocality-theorem}(b)
instead of Theorem~\ref{locality-antilocality-theorem}(b)
(cf.\ Example~\ref{termwise-cotorsion-antilocal-example}).
\end{ex}

\begin{rem}
 We have seen quite a few examples of classes of modules or complexes
that are \emph{both colocal and antilocal}, including the classes of
injective modules, cotorsion modules, contraadjusted modules, and
complexes of these.
 The class of homotopy injective complexes of injective modules is
also both colocal and antilocal (see
Examples~\ref{homotopy-injectivity-of-injectives-colocal-example}
and~\ref{homotopy-injectivity-antilocal-example}).
 Corollary~\ref{many-colocal-are-antilocal} provides many such classes.

 On the other hand, a class that is \emph{both local and antilocal}
should be a rare occurrence.
 In fact, one can see that any antilocal class of modules provided by
Theorem~\ref{locality-antilocality-theorem}(b) or~(c), or by
Theorem~\ref{antilocality-colocality-theorem}(b) or~(c), consists of
contraadjusted modules only, and any antilocal class of complexes
provided by one of the same theorems consists of complexes with
contraadjusted terms only.
 But contraadjustedness is not preserved by localizations
(see Example~\ref{cotorsion-not-local-example}).

 The class of injective modules over Noetherian rings is such a rare
example of a local antilocal class (see
Examples~\ref{injective-local-and-non-local-example}
and~\ref{injectivity-strongly-antilocal-example}).
 Another example is the class of all complexes of injective modules
over Noetherian rings, which is local essentially by
Example~\ref{injective-local-and-non-local-example}
(cf.\ Remark~\ref{diagrams-termwise-in-class-remark}) and antilocal by
Example~\ref{termwise-injectivity-antilocal-example}.
\end{rem}

 The following examples of antilocal classes we have not seen yet
in the previous section.

\begin{ex} \label{acyclic-of-contraadjusted-antilocal-example}
 Let $\sE_R=\sK_R^\cta=\bC(R\Modl^\cta)$ be the exact category of
complexes of contraadjusted $R$\+modules.
 Then the cotorsion pair formed by the classes of acyclic complexes
of modules and homotopy injective complexes of injective modules
($\bC_\ac(R\Modl)$, $\bC^\hin(R\Modl^\inj)$) in $\bC(R\Modl)$
(see Example~\ref{homotopy-injectivity-antilocal-example})
restricts to the full subcategory~$\sE_R$.
 In other words, the pair of classes of acyclic complexes of
contraadjusted $R$\+modules $\sA_R=\bC_\ac(R\Modl)\cap\bC(R\Modl^\cta)=
\bC_\ac(R\Modl^\cta)$ and homotopy injective complexes of injective
$R$\+modules $\sB_R=\bC^\hin(R\Modl^\inj)$ is a hereditary complete
cotorsion pair in $\sE_R=\bC(R\Modl^\cta)$.
 This holds because the class $\sE_R$ is closed under extensions and
cokernels of monomorphisms in $\bC(R\Modl)$, and all homotopy injective
complexes of injective $R$\+modules belong to~$\sE_R$
(see Lemmas~\ref{cotorsion-pair-restricts}(b)
and~\ref{restricted-cotorsion-hereditary}).

 Notice that any acyclic complex of contraadjusted modules has
contraadjusted modules of cocycles, because the class of contraadjusted
modules $R\Modl^\cta$ is closed under epimorphic images in $R\Modl$
(see Lemma~\ref{generated-by-projdim1-lemma}).
 Thus a complex of contraadjusted modules is acyclic in $R\Modl$ if
and only if it is acyclic in $R\Modl^\cta$, and the expression
``acyclic complex of contraadjusted modules'' is unambiguous.

 We observe that the class $\sE$ is very colocal essentially by
Example~\ref{contraadjustedness-colocal-example},
and the class $\sB$ is very colocal by
Example~\ref{homotopy-injectivity-of-injectives-colocal-example}.
 Thus the class $\sA$ of acyclic complexes of contraadjusted modules
is antilocal by Theorem~\ref{antilocality-colocality-theorem}(b).
 Alternatively, one could show that the class of acyclic complexes
of contraadjusted modules is very colocal (using
Lemma~\ref{coascent+direct-image-imply-codescent}), and apply
Corollary~\ref{many-colocal-are-antilocal} in order to conclude that
it is also antilocal.
\end{ex}

\begin{ex}
 Let $\sK_R^\cta=\bC(R\Modl^\cta)$ be the exact category of complexes
of contraadjusted $R$\+modules and $\sE_R\subset\sK_R^\cta$ be
the full subcategory of complexes of cotorsion $R$\+modules,
$\sE_R=\bC(R\Modl^\cot)$.
 Then the cotorsion pair formed by the classes of acyclic complexes
of modules and homotopy injective complexes of injective modules
($\bC_\ac(R\Modl)$, $\bC^\hin(R\Modl^\inj)$) in $\bC(R\Modl)$
restricts to the full subcategory~$\sE_R$.
 In other words, the pair of classes of acyclic complexes of
cotorsion $R$\+modules $\sA_R=\bC_\ac(R\Modl)\cap\bC(R\Modl^\cot)=
\bC_\ac(R\Modl^\cot)$ and homotopy injective complexes of injective
$R$\+modules $\sB_R=\bC^\hin(R\Modl^\inj)$ is a hereditary complete
cotorsion pair in $\sE_R=\bC(R\Modl^\cot)$.
 Similarly to the previous
Example~\ref{acyclic-of-contraadjusted-antilocal-example}, this holds
by Lemmas~\ref{cotorsion-pair-restricts}(b)
and~\ref{restricted-cotorsion-hereditary}.

 Notice that any acyclic complex of cotorsion modules has cotorsion
modules of cocycles by~\cite[Theorem~5.1(2)]{BCE}.
 Thus a complex of cotorsion modules is acyclic in $R\Modl$ if
and only if it is acyclic in $R\Modl^\cot$, and the expression
``acyclic complex of cotorsion modules'' is unambiguous.

 We observe that the class $\sE$ is very colocal essentially by
Example~\ref{cotorsion-colocal-example},
and the class $\sB$ is very colocal by
Example~\ref{homotopy-injectivity-of-injectives-colocal-example}.
 Thus the class $\sA$ of acyclic complexes of cotorsion modules
is antilocal by Theorem~\ref{antilocality-colocality-theorem}(b).
 Alternatively, one could show that the class of acyclic complexes
of cotorsion modules is very colocal (see
Example~\ref{homotopy-flat-of-flat-cta-antilocal-example} below),
and apply Corollary~\ref{many-colocal-are-antilocal} in order to
conclude that it is also antilocal.

 As another alternative, antilocality of the class of acyclic complexes
of cotorsion modules is provable using
Proposition~\ref{homotopy-flat-acyclic-cot-cotorsion-pair} below
and Theorem~\ref{locality-antilocality-theorem}(b).
\end{ex}

\begin{ex} \label{coacyclic-of-contraadjusted-antilocal-example}
 Let $\sE_R=\sK_R^\cta=\bC(R\Modl^\cta)$ be the exact category of
complexes of contraadjusted $R$\+modules.
 Then the cotorsion pair formed by the classes of Becker-coacyclic
complexes of modules and all complexes of injective modules
($\bC_\coac(R\Modl)$, $\bC(R\Modl^\inj)$) in $\bC(R\Modl)$
(see Example~\ref{termwise-injectivity-antilocal-example})
restricts to the full subcategory~$\sE_R$.
 In other words, the pair of classes of Becker-coacyclic complexes of
contraadjusted $R$\+modules $\sA_R=\bC_\coac(R\Modl)\cap\bC(R\Modl^\cta)
=\bC_\coac(R\Modl^\cta)$ and all complexes of injective
$R$\+modules $\sB_R=\bC(R\Modl^\inj)$ is a hereditary complete
cotorsion pair in $\sE_R=\bC(R\Modl^\cta)$.
 As usual, this holds because the class $\sE_R$ is closed under
extensions and cokernels of monomorphisms in $\bC(R\Modl)$, and all
complexes of injective $R$\+modules belong to~$\sE_R$
(see Lemmas~\ref{cotorsion-pair-restricts}(b)
and~\ref{restricted-cotorsion-hereditary}).

 Notice that the injective objects of the exact category $R\Modl^\cta$
of contraadjusted $R$\+modules coincide with the injective objects
of the abelian category $R\Modl$.
 Hence the complexes of injective objects in $R\Modl^\cta$ are
the same things as the complexes of injective objects in $R\Modl$.
 In this sense, one can say that a complex of contraadjusted
$R$\+modules is Becker-coacyclic as a complex in $R\Modl$ if and only
if it is Becker-coacyclic as a complex in $R\Modl^\cta$, and
the expression ``Becker-coacyclic complex of contraadjusted modules''
is unambiguous.

 Now both the classes $\sE$ and $\sB$ are very colocal essentially by
Examples~\ref{contraadjustedness-colocal-example}
and~\ref{injectivity-colocal-example}
(cf.\ Remark~\ref{diagrams-termwise-in-class-remark}).
 Thus the class $\sA$ of Becker-coacyclic complexes of
contraadjusted modules is antilocal by
Theorem~\ref{antilocality-colocality-theorem}(b).
\end{ex}

\begin{ex}
 Let $\sK_R^\cta=\bC(R\Modl^\cta)$ be the exact category of complexes
of contraadjusted $R$\+modules and $\sE_R\subset\sK_R^\cta$ be the full
subcategory of complexes of cotorsion $R$\+modules,
$\sE_R=\bC(R\Modl^\cot)$.
 Then, similarly to
Example~\ref{coacyclic-of-contraadjusted-antilocal-example},
the cotorsion pair formed by the classes of Becker-coacyclic complexes
of modules and all complexes of injective modules
($\bC_\coac(R\Modl)$, $\bC(R\Modl^\inj)$) in $\bC(R\Modl)$
restricts to the full subcategory~$\sE_R$.
 In other words, the pair of classes of Becker-coacyclic complexes of
cotorsion $R$\+modules $\sA_R=\bC_\coac(R\Modl)\cap\bC(R\Modl^\cot)
=\bC_\coac(R\Modl^\cot)$ and all complexes of injective
$R$\+modules $\sB_R=\bC(R\Modl^\inj)$ is a hereditary complete
cotorsion pair in $\sE_R=\bC(R\Modl^\cot)$.

 Similarly to
Example~\ref{coacyclic-of-contraadjusted-antilocal-example},
one observes that injective objects of the exact category $R\Modl^\cot$
of cotorsion $R$\+modules coincide with the injective objects of
the abelian category $R\Modl$.
 Therefore, one can say that a complex of cotorsion $R$\+modules is
Becker-coacyclic as a complex in $R\Modl$ if and only if it is
Becker-coacyclic as a complex in $R\Modl^\cot$, and the expression
``Becker-coacyclic complex of cotorsion modules'' is unambiguous.

 Both the classes $\sE$ and $\sB$ are very colocal essentially by
Examples~\ref{cotorsion-colocal-example}
and~\ref{injectivity-colocal-example}.
 The class $\sE_R$ is also (obviously) closed under extensions and
direct summands in $R\Modl^\cta$.
 Thus the class $\sA$ of Becker-coacyclic complexes of cotorsion modules
is antilocal by Theorem~\ref{antilocality-colocality-theorem}(b).
\end{ex}

\begin{ex} \label{flat-acyclic-complexes-of-cta-antilocal-example}
 Let $\sE_R=\sK_R^\cta=\bC(R\Modl^\cta)$ be the exact category of
complexes of contraadjusted $R$\+modules.
 Then the cotorsion pair formed by the classes of acyclic complexes of
flat modules with flat modules of cocycles and all complexes of
cotorsion modules ($\bC_\ac^\bfl(R\Modl_\fl)$, $\bC(R\Modl^\cot)$)
in $\bC(R\Modl)$ (see
Proposition~\ref{very-flat-tilde-termwise-cta-cot-cotorsion-pair}(b))
restricts to the full subcategory~$\sE_R$.
 In other words, the pair of classes of acyclic complexes of flat
contraadjusted $R$\+modules with flat modules of cocycles
$\sA_R=\bC_\ac^\bfl(R\Modl_\fl)\cap\bC(R\Modl^\cta)
=\bC_\ac^\bfl(R\Modl_\fl^\cta)$ and all complexes of cotorsion
$R$\+modules $\sB_R=\bC(R\Modl^\cot)$ is a hereditary complete
cotorsion pair in $\sE_R=\bC(R\Modl^\cta)$.
 Similarly to the previous
Example~\ref{coacyclic-of-contraadjusted-antilocal-example}, this
holds by Lemmas~\ref{cotorsion-pair-restricts}(b)
and~\ref{restricted-cotorsion-hereditary}.

 Now both the classes $\sE$ and $\sB$ are very colocal essentially by
Examples~\ref{contraadjustedness-colocal-example}
and~\ref{cotorsion-colocal-example}
(cf.\ Remark~\ref{diagrams-termwise-in-class-remark}).
 Thus the class $\sA$ of acyclic complexes of flat contraadjusted
modules with flat modules of cocycles is antilocal by
Theorem~\ref{antilocality-colocality-theorem}(b).
\end{ex}

 A complex of $R$\+modules $P^\bu$ is said to be \emph{homotopy
projective} (or ``$K$\+pro\-jec\-tive'')~\cite{Spal} if, for
any acyclic complex of $R$\+modules $Y^\bu$, any morphism of complexes
of $R$\+modules $P^\bu\rarrow Y^\bu$ is homotopic to zero.
 Any complex of $R$\+modules is quasi-isomorphic to a homotopy
projective complex (and even to a homotopy projective complex of
projective $R$\+modules), which is defined uniquely up to
homotopy equivalence.

 A complex of left $R$\+modules $F^\bu$ is said to be \emph{homotopy
flat} (or ``$K$\+flat'')~\cite{Spal} if, for any acyclic complex of
right $R$\+modules $X^\bu$, the complex $X^\bu\ot_R F^\bu$ is acyclic.
 Any homotopy projective complex is homotopy flat.

\begin{prop} \label{homotopy-flat-acyclic-cot-cotorsion-pair}
 For any associative ring $R$, consider the following pair of classes
of complexes of left $R$\+modules.
 The left class\/ $\sA_R=\bC_\hfl(R\Modl_\fl)$ is the class of homotopy
flat complexes of flat $R$\+modules.
 The right class\/ $\sB_R=\bC_\ac(R\Modl^\cot)=\bC_\ac(R\Modl)\cap
\bC(R\Modl^\cot)$ is the class of acyclic complexes of cotorsion
$R$\+modules.
 Then the pair of classes $(\sA_R,\sB_R)$ is a hereditary complete
cotorsion pair in the abelian category of complexes $\bC(R\Modl)$.
\end{prop}

\begin{proof}
 This is~\cite[Chapter~4]{GR}, \cite[Proposition~4.11 and
Corollary~4.18]{Gil0}, and~\cite[Theorem~5.1(2)]{BCE}
(see also~\cite[Proposition~4.1]{Em}).
 Let us spell out some details, for the reader's benefit.

 First of all, an acyclic complex of flat modules is homotopy flat if
and only if it has flat modules of cocycles.
 Indeed, if $F^\bu$ is an acyclic complex of flat left modules with
flat modules of cocycles and $X^\bu$ is an arbitrary complex of right
modules, then the complex $X^\bu\ot_R F^\bu$ is acyclic, as one can
see by representing $X^\bu$ as a direct limit of bounded complexes of
modules (using the silly truncation on one side and the canonical
truncation on the other side) and reducing the question to the case
when $X^\bu$ is a one-term complex.
 Conversely, given an acyclic complex of left flat modules $F^\bu$
and a right module $X$, choose a projective resolution $P^\bu$ for
$X$ and notice that $P^\bu\ot_R F^\bu$ is an acyclic complex (since
$P^\bu$ is a bounded above complex of flat modules and $F^\bu$ is
an acyclic complex).
 Now if the complex $(P^\bu\to X)\ot_R F^\bu$ is acyclic, then
the complex $X\ot_RF^\bu$ is acyclic, and it follows that the modules
of cocycles of the acyclic complex of flat modules $F^\bu$ are flat.

 Next, let us show that the full triangulated subcategory of homotopy
flat complexes of flat modules in the homotopy category of complexes
$\bH(R\Modl)$ is generated by the homotopy projective complexes of
projective modules and acyclic complexes of flat modules with flat
modules of cocycles.
 Indeed, let $F^\bu$ be a homotopy flat complex of flat modules.
 Choose a homotopy projective complex of projective modules $P^\bu$
together with a quasi-isomorphism $P^\bu\rarrow F^\bu$, and denote
by $G^\bu$ the cone of this morphism.
 Then $G^\bu$ is an acyclic homotopy flat complex of flat modules,
hence its modules of cocycles are flat.

 Let~$\lambda$ be an infinite cardinal greater or equal to
the cardinality of the ring $R$, and let $\sS_0$ be the set of
(representatives of isomorphism classes of all) flat
$R$\+modules of the cardinality~$\le\lambda$.
 Then all flat $R$\+modules are filtered by the modules
from~$\sS_0$\, (see~\cite[Lemma~1 and Proposition~2]{BBE}
or~\cite[Lemma~6.23]{GT}).
 Denote by $\sS$ the class of all one-term complexes obtained by placing
the modules from $\sS_0$ in various cohomological degrees $n\in\boZ$.
 We claim that $(\sA_R,\sB_R)$ is the cotorsion pair in $\bC(R\Modl)$
generated by~$\sS$.

 Indeed, Lemma~\ref{disk-complexes-lemma} together with
Lemma~\ref{eklof-lemma} imply that all complexes from $\sS^{\perp_1}$
are complexes of cotorsion modules.
 Futhermore, all complexes from $\sS^{\perp_1}$ are acyclic, since
the free $R$\+module $R$ belongs to~$\sS$.
 It is helpful to keep Lemma~\ref{ext-homotopy-hom-lemma} in mind.
 On the other hand, in any complex of cotorsion modules, the modules
of cocycles are cotorsion~\cite[Theorem~5.1(2)]{BCE}, hence
any morphism from a complex belonging to $\sS$ to a complex belonging
to $\sB_R$ is homotopic to zero.
 Thus $\sS^{\perp_1}=\sB_R$.
 
 It remains to show that $\sA_R={}^{\perp_1}\sB_R$.
 For any homotopy projective complex $P^\bu$ and any complex $C^\bu
\in\sB_R$, any morphism of complexes $P^\bu\rarrow C^\bu$ is homotopic
to zero, since $C^\bu$ is an acyclic complex.
 For any acyclic complex of flat modules $F^\bu$ with flat modules
of cocycles, and any complex $C^\bu\in\sB_R$, any morphism of
complexes $F^\bu\rarrow C^\bu$ is homotopic to zero, since $C^\bu$
is a complex of cotorsion modules~\cite[Theorem~5.3]{BCE}.
 Hence $\sA_R\subset{}^{\perp_1}\sB_R$.
 On the other hand, ${}^{\perp_1}\sB_R$ is the class of all complexes
filtered by complexes from $\sS$ (essentially by
Theorem~\ref{eklof-trlifaj-theorem}), and all such complexes belong
to $\sA_R$ because the class of homotopy flat complexes of flat
modules is closed under extensions and direct limits.

 Now Theorem~\ref{eklof-trlifaj-theorem} tells that $(\sA_R,\sB_R)$
is a complete cotorsion pair in $\bC(R\Modl)$.
 As the class $\sB_R$ is clearly closed under the cokernels of
monomorphisms in $\bC(R\Modl)$, it follows that this cotorsion pair is
hereditary.
\end{proof}

\begin{ex} \label{homotopy-flat-of-flat-cta-antilocal-example}
 Let $\sE_R=\sK_R^\cta=\bC(R\Modl^\cta)$ be the exact category of
complexes of contraadjusted $R$\+modules.
 Then the cotorsion pair formed by the classes of homotopy flat
complexes of flat modules and acyclic complexes of cotorsion modules
($\bC_\hfl(R\Modl_\fl)$, $\bC_\ac(R\Modl^\cot)$) in $\bC(R\Modl)$ (see
Proposition~\ref{homotopy-flat-acyclic-cot-cotorsion-pair})
restricts to the full subcategory~$\sE_R$.
 In other words, the pair of classes of homotopy flat complexes of
flat contraadjusted $R$\+modules $\sA_R=\bC_\hfl(R\Modl_\fl)\cap
\bC(R\Modl^\cta)=\bC_\hfl(R\Modl_\fl^\cta)$ and acyclic complexes of
cotorsion $R$\+modules $\sB_R=\bC_\ac(R\Modl^\cot)$ is a hereditary
complete cotorsion pair in $\sE_R=\bC(R\Modl^\cta)$.
 Similarly to the previous examples, this holds by
Lemmas~\ref{cotorsion-pair-restricts}(b)
and~\ref{restricted-cotorsion-hereditary}.

 The class $\sE$ is very colocal by
Example~\ref{contraadjustedness-colocal-example}.
 Furthermore, the class $\sB$ of acyclic complexes of cotorsion
modules is very colocal as well.
 It is important here that all acyclic complexes of cotorsion modules
have cotorsion modules of cocycles~\cite[Theorem~5.1(2)]{BCE}.
 Hence one can see that, more generally, the complex
$\Hom_R(S,C^\bu)$ is an acyclic complex of cotorsion $S$\+modules for
any acyclic complex of cotorsion $R$\+modules $C^\bu$ and any
commutative ring homomorphism $R\rarrow S$ making $S$ a flat
$R$\+module.
 The underlying complex of $R$\+modules of any acyclic complex of
cotorsion $S$\+modules is an acyclic complex of cotorsion
$R$\+modules for any ring homomorphism $R\rarrow S$ (cf.\
Example~\ref{cotorsion-colocal-example}).
 Thus the coascent and direct image conditions hold.
 The class $\sB_R$ is clearly closed under cokernels of monomorphisms
in $\bC(R\Modl)$.
 So Lemma~\ref{coascent+direct-image-imply-codescent} tells that
the class $\sB$ satisfies codescent.

 Therefore, we can conclude that the class $\sA$ of homotopy flat
complexes of flat contraadjusted modules is antilocal
by Theorem~\ref{antilocality-colocality-theorem}(b).

 Similarly, one can restrict the same cotorsion pair in $\bC(R\Modl)$
to the exact subcategory of complexes of cotorsion $R$\+modules
$\sE_R=\bC(R\Modl^\cot)\subset\bC(R\Modl^\cta)=\sK_R^\cta$.
 Then Theorem~\ref{antilocality-colocality-theorem}(b) tells that
the class of homotopy flat complexes of flat cotorsion modules is
antilocal.
\end{ex}

 At last, let us suggest an alternative approach to
Example~\ref{termwise-very-flat-contraadjusted-antilocal-example},
and use this also as an occasion to discuss the colocality of
Becker-contraacyclicity.

\begin{ex} \label{termwise-very-flat-cta-antilocal-via-colocality}
 Let $\sE_R=\sK_R^\cta=\bC(R\Modl^\cta)$ be the exact category of
complexes of contraadjusted $R$\+modules.
 Then the cotorsion pair formed by the classes of all complexes of
very flat modules and Becker-contraacyclic complexes of contraadjusted
modules ($\bC(R\Modl_\vfl)$, $\bC^\ctrac(R\Modl^\cta)$) in $\bC(R\Modl)$
(see Proposition~\ref{very-flat-contraacyclic-cotorsion-pair}(a))
restricts to the full subcategory~$\sE_R$.
 In other words, the pair of classes of all complexes of very flat
contraadjusted $R$\+modules $\sA_R=\bC(R\Modl_\vfl)\cap
\bC(R\Modl^\cta)=\bC(R\Modl_\vfl^\cta)$ and Becker-contraacyclic
complexes of contraadjusted $R$\+modules $\sB_R=\bC^\ctrac(R\Modl^\cta)$
is a hereditary complete cotorsion pair in $\sE_R=\bC(R\Modl^\cta)$.
 As usual, this holds by Lemmas~\ref{cotorsion-pair-restricts}(b)
and~\ref{restricted-cotorsion-hereditary}.

 The class $\sE$ is very colocal by
Example~\ref{contraadjustedness-colocal-example}.
 Furthermore, the class $\sB$ of Becker-contraacyclic complexes of
contraadjusted modules is also very colocal.
 Indeed, the class $\sB$ satisfies the coascent and direct image
conditions because the class of all complexes of very flat modules
$\bC(R\Modl_\vfl)$ satisfies the ascent and direct image conditions
(see Example~\ref{very-flatness-local-example} and
Proposition~\ref{left-right-ascent-coascent-direct-image-prop}).
 Since the class $\sB$ is also closed under cokernels of
monomorphisms, the codescent holds for it
by Lemma~\ref{coascent+direct-image-imply-codescent}.
 This is the dual argument to the one in
Example~\ref{termwise-injectivity-antilocal-example}.

 Thus the class $\sA$ of all complexes of very flat contraadjusted
$R$\+modules is antilocal by
Theorem~\ref{antilocality-colocality-theorem}(b).
 We have obtained an alternative proof of the antilocality of
termwise contraadjusted very flatness, based on
Theorem~\ref{antilocality-colocality-theorem}(b) instead of
Theorem~\ref{locality-antilocality-theorem}(b) (cf.\
Example~\ref{termwise-very-flat-contraadjusted-antilocal-example}).

 The antilocality of termwise contraadjusted flatness
(as in Example~\ref{termwise-flat-contraadjusted-antilocal-example})
and of termwise flat cotorsion
(Example~\ref{termwise-flat-cotorsion-antilocal-example}) can be proved
similarly, using Theorem~\ref{antilocality-colocality-theorem}(b) with
Proposition~\ref{very-flat-contraacyclic-cotorsion-pair}(b).
 One takes $\sE_R=\sK_R^\cta=\bC(R\Modl^\cta)$ in the former case and
$\sE_R=\bC(R\Modl^\cot)$ in the latter one.
\end{ex}

\end{document}